\documentclass[a4paper, 12pt]{amsart}

\makeatletter
\def\@tocline#1#2#3#4#5#6#7{\relax
  \ifnum #1>\c@tocdepth 
  \else
    \par \addpenalty\@secpenalty\addvspace{#2}
    \begingroup \hyphenpenalty\@M
    \@ifempty{#4}{
      \@tempdima\csname r@tocindent\number#1\endcsname\relax
    }{
      \@tempdima#4\relax
    }
    \parindent\z@ \leftskip#3\relax \advance\leftskip\@tempdima\relax
    \rightskip\@pnumwidth plus4em \parfillskip-\@pnumwidth
    #5\leavevmode\hskip-\@tempdima
      \ifcase #1
       \or\or \hskip 1em \or \hskip 2em \else \hskip 3em \fi
      #6\nobreak\relax
    \hfill\hbox to\@pnumwidth{\@tocpagenum{#7}}\par
    \nobreak
    \endgroup
  \fi}
\makeatother

\usepackage{amsmath,amssymb,amsthm,url}
\usepackage[utf8]{inputenc}
\usepackage[T1]{fontenc}
\usepackage{quiver}
\usepackage{libertine}
\usepackage[libertine,cmintegrals,cmbraces]{newtxmath}
\usepackage[shortlabels]{enumitem}

\usepackage{mathtools}
\usepackage{adjustbox}

\usepackage{xcolor}
\usepackage{tikz}
\usepackage{tikz-cd}
\usetikzlibrary{calc,intersections}
\usepackage{nicefrac}
\usepackage{dsfont}
\usepackage{todonotes}
\usepackage{a4wide}

\makeatletter
\newcommand{\mynameis}[1]{#1\renewcommand{\@currentlabel}{#1}}
\makeatother

\AtBeginDocument{
   \def\MR#1{}
} 
\usepackage{euscript}
\usepackage[all]{xy}

\usepackage[pagebackref]{hyperref}
  
\hypersetup{
  bookmarksnumbered=true,
  colorlinks=true,
  linkcolor=seagreen,
  citecolor=seagreen,
  filecolor=seagreen,
  menucolor=seagreen,
  urlcolor=seagreen,
  pdfnewwindow=true,
  pdfstartview=FitBH}

\allowdisplaybreaks

\usepackage{xcolor}
\usepackage{amsmath}
\definecolor{seagreen}{RGB}{46,139,87}
\definecolor{maroon}{RGB}{128,0,0}
\definecolor{darkviolet}{RGB}{148,0,211}
\definecolor{twelve}{RGB}{100,100,170}
\definecolor{thirteen}{RGB}{100,150,50}
\definecolor{fourteen}{RGB}{200,0,0}
\definecolor{fifteen}{RGB}{0,200,0}
\definecolor{sixteen}{RGB}{0,0,200}
\definecolor{seventeen}{RGB}{200,0,200}
\definecolor{eighteen}{RGB}{0,200,200}

\DeclareMathOperator{\colim}{\mathrm{colim}}

\newcommand{\del}{\partial}
\newcommand{\Z}{\mathbb{Z}}

\newcommand{\FM}{\text{FM}}

  \newcommand{\adjunction}[4]{
\xymatrix{
#1:#2 \ar@<.5ex>[r] &
\ar@<.5ex>[l] #3:#4
}}

\newtheorem{thm}{Theorem}[subsection]
\newtheorem{prop}[thm]{Proposition}
\newtheorem{example}[thm]{Example}

\newtheorem{warning}[thm]{Warning}

\newtheorem{lem}[thm]{Lemma}
\newtheorem{cor}[thm]{Corollary}

\newtheorem*{thm*}{Theorem}

\theoremstyle{definition}

\newtheorem{definition}[thm]{Definition}

\newtheorem{remark}[thm]{Remark}
\newtheorem{ex}[thm]{Example}

\title{Surgery on manifold operads}
\author{Xujia Chen}
\author{Connor Malin}
\author{Paolo Salvatore}

\begin{document}

\begin{abstract} 
We study cobordisms of a class of topological operads called ``manifold operads''. These operads are generalizations of the Fulton-MacPherson operad: an operad built from configurations of points in Euclidean space.  Cobordism of manifold operads, along with the associated theory of surgery, depends crucially on delicate combinatorial results for trees associated to operadic bimodules.  As an application of surgery, we produce infinitely many manifold operads which are left or right ``bimodule cobordant'' to, but not homotopy equivalent to the Fulton-MacPherson operad.

\end{abstract}
\maketitle

\tableofcontents

\section{Introduction and main results}

The $d$-dimensional Fulton-MacPherson operad $\mathrm{FM}_d$ was defined in \cite{getzler_jones} to have underlying spaces the so-called ``infinitesimal configurations'' of points in $\mathbb{R}^d$ modulo the relations given by translation and scaling by a positive real number. The $m$th space $\mathrm{FM}_d(m)$ is a compact manifold with boundary, and it is homotopy equivalent to the ordered configuration space of $m$ points $\mathrm{Conf}(\mathbb{R}^d,m)$. The operad structure is given by insertion of infinitesimal configurations. 

The Fulton-MacPherson operad $\FM_d$ is an example of a topological manifold operad $O$ of dimension $d$ (Definition \ref{manifoldoperad_dfn}): 
a reduced operad $O$ such that
for all $m \geq 2$,
$O(m)$ is a compact manifold
  of dimension $md-d-1$, and 
the boundary $\partial O(m)$ is freely generated by the operations of arity less than $m$.
In particular, there is an equality \[\partial O(m) = \mathrm{Decom}(O)(m)\] between the boundary of $O(m)$ and the space of decomposable elements of arity $m$.

Manifold operads are a setting in which it is possible to make surgery theoretic arguments. To make this precise, we introduce bimodule cobordisms between manifold operads (Definition \ref{bimodulecobordism_dfn}): a bimodule cobordism between two manifold operads $O,P$ of dimension $d$ is a reduced operadic $(O,P)$-bimodule $W$ such that for all $m \geq 2$, $W(m)$ is a compact manifold of dimension $md-d$, and the boundary $\partial W(m)$ is freely
     generated by the actions of the operations of the operads $O$ and $P$ on elements of $W$ of arity less than $m$.

     In particular, there is an equality \[\partial W(m) = O(m) \cup \mathrm{Decom}(W)(m) \cup P(m), \] 
where $\mathrm{Decom}(W)$ denotes the image of the bimodule composites involving elements of $W$ of arity $\geq 2$, while the copies of $O(m)$ and $P(m)$ are those involving $W(1)=\{*\}$.

Unlike cobordism of manifolds which is symmetric, bimodule cobordism of operads is asymmetric due to the asymmetry of left and right modules over operads. Given a $(B,O)$-bimodule cobordism $W$ and an $(O,P)$-bimodule cobordism $V$
\[B \hookrightarrow W \hookleftarrow O \hookrightarrow V \hookleftarrow P\]
we may form the relative composite $W \circ_O V$ which is a $(B,P)$-bimodule. In arity $2$ this relative composite is simply the gluing of cobordisms, however in arity $m$ it is more complicated and involves products of all the lower arity spaces of operations.

\begin{thm}[Theorem \ref{thm: composition is manifold}]
    If $W$ is a $(B,O)$-bimodule cobordism, $V$ an $(O,P)$-bimodule cobordism, and everything in sight admits levelwise $\Sigma_m$-equivariant collars, then the relative composite $W \circ_O V$ is a $(B,P)$-bimodule cobordism. 
\end{thm}

One may define the naïve category of bimodule cobordisms $\mathrm{BimodBord}_{/O}$ on a fixed and possibly $n$-truncated manifold operad $O$, where $n$-truncated means only defined up to arity $n$. The objects are the bimodule cobordisms which have a right action by $O$ and the morphisms are compatible pairs of operad and bimodule maps. In particular, underlying the morphisms in this category is the data of a sequence of cobordisms on the manifolds $O(n)$

\[\begin{tikzcd}
	P(n) & W(n) & O(n) \\
	{P'(n)} & {W'(n)} & O(n)
	\arrow[hook, from=1-1, to=1-2]
	\arrow[from=1-1, to=2-1]
	\arrow[from=1-2, to=2-2]
	\arrow[hook', from=1-3, to=1-2]
	\arrow["{=}", from=1-3, to=2-3]
	\arrow[hook, from=2-1, to=2-2]
	\arrow[hook', from=2-3, to=2-2]
\end{tikzcd}\]

Similarly, one can define the category of bimodule cobordisms $\mathrm{BimodBord}_{O/}$  under a given, possibly $n$-truncated, manifold operad $O$.

By forgetting higher arity information there are truncation functors
\[(-)^{\leq n}:\mathrm{BimodBord}_{/O} \rightarrow \mathrm{BimodBord}_{/O^{\leq n}},\]
\[(-)^{\leq n}:\mathrm{BimodBord}_{O/} \rightarrow \mathrm{BimodBord}_{O^{\leq n}/}.\]
The main construction of the paper is arity extension for bimodule cobordisms:
\begin{thm}[Theorem \ref{thm: left surgery}, Theorem \ref{thm: right surgery}]\label{1_thm}
There are explicit functors
    \[X_L:\mathrm{BimodBord}_{/O^{\leq n}} \rightarrow \mathrm{BimodBord}_{/O},\]
\[X_R:\mathrm{BimodBord}_{O^{\leq n}/} \rightarrow \mathrm{BimodBord}_{O/}\]
providing sections to the $n$-truncation functors $(-)^{\leq n}$.
\end{thm}

As discussed in Section \ref{section: surgery on manifold operads}, this result asserts there are two canonical ways (up to choices of collars) to propagate an equivariant surgery on the interior of $O(n)$ to the higher arities of $O$, and the resulting manifold operad will be bimodule cobordant to $O$ on either the left or right depending on the construction used. 
We expect this construction to be universal among sections of the truncation functors.

As the main application of our theorem, we produce the first examples of manifold operads which are not isomorphic to $\mathrm{FM}_d$.

\begin{thm}[Theorem \ref{thm: sufficient condition for existence}]\label{main_thm}
  If $M$ is a nullbordant closed $\Sigma_2$-manifold, then there is a manifold operad $O$ such that $O(2)=M$.
\end{thm}

Manifold operads have a few excellent properties. Recall that $W(O)$ is a certain cofibrant replacement of $O$ \cite{boardman_vogt_1973} which for manifold operads can be understood as adding an external collar compatible with the stratification by trees, as witnessed by the following generalization of the third author's result for the Fulton-MacPherson operad \cite{salvatore_2021}.

\begin{prop}[Proposition \ref{prop: isomorphic to w construction}]
    If $O$ is a manifold operad and each $O(m)$ admits a $\Sigma_m$-equivariant collar, then there is an operad isomorphism
    \[O \cong W(O).\]

\end{prop}

This topological operad isomorphism has interesting implications for the underlying algebraic operad $C_\ast(O; \mathbb{Z}/2)$ generalizing work of Getzler-Jones \cite{getzler_jones}, as well as work of the second author \cite{malinhha}. In particular, it implies that this algebraic operad is equivalent to a shift of its Koszul dual operad, at least at the level of homology. Let $B(-)$ denote the cooperad given by the bar construction of an algebraic operad, $[-]$ the operation of operadic suspension, and $(-)^\vee$ the linear dual.

\begin{prop}[Proposition \ref{prop: koszul self dual}]
    If $O$ is a manifold operad of dimension $d$ and each $O(m)$ admits a $\Sigma_m$-equivariant collar, then there is an algebraic operad isomorphism \[H_\ast (C_\ast(O; \mathbb{Z}/2)) \cong H_\ast (B(C_\ast (O;\mathbb{Z}/2))^\vee)[d].\] 
\end{prop}

\subsection{Outline of the proof}

We now sketch the idea behind Theorem \ref{main_thm} which constructs a manifold operad $O$ such that $O(2)=M$ given an equivariant cobordism $W(2)$ between $S^{d-1}$ with the antipodal action and some closed $\Sigma_2$-manifold $M$.

Let us assume that $O(2),\ldots,O(n-1)$ have already been constructed. 
The only requirement imposed on the manifold $O(n)$ by the definition of a manifold operad is that its boundary is already determined: it is a space (which we call $Q(n)$) obtained by gluing some products of lower-arity terms of $O$ along their boundary strata in a specified way. 
Now, two problems arise: 
\begin{itemize}
    \item Since $Q(n)$ is supposed to be $\del O(n)$, it must be a closed manifold. Is this actually the case?

    \item Supposing that $Q(n)$ is indeed a closed manifold, is it nullbordant? 
\end{itemize}

Regarding this first problem, one can use a natural stratification of $Q(n)$ to show that a point $p$ in a stratum $Z$ has neighborhoods of the form 
$$\mathbb{R}^{\dim(Z)}\times \mathrm{Cone}(\mathrm{link}(Z))$$ 
for some space $\mathrm{link}(Z)$ that depends only on the combinatorics of the stratification (see below). 
A sufficient condition for the glued-together space to be a closed topological manifold is that $\mathrm{link}(Z)$ is homeomorphic to a sphere for all strata $Z$ (and for a general manifold with boundary one should check the links of points on the boundary are discs):
\begin{center}
    \includegraphics[scale=0.3]{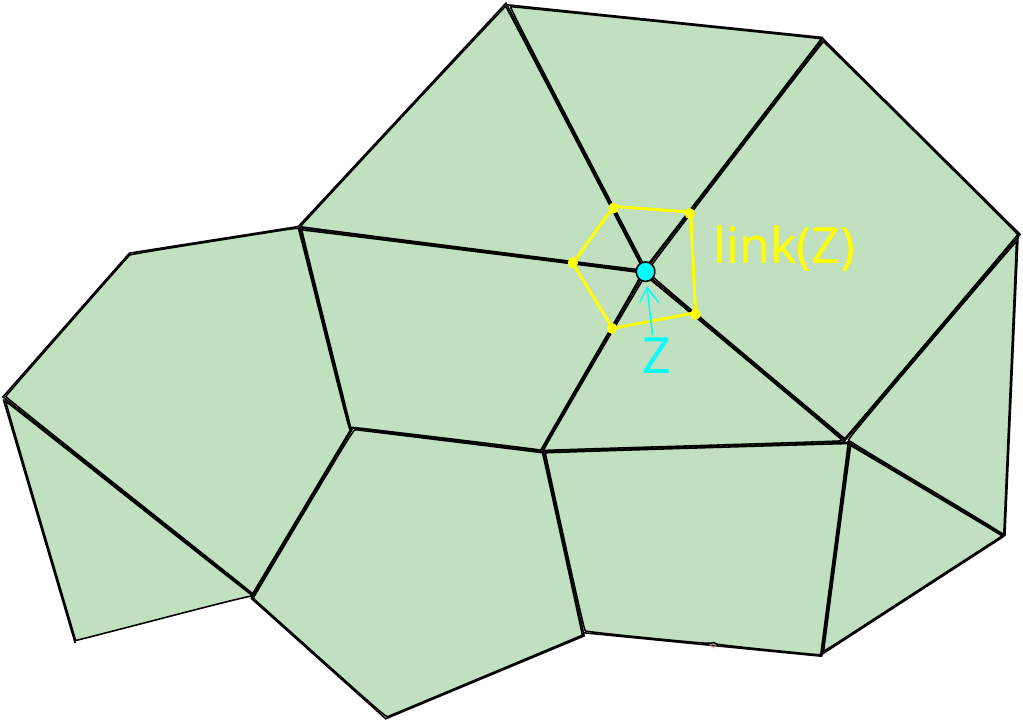} 
\end{center}
The answer to this first question is relatively simple, and some gentle combinatorics will show that $Q(n)$ is indeed a topological manifold. 

The second problem is more difficult, and we construct $Q(n)$'s nullbordism $O(n)$ (the $n$th space of our desired operad) by doing surgery on $\mathrm{FM}_d(n)$. By hypothesis, $M$ is cobordant to $\mathrm{FM}_d(2) = S^{d-1}$.  To start the inductive construction, we seek a manifold $O(3)$ whose boundary is 
$$\big(O(2)\times O(2)\big)\sqcup \big(O(2)\times O(2)\big)\sqcup\big(O(2)\times O(2)\big).$$
Observe that 
\begin{align*}
    \del \big(W(2)\times O(2)\big)&=\big(\mathrm{FM}_d(2)\times O(2)\big)\sqcup\big(O(2)\times O(2)\big),\\
    \del \big(\mathrm{FM}_d(2)\times W(2)\big)&=\big(\mathrm{FM}_d(2)\times \mathrm{FM}_d(2)\big)\sqcup\big(\mathrm{FM}_d(2)\times O(2)\big).
\end{align*}
Define a manifold $X$ by gluing the above two manifolds together along their common boundary $\mathrm{FM}_d(2)\times O(2)$.

The boundary of $X$ is then 
$$\del X=\big(\mathrm{FM}_d(2)\times \mathrm{FM}_d(2)\big)\sqcup\big(O(2)\times O(2)\big).$$
We can now glue $\bigsqcup_3X$ to $\mathrm{FM}_d(3)$ along their common boundary, and the result is a manifold whose boundary is $\bigsqcup_3\big(O(2)\times O(2)\big)$, and so yields the desired nullbordism. We define $O(3)$ to be this nullbordism:
\begin{center}   \includegraphics[scale=0.13]{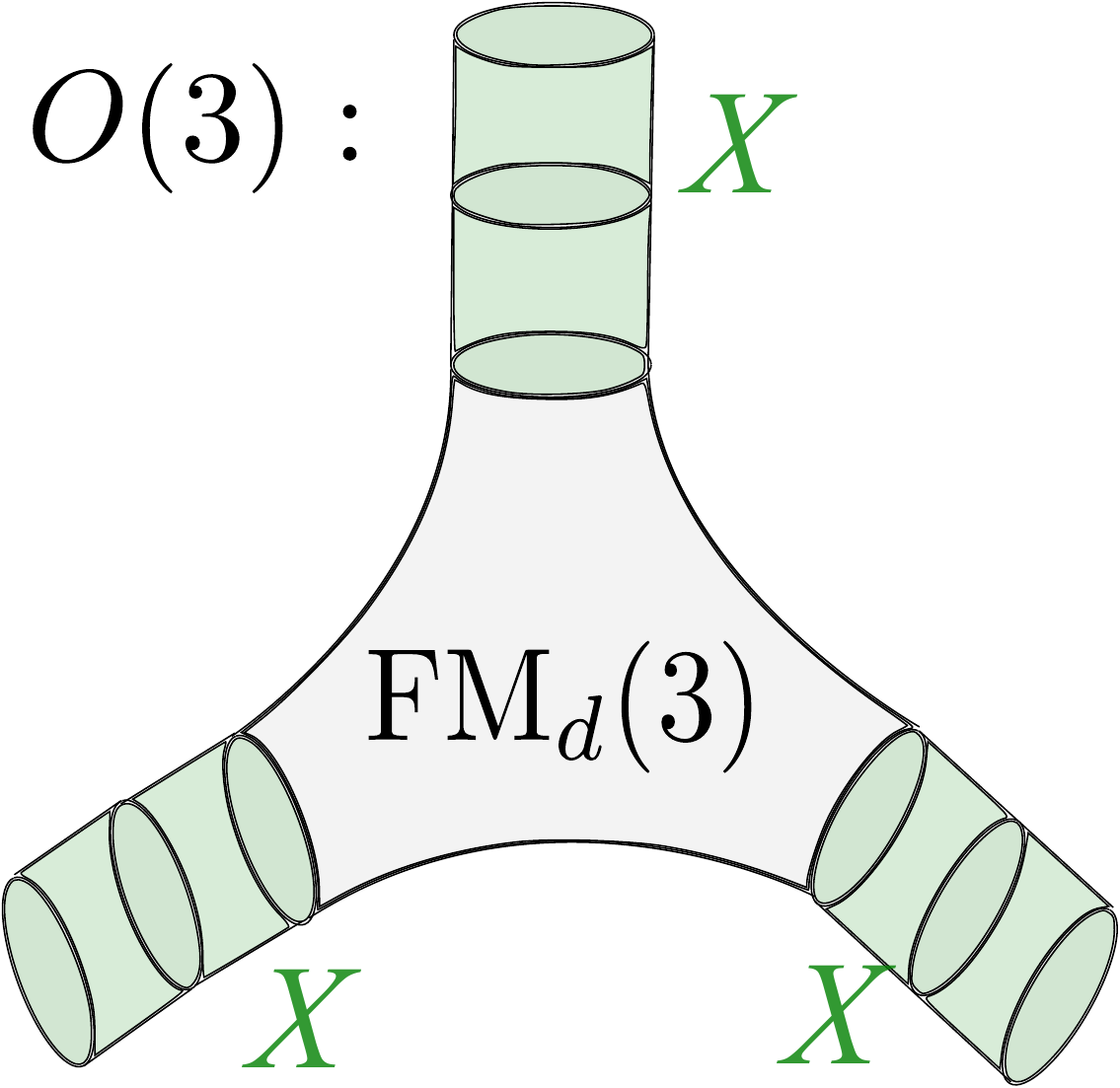}
\end{center}

Defining $O(3)$ like this, we have made use of $W(2)$---a cobordism between $\mathrm{FM}_d(2)$ and $O(2)$. To construct $O(4)$, it is natural to expect that some ``cobordism'' $W(3)$ between $O(3)$ and $\mathrm{FM}_d(3)$ is required, see Definition \ref{RBWmfld_dfn} for our notion of cobordism between manifolds with boundary.

In fact, a natural choice for this cobordism is $W(3):=O(3)\times[0,1]$.
Its boundary can be decomposed into three pieces: 
$$\del W(3)=\mathrm{FM}_d(3)\,\cup \,\sqcup_3X\cup \big(O(3)\,\cup\, \del O(3)\times[0,1]\big)\cong \mathrm{FM}_d(3)\,\cup \,\sqcup_3X\cup O(3),$$
where the ``$\cong$'' holds when we assume $\del O(3)$ has an equivariant collar neighborhood in $O(3)$. 

Constructing $O(4),O(5),\ldots$ is much more complicated. For $n>3$, $\del\mathrm{FM}_d(n)$ is no longer a disjoint union of closed manifolds: $\mathrm{FM}_d(n)$ has the canonical structure of a manifold with corners and its boundary stratification becomes relevant. 
Nonetheless, it turns out that a similar strategy to the arity-3 case will work. Suppose $O(m)$ and $W(m)$ are already constructed for all $2\le m<n$, then by gluing together products of $\mathrm{FM}_d(-),W(-),$ and $O(-)$ according to a certain poset of $3$-colored trees called ``RBW-trees''

\begin{figure}[h]
    \centering    \includegraphics[scale=0.25]{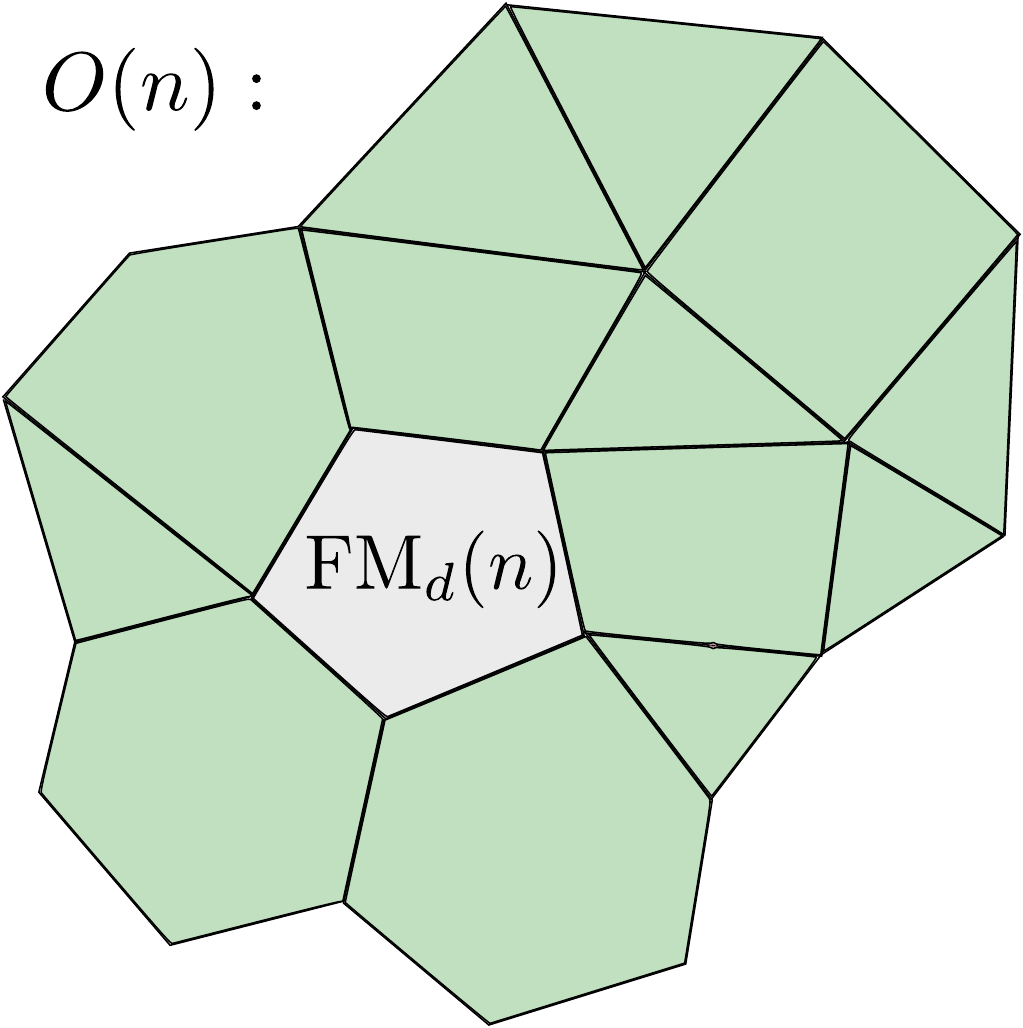}    \caption{Constructing $O(n)$ by gluing pieces to $\mathrm{FM}_d(n)$}
    \label{fig:O(n)_from_FM(n)}
\end{figure}

we may define $O(n)$ and take $W(n)= O(n) \times [0,1]$. The major difficulty then lies in proving that the space $O(n)$ is actually a topological manifold with boundary. 
We prove directly that the links are spheres and discs, and so the space is a topological manifold with boundary. The key to this is the purely combinatorial Proposition \ref{link_prop}. The main technical tool we use to keep track of this data is the notion of stratified bimodules between stratified operads. We demonstrate that the collection $W:=\{W(n)\}_{n=2}^{\infty}$ forms a stratified bimodule between $O$ and $\mathrm{FM}_d$ which implies that $O$ is a manifold operad by Theorem \ref{thm: bimodule cobordism is stratified}.

\textbf{Smoothness.} There is a natural definition of smooth manifold operads which asserts that each $O(m)$ is a smooth manifold with corners, and its canonical boundary stratification agrees with the stratification induced by the operad structure. A natural, yet exceedingly difficult, question is whether or not the construction of operadic surgery just sketched can be replicated in the smooth world.

\subsection{Organization of the paper}

In Section \ref{section: operads and manifolds}, we introduce the main objects of study: manifold operads and bimodule cobordisms. 

In Section \ref{section: Combinatorics of three colored trees}, we study the combinatorics of trees. In particular, we study geometric realizations of posets associated to three-colored trees which govern operadic bimodules.

 Section \ref{section: Stratifications of manifold operads and bimodule cobordisms} is a technical section dedicated to packaging the results of the previous section in such a way that they can be rigorously applied to our situation.

In Section \ref{section: surgery on manifold operads}, we formulate elementary surgery for manifold operads. This main result is the construction of $X_R$ and $X_L$, the two truncated bimodule cobordism extension functors. We study several examples of surgery on $\mathrm{FM}_d$. Examples may be found in Section \ref{subsection: examples}.

Section \ref{section: appendix} is an appendix dealing with point set topology of colimits.

\subsection{Acknowledgments}
We express our thanks to Victor Turchin, Thomas Willwacher, Ricardo Campos, Joost Nuiten, Noah Porcelli, and Kaif Hilman for enlightening conversations.

The authors thank the Max-Planck-Institut f\"ur Mathematik, Bonn where most of this project was worked out. 

Connor Malin would like to thank Ezra Getzler for conversations which motivated this project during his Ph.D. 

Paolo Salvatore acknowledges the MUR
Excellence Project MatMod@TOV awarded to the
	Department of Mathematics, University of Roma Tor Vergata,  CUP E83C18000100006.

Connor Malin also extends his deep gratitude to the members of the Capstone Village for their support while this project was in its early stages. He is especially grateful to Professors Christel and Robert Bell who continue to be a source of inspiration to him. 

The authors used ChatGPT to find the example of Proposition \ref{prop: rightnullbordant}, to identify the spherical tangent bundle in Example \ref{542_example}, and to point out numerous typos after the completion of the paper.

\section{Operads and manifolds} \label{section: operads and manifolds}

\subsubsection{Operads and trees}
A \textit{symmetric sequence} $O$ is a sequence of spaces $\{O(m)\}_{m \in \mathbb{N}}$ equipped with respective $\Sigma_m$-actions. An \textit{operad} is a symmetric sequence $O$ equipped with suitably equivariant and associative composition maps
\[O(k) \times O(j_1) \times \dots \times O(j_k) \rightarrow O(j_1 + \dots + j_k ).\]
An \textit{$n$-truncated operad} $O$ is an \textit{$n$-truncated symmetric sequence} $\{O(m)\}_{m \leq n}$ together with the operadic composition maps which have codomain $O(m)$ for $m \leq n$, satisfying the relations which involve only the spaces $O(m')$ for $m' \leq n$. We will often refer to operads as $\infty$-truncated operads, in which case any conditions involving the top arity of the truncated operad should be ignored. We refer to $O(m)$ as the space of $m$-ary operations.

We will study operads in terms of categories of rooted trees. We work specifically with \textit{reduced operads}: operads $O$ with $O(0)=\emptyset$ and $O(1)=\{*\}$. 
Consequently, our trees have the property (\ref{stability_item}) in Definition \ref{tree_dfn} below.

It will be quite useful to label our spaces by finite sets $S$ rather than cardinalities $n \in \mathbb{N}$, and we will do so without further comment. The term ``tree'' refers to the following combinatorial object unless otherwise stated.

\begin{definition}\label{tree_dfn}
For a finite set $S$, an \textit{$S$-labeled tree} consists of
\begin{itemize}
\item a tree $T$; we denote the vertex set and edge set of $T$ by $V(T), E(T)$\footnote{The given conditions will determine $E(T)$ up to canonical bijection. We implicitly mod out by this bijection; see \cite[Definition 7.2.]{chingSalvatore} for a canonical description of $E(T)$ in terms of $V(T)$.}, respectively;
\item a vertex $r\in V(T)$, ``the root'';
for two vertices $v\neq w\in V(T)$, if the unique path between $v$ and $r$ passes through $w$, we say $w$ is an {\it ancestor} of $v$ and $v$ is a {\it descendant} of $w$, and denote $v>w$; 
if, additionally, $v$ and $w$ are also adjacent, then we say $w$ is the {\it parent} of $v$ and $v$ is a {\it child} of $w$;  
denote the set of children of $v$ by $cld(v)$; 
\item to every vertex $v\in V(T)$ an assignment  $lb(v)\subset S$; elements of $lb(v)$ are called the \textit{labels} of $v$; 
\end{itemize}
The assignments should satisfy the conditions
\begin{enumerate}

\item $\bigsqcup_{v\in V(T)}lb(v)=S$;
\item \label {stability_item}for all $v$, $|cld(v)|+|lb(v)|\ge2$. 
\end{enumerate}

We denote 
$$\overline{cld}(v):=cld(v)\sqcup lb(v).$$
\end{definition}

For an $S$-labeled tree $T$ and a subset $E\subset{E(T)}$, define the contracted tree
$\mathfrak{c}_E(T)$ to be the $S$-labeled tree obtained from $T$ by contracting each connected component in the union of $E$ to a vertex. The labels of these new vertices are the (disjoint) union of the labels of the vertices which contract to it.  For each vertex $v\in \mathfrak{c}_E(T)$, denote by $\mathfrak{c}^{-1}(v)$ the maximal subtree of $T$ that contracts to $v$.

\begin{center}
    \begin{tikzpicture}
        \draw [fill] (0,0) circle [radius=0.02]; 
        \node [left] at (0,0) {\tiny \textcolor{seagreen}{5}};
        \node [below] at (0,0) {\tiny root}; 
        \draw [fill] (-1,1) circle [radius=0.02];
        \node [left] at (-1,1) {\tiny \textcolor{seagreen}{1}};
        \draw [fill] (0,1) circle [radius=0.02];
        \draw [fill] (1,1) circle [radius=0.02];
        \node [right] at (1,1) {\tiny \textcolor{seagreen}{3,9}};
        \draw [fill] (-1.5,1.7) circle [radius=0.02];
        \node [left] at (-1.5,1.7) {\tiny \textcolor{seagreen}{2,6}};
        \draw [fill] (-0.7,2) circle [radius=0.02];
        \node [above] at (-0.7,2) {\tiny \textcolor{seagreen}{7,10,11}};
        \draw [fill] (0.7,2) circle [radius=0.02];
        \node [above] at (0.7,2) {\tiny \textcolor{seagreen}{4,8}};
        \draw (0,0) to (-1,1) to (-1.5,1.7);
        \draw (-0.7,2) to (0,1) to (0.7,2);
        \draw (0,1) to (0,0) to (1,1);
        \node at (-1.35,1.3) {\tiny \textcolor{gray}{$e_1$}};
        \node at (-0.65,0.5) {\tiny \textcolor{gray}{$e_2$}};
        \node at (-0.5,1.5) {\tiny \textcolor{gray}{$e_3$}};
        \node at (0.55,1.5) {\tiny \textcolor{gray}{$e_4$}};
        \node at (-0.13,0.5) {\tiny \textcolor{gray}{$e_5$}};
        \node at (0.7,0.5) {\tiny \textcolor{gray}{$e_6$}};
        
        \node at (0,-1) {\tiny \textcolor{seagreen}{$S=\{1,2,\ldots,11\}$}};

        \node at (3,1) {$\xrightarrow{\mathfrak{c}_{E},\ E=\{e_{1},e_{4},e_{5},e_6\}}$}; 

        \draw [fill] (6,0.5) circle [radius=0.02]; 
        \node [left] at (6,0.5) {\tiny \textcolor{seagreen}{3,4,5,8,9}};
        \node [below] at (6,0.5) {\tiny root}; 
        \draw [fill] (5.5,1.5) circle [radius=0.02];
        \node [left] at (5.5,1.5) {\tiny \textcolor{seagreen}{1,2,6}};
        \draw [fill] (6.5,1.5) circle [radius=0.02];
        \node [right] at (6.5,1.5) {\tiny \textcolor{seagreen}{7,10,11}};
        \draw (5.5,1.5) to (6,0.5) to (6.5,1.5);
    \end{tikzpicture}
\end{center}
If the collection of edges is written as the disjoint union of subtrees $T_1\sqcup \dots \sqcup T_i$, then we write $T / (T_1\sqcup \dots \sqcup T_i )$ for the contracted tree and $[T_k]$ for the quotient vertices.

When we draw a tree, we always draw the root as the lowermost vertex; in future pictures we omit the word ``root'' and always tacitly assume that the root is the lowermost vertex.

By convention, we will draw trees with as few edges and vertices as possible. For example, a tree with leaf set $\{1,2,3,4\}$ and two internal edges can be drawn as
\begin{tikzpicture}[scale=0.7]
        \draw [fill] (0,0) circle [radius=0.02]; 
        \draw [fill] (-0.5,0.5) circle [radius=0.02]; 
        \node [left, seagreen] at (-0.5,0.5) {\tiny $1,2$}; 
        \draw (-0.5,0.5) to (0,0); 
        \draw [fill] (0.5,0.5) circle [radius=0.02]; 
        \node [right,seagreen] at (0.5,0.5) {\tiny $3,4$}; 
        \draw (0.5,0.5) to (0,0); 
    \end{tikzpicture} 
instead of as
\begin{tikzpicture}[scale=0.7]
        \draw [fill] (0,0) circle [radius=0.02]; 
        \draw [fill] (-0.5,0.5) circle [radius=0.02]; 
        \draw (-0.5,0.5) to (0,0); 
        \draw [fill] (0.5,0.5) circle [radius=0.02]; 
        \draw (0.5,0.5) to (0,0); 
        
        \draw [fill, seagreen] (-1,1) circle [radius=0.02]; 
        \draw [fill,seagreen] (-0.3,1) circle [radius=0.02]; 
        \draw [fill,seagreen] (0.3,1) circle [radius=0.02]; 
        \draw [fill,seagreen] (1,1) circle [radius=0.02]; 
        \draw[seagreen] (-1,1) to (-0.5,0.5) to (-0.3,1); 
        \draw[seagreen] (0.3,1) to (0.5,0.5) to (1,1);
        \draw [brown] (0,0) to (0,-0.5);

        \node [above, seagreen] at (-1,1) {\tiny 1}; 
        \node [above, seagreen] at (-0.3,1) {\tiny 2}; 
        \node [above, seagreen] at (0.3,1) {\tiny 3}; 
        \node [above, seagreen] at (1,1) {\tiny 4}; 
    \end{tikzpicture} 
which is another common convention. We use this convention because it allows us to drop the set of leaves from the picture when it is not relevant to the argument, as well as avoid repeatedly distinguishing between leaf edges and internal edges.

We use the notation $\bullet_S$ to denote the $S$-labeled tree with a single vertex. 
When $S=\{1,\ldots,n\}$, we denote $\mathcal{T}(n):=\mathcal{T}(S)$ and $\bullet_n:=\bullet_S$. 

\begin{definition}
    The category of $S$-labeled trees $\mathcal{T}(S)$ is the poset whose objects are $S$-labeled trees and has morphisms $ T \leq \mathfrak{c}_E(T)$.
\end{definition}

We say that a tree is \textit{decomposable} if its set of edges is non-empty. 

A tree is a \textit{corolla} if it has no edges. The $S$-labeled corolla $\bullet_S$ is the maximal element of  $\mathcal{T}(S)$.

\begin{definition}
    For $n\leq \infty$, the category of $n$-truncated trees $\mathcal{T}_{\leq n}$ has objects the $S$-labeled trees for all finite sets $S$ such that $|S| \leq n$ and morphisms $T \rightarrow T'$ are given by a pair
    $(f,\mathfrak{c})$
    where $f$ is a bijection $lb(T) \cong lb(T')$, and $\mathfrak{c}$ is a contraction from $T$ to $f^{-1}(T')$, where $f^{-1}(T')$ denotes relabeling via the bijection $f$.
\end{definition}

We shall write $\mathcal{T}:=\mathcal{T}_{\leq \infty}$, and $\mathcal{T}_n$ for the subcategory of $\mathcal{T}$ of trees with label set exactly of cardinality $n$.

\begin{definition}
For $n\leq \infty$, an \textit{$n$-truncated reduced operad} $O$ is specified by the data of a functor
     \[\phi:\mathcal{T}_{\leq n} \rightarrow \mathrm{Top}\]
     equipped with compatible and equivariant homeomorphisms
       \[\phi(T)\cong \prod_{v \in V(T)} \phi(\bullet_{\overline{cld}(v)})\]
   such that 
    if $T$ is covered by disjoint subtrees $T_1,\dots,T_i$ then the value of $\phi$ on a contraction $T \leq T / (T_1\sqcup \dots \sqcup T_i )$ is given by
       \[\phi(T \rightarrow T') \cong  \prod^i_{k=1} \phi(T_k \rightarrow \bullet_{\overline{cld}([T_k])}).\]
       
Write
$$O(m):=\phi(\bullet_{\{1,\ldots,m\}}) \textnormal{ for }{2 \leq m \leq n}.$$
 \end{definition}

By considering trees with one internal edge, one may extract maps
\[O(m-j+1) \times O(j)  \rightarrow O(m)\]
which determine the operad structure. These are called the partial composites of $O$.

\begin{definition}
    For $n\leq \infty$, given an $n$-truncated reduced operad $O$ represented by $\phi:\mathcal{T}_{\leq n} \rightarrow \mathrm{Top}$, then for $S$ with $|S|=m\leq n$ the operad decomposables are 
    \[\mathrm{Decom}(O)(S) := \bigcup_{T \in \mathcal{T}(S) \setminus \bullet_{S} } \mathrm{Image}\big(\phi(T) \rightarrow \phi(\bullet_{S})\big).\]
\end{definition}

The space $\mathrm{Decom}(O)(S)$ has a universal approximation which depends only on $O^{\leq m-1}$:
\[\colim_{{\mathcal{T}(S)\backslash\bullet_{S}}}\phi. \]
It supports a continuous and surjective map 
\[\colim_{{\mathcal{T}(S)\backslash\bullet_{S}}}\phi \rightarrow \mathrm{Decom}(O)(S),\]
and the fibers of this map encode the relations the truncated operad $O^{\leq m}$ satisfies.

\subsection{Manifold operads}

We are interested in a certain class of operads which intertwine important categorical invariants of operads with the point set topology of manifolds. Unless otherwise specified, we take $n \leq \infty$.

\begin{definition}\label{manifoldoperad_dfn}

    An \textit{$n$-truncated, $d$-dimensional manifold operad} is an $n$-truncated operad $O$, represented by $\phi: \mathcal{T}_{\leq n} \rightarrow \mathrm{Top}$, which for all $S$ with $|S|=m\leq n$ satisfies:
    \begin{enumerate}
        \item \label{manifoldoperad1_item}
        $O(S)$ is a compact $(md-d-1)$-dimensional topological manifold with boundary, and its boundary is $\mathrm{Decom}(O)(S)$.
        \item \label{manifoldoperad2_item}
        The canonical map  
        \[\colim_{{\mathcal{T}(S)\backslash\bullet_{S}}}\phi 
          \rightarrow \mathrm{Decom}(O)(S) \]
       is a homeomorphism.
    \end{enumerate}
\end{definition}

If the dimension $d$ is not relevant, we will  refer to $O$ as an $n$-truncated manifold operad. If $n=\infty$, we will refer to them as manifold operads.

This definition implies, by  Lemma \ref{colimtopology_lmm} and Lemma \ref{operadtopology_lmm}, that for every $T\in \mathcal{T}(S)\backslash\bullet_S$, the obvious map $\phi(T)\to \colim_{\mathcal{T}(S)\backslash\bullet_S}\phi$ (therefore also the map $\phi(T)\to O(S)$) is a closed embedding, and the images of these maps only intersect ``when necessary''.

The prototypical example of a $d$-dimensional manifold operad is $\FM_d$, the $d$th Fulton-MacPherson operad, originally defined by Getzler and Jones \cite{getzler_jones}, see also
\cite{sinhacompact,salvatore_2022}.

We recall the definition of its underlying spaces.
\begin{definition}
    Consider the map 
\begin{align*}
    \psi_d:\mathrm{Conf}(\mathbb{R}^d,m)
&\to (S^{d-1})^{\binom{m}{2}} \times [0,+\infty]^{\binom{m}{3}} \\
(x_1,\dots,x_m) &\mapsto ((\frac{x_i-x_j}{|x_i-x_j|})_{i<j},(|x_i-x_j|/|x_i-x_k|)_{i<j<k})
\end{align*}
The space $\mathrm{FM}_d(m)$ is the closure of the image 
of $\psi_d$. 
\end{definition}
We recall that $\FM_d(m)$ is a smooth manifold with corners of dimension $md-d-1$.
Its interior is identified via $\psi_d$ to $\mathrm{Conf}(\mathbb{R}^d,m)$, the configuration space of $m$ ordered points in $\mathbb{R}^d$, modulo translations and positive scaling. 
The boundary strata of $\mathrm{FM}_d(m)$ are indexed by decomposable trees in $\mathcal{T}_m$ and correspond to possible degenerations of configurations where some points come infinitesimally close. The operad composition replaces a point in a configuration by a cluster of points that are infinitesimally close to each other. 
Precisely the operad composition
$\circ_i: \FM_d(m) \times \FM_d(n) \to \FM_d(m+n-1)$
is the unique map such that
$$(x_1,\dots,x_m)\circ_i (y_1,\dots,y_n)= \lim_{\varepsilon \to 0^+}(x_1,\dots,x_{i-1},x_i+\varepsilon{y_1},\dots,x_i+\varepsilon{y_n},x_{i+1},\dots,x_m)$$
when $(x_1,\dots,x_m)$ and $(y_1,\dots,y_n)$ are configurations with barycenter in the origin and norm 1.

The Fulton-MacPherson operads are manifold operads, and in fact seem to be the only known examples. However, for finite $n$, we will see that in fact it is quite easy to produce examples of $n$-truncated manifold operads. These may be obtained by fixing an $n$-truncated manifold operad and performing a $\Sigma_n$-equivariant surgery on the interior of the manifold of $n$-ary operations. For instance, corresponding to surgery on the orbit of a small embedded $S^0 \times D^{nd-d-1}$ in $\mathring{\mathrm{FM}}_d(n)$, there is an $n$-truncated manifold operad of dimension $d$ such that
\[
  O(m) \cong
  \begin{cases} 
   \FM_d(m) & m <n\\
    \FM_d(n) \#_{n!} (S^{ 1} \times S^{ nd-d-2})  & m =n.
  \end{cases}
\]
Issues arise when trying to extend this technique to surgery below the top arity. This is because the operad structure necessarily propagates the effects of surgery towards the higher arities.  In fact, it is quite easy to obstruct the existence of certain manifold operads. Recall that the collection of all unoriented topological bordism groups \[\Omega^\mathrm{Top}:= \bigoplus_{n \in \mathbb{N}} \Omega_n^\mathrm{Top}\] supports the structure of a commutative ring where addition is given by disjoint union and multiplication given by cartesian product.

\begin{prop} \label{prop: example nonexistence}

    If $M$ is such that
    \[[M]^2 \neq 0 \in \Omega^\mathrm{Top}, \]
    then there is no $\geq 3$-truncated manifold operad satisfying $O(2)=M$.
\end{prop}
\begin{proof}
    If $O$ is a $3$-truncated manifold operad, then $\partial O(3) = \mathrm{Decom}(O)(3) \cong  (O(2) \times O(2)) \sqcup (O(2) \times O(2)) \sqcup (O(2) \times O(2))$. By definition, $O(3)$ will be a nullbordism of $3[M]^2$, and so represent $0$ in the cobordism group. If $3[M]^2=0$ in the unoriented bordism ring, then $[M]^2=0$ since every element is $2$-torsion.

\end{proof}

\begin{cor}
      There is no manifold operad with $O(2)= \mathbb{R}P^2$.
\end{cor}

\subsection{Obstructing collarable manifold operads}

Fix some $n \leq \infty$.

\begin{definition}
    A collarable $n$-truncated manifold operad $O$ is an $n$-truncated manifold operad $O$ such that for all $m \leq n$ there exists a $\Sigma_m$-equivariant collar neighborhood of $\partial O(m) \subset O(m)$.
\end{definition}

In practice, this collaring condition is not too severe. If $G$ is a finite group, a smooth $G$-manifold with boundary always admits a $G$-equivariant collar. 
If $G$ is a finite group acting freely on a topological $G$-manifold with boundary, then a $G$-equivariant collar neighborhood exists, by lifting the collar of the quotient manifold.

Therefore we know that $\mathrm{FM}_d$ is collarable.

\begin{definition}
    Let $G$ be a finite discrete group. By a \textit{$G$-cobordism of $\,G$-manifolds} $W$ we mean a compact $G$-manifold $W$ and an ordered $G$-invariant decomposition of the boundary
    \[\partial W = \partial_1 W \sqcup \partial_2 W.\]
\end{definition}

We say that a $G$-cobordism is collarable if it admits an equivariant collar of its boundary. This collarability is required, for instance, if one wishes to show that the trivial cobordism acts as an identity for the composition of equivariant cobordisms. 

\begin{definition}
    For a finite group $G$, the \textit{topological $G$-cobordism group} $\Omega_{m}^{G\mathrm{Top}}$ is obtained as the quotient of the monoid (under disjoint union) of closed topological $m$-dimensional manifolds with a $G$-action  by $G$-cobordisms.
\end{definition}

Note that although the collection of $G$-cobordisms does not coincide with the collection of collarable $G$-cobordisms, there exists a $G$-cobordism between $M$ and $N$ if and only if there exists a collarable $G$-cobordism between $M$ and $N$. This is because one can glue an external collar onto any $G$-cobordism without affecting the equivariant homeomorphism type of the boundary. We state and prove the following which depends crucially on the rather technical Section \ref{section: Stratifications of manifold operads and bimodule cobordisms}.

\begin{prop} \label{prop: operad obstruction}
    If $O$ is an $n$-truncated, $d$-dimensional collarable manifold operad, then there is a canonical obstruction in $\Omega_{nd-2}^{\Sigma_{n+1}\mathrm{Top}}$ which vanishes if and only if $O$ extends to an $(n+1)$-truncated, $d$-dimensional collarable manifold operad.
\end{prop}

\begin{proof}

  Let $O$ be represented by $\phi:\mathcal{T}_{\leq n} \rightarrow \mathrm{Top}$.  We will take our obstruction to be $$[\colim_{{\mathcal{T}(n+1)\backslash\bullet_{n+1 }}}\phi ]\in \Omega_{nd-2}^{\Sigma_{n+1}\mathrm{Top}}$$ This colimit has a canonical $\Sigma_{n+1}$-action and by Theorem \ref{thm: manifold operad is stratified} combined with Proposition \ref{prop: decomposables are manifold} is a manifold. Checking the maximal faces of the colimit, corresponding to trees with one edge, one sees that the dimension must be
    \[\mathrm{dim}\big(O(n+2-j) \times O (j)\big) = ((n+2-j)d -d -1) + (jd -d -1) = nd +2d -jd -d -1 +jd-d-1=nd-2. \]

    Hence, the obstruction $[\colim_{{\mathcal{T}(n+1)\backslash\bullet_{n+1 }}}\phi]\in \Omega_{nd-2}^{\Sigma_{n+1}\mathrm{Top}}$ is well defined. We now must show that this is a complete obstruction to the problem of extending $O$ to a collarable $(n+1)$-truncated $d$-dimensional manifold operad. 

    If $O(n+1)$ is defined so as to extend $O$ to a $(n+1)$-truncated $d$-dimensional manifold operad, then by the definition of a manifold operad $O(n+1)$ is a nullbordism of $\colim_{{\mathcal{T}(n+1)\backslash\bullet_{n +1}}}\phi$. Hence, the obstruction is trivial in $\Omega_{nd-2}^{\Sigma_{n+1}\mathrm{Top}}$. Conversely, given a nullbordism $W$ of $\colim_{{\mathcal{T}(n+1)\backslash\bullet_{n+1 }}}\phi$, which can be assumed to be collarable by our previous discussion, we may simply set $O(n+1):=W$ and define its new operadic structure maps to factor through the boundary inclusion $\colim_{{\mathcal{T}(n+1)\backslash\bullet_{n+1 }}}\phi\hookrightarrow W$. This satisfies the definition of a $d$-dimensional $(n+1)$-truncated manifold operad by construction.
\end{proof}

\begin{remark}
    The above obstruction theory does not readily generalize to the entire class of manifold operads, in the sense that for an arbitrary $n$-truncated manifold operad $O$ the space $\colim_{{\mathcal{T}(n+1)\backslash\bullet_{n +1}}}\phi$ doesn't seem to be a manifold.
\end{remark}

It is natural to ask whether the collars of a collarable manifold operad can be chosen to respect the operad composition. In order to formulate this condition, we recall a variant of the $W$-construction of an operad \cite{boardman_vogt_1973}.

\begin{definition}
If $O$ is an $n$-truncated operad, the \textit{thickened $W$-construction} $W_{[0,1]}(O)$ is the $n$-truncated operad defined by setting for $S$ with $|S|\leq n$
    \[W_{[0,1]}(O)(S):= \left(\bigsqcup_{T \in \mathcal{T}(S)} \prod_{v \in V(T)} [0,1] \times O\big(\overline{cld}(v)\big)\right) \bigg/ \sim.\]
We describe the relation $\sim$ in terms of \textit{``$O$-labeled metric trees''} by noting that a point in the disjoint union may be identified with 
\begin{itemize}
    \item An $S$-labeled tree $T$ with the addition of a hair\footnote{A hair is an edge attached to only one vertex, also called a half-edge.} attached to its root.
    \item An assignment of a length $t_e \in [0,1]$ to each edge $e \in E(T)$, as well as to the hair. 
    \item An assignment of a point in $O(\overline{cld}(v))$ to each vertex $v \in V(T)$.
\end{itemize}
Under this identification, the relation $\sim$ contracts length $0$ edges, but not a length $0$-hair, and uses the operadic composition to multiply the relevant operad labels. Operadic composition is given by grafting trees via the root's hair.
\end{definition}

\begin{definition}
    Given an $n$-truncated operad $O$, the \textit{$W$-construction} $W(O)$ is the $n$-truncated suboperad of $W_{[0,1]}(O)$ for which the underlying metric tree has a length $1$ hair attached to the root.
\end{definition}

It is important to note that as $n$-truncated symmetric sequences $W(O) \times [0,1] \cong W_{[0,1]}(O)$, though this identification does not interact with the operad structures in a reasonable way.

\begin{prop}\label{prop: isomorphic to w construction}
    An $n$-truncated collarable manifold operad admits an $n$-truncated operad isomorphism
    \[O \cong W(O).\]
    Furthermore, for $m \leq n$ the subspace of $W(O)(m)$ consisting of points which admit a representation by a decomposable labeled metric tree is a collar neighborhood of \[\partial W(O)(m) = \mathrm{Decom}(W(O))(m).\]
\end{prop}
\begin{proof}
The proof is the same as that of the main theorem of \cite{salvatore_2021}, up to arity $m$, replacing $\FM_d$ with $O$. 
\end{proof}
The data of the isomorphism from the proposition encodes a collection of compatible tubular neighborhoods of the boundary strata of $O$. Note that an obvious consequence of this proposition is that the $W$-construction of a collarable manifold operad is a manifold operad, though the full strength of equivariant collars is not necessary to prove this, see Corollary \ref{cor: w of stratified}.

We briefly remark on two homotopical implications of this result. Recall that an $n$-truncated operad is called $\Sigma$-cofibrant if the underlying symmetric sequence is cofibrant in the projective model structure on symmetric sequences. When everything involved is a compact manifold, this is quite close to, but unfortunately not exactly equivalent with, requiring that the symmetric group actions are free. 

\begin{cor}
A collarable $n$-truncated manifold operad is cofibrant in the projective model structure on $n$-truncated operads if and only if the underlying symmetric sequence is cofibrant in the projective model structure on $n$-truncated symmetric sequences.
\end{cor}
\begin{proof}
    Assume $n=\infty$, the truncated case being similar. The $W$-construction of a $\Sigma$-cofibrant operad is cofibrant in the projective model structure on operads \cite[4.1. Theorem]{vogtcofibrant}, and all cofibrant operads are $\Sigma$-cofibrant  \cite[Proposition 2.3.1]{chainruleac}.
\end{proof}

Recall that there is a bar construction $B$ for topological operads lifting the algebraic bar construction \cite[Proposition 9.39]{chingDerivativesOfIdentity}. Let $s_d$ denote algebraic operadic suspension, defined as tensoring with the endomorphism operad of $k[-d]$ where $k$ is the ground field.

\begin{prop}\label{prop: koszul self dual}
    A $d$-dimensional, possibly truncated, collarable manifold operad  satisfies
    \[H_\ast(O;\mathbb{Z}/2) \cong s_d \bar{H}_\ast (B(C_\ast(O; \mathbb{Z}/2)^\vee)).\]
\end{prop}
\begin{proof}
    The statement and proof of the theorem are identical to those of \cite[Theorem 5.10]{malinhha} which is a simplification of the original argument for the Fulton-MacPherson operad \cite[Corollary 3.5]{getzler_jones}. Briefly, since $W(O)$ is a manifold operad by Corollary \ref{cor: w of stratified}, the operad partial composites of $\bar{H}^\ast (B(O);\mathbb{Z}/2)$ are Poincaré-Lefschetz dual to the partial composites of  $H_\ast(W_{[0,1]}(O);\mathbb{Z}/2) \cong H_\ast (O;\mathbb{Z}/2)$ \cite[Proposition 5.8]{malinhha}.
\end{proof}

\begin{remark}

Proposition \ref{prop: koszul self dual} is stated with $\mathbb{Z}/2$-coefficients. It is possible to formulate this theorem with $\mathbb{Z}$-coefficients for a suitable definition of oriented manifold operads which requires some attention to signs. The lack of signs in the second author's definition of an oriented Poincaré duality operad \cite[Definition 5.2]{malinhha} means that when $d$ is even the Fulton-MacPherson operads do not satisfy the given definition. This is an error and is a consequence of incorrectly defining operadic suspension.
\end{remark}

\subsection{Bimodules and trees}\label{subsection: bimodules and trees}
Analogies are abound between bimodules between algebras and cobordisms between manifolds. With no eye to subtlety, we prepare to study the class of operadic bimodules between manifold operads that form a symmetric sequence of cobordisms between the underlying symmetric sequences of manifolds.

Recall that if $B$ and $R$ are operads, then a $(B,R)$-bimodule $W$ is classically defined as a symmetric sequence $W$ together with commuting left and right actions
\[B \circ W \circ R \rightarrow W,\]
where $\circ$ is a certain product of symmetric sequences.
 Bimodules arise naturally in many instances. Given a span of operads $P \leftarrow O \rightarrow P'$ one may form the relative composition product $P \circ_O P'$, defined as the coequalizer of the diagram
\begin{center}

\begin{tikzcd}
P \circ O \circ  P' \ar[r,shift left=.75ex]
  \ar[r,shift right=.75ex,swap,]
&
P \circ P' \ar[r] 
&
P \circ_O P'. 

\end{tikzcd}

\end{center}
This symmetric sequence inherits a left action by $P$ and a right action by $P'$ making it into a $(P,P')$-bimodule. Dually, given a cospan of operads \[P \rightarrow O \leftarrow P',\] $O$ naturally inherits the structure of a $(P,P')$-bimodule, as well as a $(P',P)$-bimodule, by pushing forward elements of $P$ and $P'$ along the operad maps and using the composition in $O$.

 Given a $(B,R)$-bimodule $W$ such that all three symmetric sequences $R,B,W$ are reduced 
\[R(0)=B(0)=W(0)=\emptyset,\]
\[R(1)=B(1)=W(1)=\ast,\] we will say that the bimodule is {\it reduced}. Reduced bimodules can be described in terms of a category of three-colored trees. The objects of this category are defined as follows.

\begin{definition}\label{RBtree_dfn}
    Given an $S$-labeled tree $T$, an \textit{$(R,B,W)$-coloring} of $T$ is an assignment of an element of $\{R,B, W \}$\footnote{They represent red, blue and white.} to each vertex of $T$, which we call its \textit{color}. The colors must satisfy the conditions:
    \begin{itemize}
        \item all descendants of a vertex of color $R$ must be of color $R$; 
        \item all ancestors of a vertex of color $B$ must be of color $B$;
        \item no vertex of color $ W $ can be an ancestor or descendant of another vertex of color $ W $. 
    \end{itemize}
We use 
$$\text{Color}:V(T)\to \{R,B, W \}$$
to denote the assignment of a color. 
We call such a tree a \textit{legal RBW-tree} or simply an \textit{RBW-tree}. 
\end{definition}
 If an assignment $V(T)\to\{R,B, W \}$ does not satisfy the above conditions it is an \textit{illegal coloring}.

 There are several equivalent formulations of legal colorings. A coloring is legal if each path in $T$ going from a leaf to the root is of the form 
$$\underbrace{R\ R\ \ldots\ R\ }_{\ge0 \text{ times}}\underbrace{ W }_{0 \text{ or }1 \text{ time}}\underbrace{\ B\ldots\ B\ B}_{\ge0\text{ times}}.$$
This can be checked on pairs: a coloring is legal exactly when for each pair of vertices $v_1,v_2\in V(T)$ such that $v_2$ is an ancestor of $v_1$, the coloring of $(v_1,v_2)$ is one of the following ``legal colorings'':
$$(R,R), (R, W ),(R,B),( W ,B), (B,B).$$ 
It is even enough to check this condition on pairs $(v_1,v_2)$ such that $v_2$ is a parent of $v_1$.
The remaining \textit{illegal} colorings of such pairs $(v_1,v_2)$ are then $( W , R), ( W , W ), (B,R), (B, W )$. 
Notice that 
\begin{itemize}
    \item changing a color in an illegal coloring of such a pair from anything to $ W $ will never make the coloring of the pair legal; 
    \item changing a color in a legal coloring of such a pair from $ W $ to anything else, the resulting coloring of such a pair is still legal; as a consequence, given a legal $RBW$-tree, changing any $W$-colored vertex to either $B$ or $R$ will result in a legal $RBW$-tree. 
\end{itemize}

For $C=R,B$ or $ W $, we use $\bullet_{C,S}$ to denote the colored $S$-labeled tree with only one vertex and its color is $C$. When $S$ is clear from the context, we also just write $\bullet_C$.

\begin{definition}\label{contraction_dfn}
    An  $RBW$-\textit{contraction} $\mathfrak{c}$ of colored, labeled trees consists of 
    \begin{itemize}
        \item a finite set $S$ and two colored $S$-labeled trees $T',T$; 
        \item a contraction $\mathfrak{c}:T'\to T$ as uncolored $S$-labeled trees;
    \end{itemize}
    satisfying the condition: 
    \begin{itemize}
        \item for each $v\in V(T)$, if the color of $v$ is $R$ (resp. $B$), then all the vertices of $\mathfrak{c}^{-1}(v)$ must be of color $R$ (resp. $B$). 
    \end{itemize}
\end{definition}

\begin{center}
    \begin{tikzpicture}
        \draw [fill,blue] (0,0) circle [radius=0.02]; 
        \node [left] at (0,0) {\tiny \textcolor{seagreen}{5}};
        \node [right] at (0,0) {\tiny \textcolor{blue}{B}};
        \draw [fill,red] (-1,1) circle [radius=0.02];
        \node [left] at (-1,1) {\tiny \textcolor{seagreen}{1}};
        \node [right] at (-1,1) {\tiny \textcolor{red}{R}};
        \draw [fill] (0,1) circle [radius=0.02];
        \node [right] at (0,1) {\tiny {W}};
        \draw [fill,blue] (1,1) circle [radius=0.02];
        \node [above] at (1,1) {\tiny \textcolor{seagreen}{3,9}};
        \node [right,blue] at (1,1) {\tiny \textcolor{blue}{B}};
        \draw [fill,red] (-1.5,1.7) circle [radius=0.02];
        \node [left] at (-1.5,1.7) {\tiny \textcolor{seagreen}{2,6}};
        \node [right] at (-1.5,1.7) {\tiny \textcolor{red}{R}};
        \draw [fill,red] (-0.7,2) circle [radius=0.02];
        \node [above] at (-0.7,2) {\tiny \textcolor{seagreen}{7,10,11}};
        \node [right] at (-0.7,2) {\tiny \textcolor{red}{R}};
        \draw [fill,red] (0.7,2) circle [radius=0.02];
        \node [above] at (0.7,2) {\tiny \textcolor{seagreen}{4,8}};
        \node [right] at (0.7,2) {\tiny \textcolor{red}{R}};
        \draw (0,0) to (-1,1) to (-1.5,1.7);
        \draw (-0.7,2) to (0,1) to (0.7,2);
        \draw (0,1) to (0,0) to (1,1);
        \node at (-1.35,1.3) {\tiny \textcolor{gray}{$e_1$}};
        \node at (-0.65,0.5) {\tiny \textcolor{gray}{$e_2$}};
        \node at (-0.5,1.5) {\tiny \textcolor{gray}{$e_3$}};
        \node at (0.55,1.5) {\tiny \textcolor{gray}{$e_4$}};
        \node at (-0.13,0.5) {\tiny \textcolor{gray}{$e_5$}};
        \node at (0.7,0.5) {\tiny \textcolor{gray}{$e_6$}};

        \node at (3,1) {$\xrightarrow{\mathfrak{c}_{E},\ E=\{e_{1},e_{4},e_{5},e_6\}}$}; 

        \draw [fill] (6,0.5) circle [radius=0.02]; 
        \node [left] at (6,0.5) {\tiny \textcolor{seagreen}{3,4,5,8,9}};
        \node [right] at (6,0.5) {\tiny {W}};
        \draw [fill,red] (5.5,1.5) circle [radius=0.02];
        \node [left] at (5.5,1.5) {\tiny \textcolor{seagreen}{1,2,6}};
        \node [right] at (5.5,1.5) {\tiny \textcolor{red}{R}};
        \draw [fill] (6.5,1.5) circle [radius=0.02];
        \node [above] at (6.5,1.5) {\tiny \textcolor{seagreen}{7,10,11}};
        \node [right] at (6.5,1.5) {\tiny \textcolor{red}{R}};
        \draw (5.5,1.5) to (6,0.5) to (6.5,1.5);
        \node at (3,-0.2) {\tiny{A valid $RBW$-contraction}}; 
    \end{tikzpicture}
\end{center}

One immediately sees that the composition $\mathfrak{c}\circ\mathfrak{c}':T''\to T$ of two contractions $\mathfrak{c'}:T''\to T'$ and $\mathfrak{c}:T'\to T$ is again a contraction. 
There is at most one contraction between two trees, and so to specify a contraction it suffices to list the domain and target. We call the identity contraction $T\to T$ \textit{trivial}. 

Notice that there are instances where simply changing the color of a vertex from $R$ to $W$ or from $B$ to $W$ is a contraction, and this would be considered a nontrivial contraction. However, this is not always a valid contraction. 
    For example, in the tree 
    \begin{tikzpicture}[scale=0.7]
        \draw [fill] (0,0) circle [radius=0.02]; 
        \node [right] at (0,0) {\tiny $B$};
        \draw [fill] (-0.5,0.5) circle [radius=0.02]; 
        \node [left] at (-0.5,0.5) {\tiny $ W $}; 
        \draw (-0.5,0.5) to (0,0); 
        \draw [fill] (0.5,0.5) circle [radius=0.02]; 
        \node [right] at (0.5,0.5) {\tiny $B$}; 
        \draw (0.5,0.5) to (0,0); 
    \end{tikzpicture} 
    the upper $B$-colored vertex can be turned into a $ W $-colored vertex, but the lower $B$ vertex cannot; the right $B-B$ edge can be contracted to a vertex of color $B$, but the left $ W -B$ edge cannot be contracted to a vertex of color $ W $. 

\begin{definition}
    The poset $\mathcal{T}^{RBW}(S)$ has objects the $S$-labeled RBW-trees with morphisms given by $RBW$-contractions.
\end{definition}

\begin{definition}\label{labelchangetreecategory_dfn}

    For $n\leq \infty$, the category of $n$-truncated $RBW$-trees $\mathcal{T}^{RBW}_{\leq n}$ has objects the $S$-labeled trees for all finite $S$ such that $|S| \leq n$ and morphisms $T \rightarrow T'$ are given by a pair
    $(f,\mathfrak{c})$
    where $f$ is a bijection $lb(T) \cong lb(T')$, and $\mathfrak{c}$ is an $RBW$-contraction from $T$ to $f^{-1}(T')$.

\end{definition}

  The category $\mathcal{T}_{\leq n}^{RBW}(S)$ has subcategories consisting of the trees all of whose vertices are red or blue, respectively, $\mathcal{T}^{R}(S),\mathcal{T}^{B}(S)$:
\[\mathcal{T}(S)\cong \mathcal{T}^B(S)\subset \mathcal{T}^{RBW}(S) \supset  \mathcal{T}^R(S) \cong \mathcal{T}(S),\] 
and similarly for the category of trees where the labels may vary:
\[\mathcal{T}_{\leq n}\cong \mathcal{T}_{\leq n}^{B}\subset \mathcal{T}_{\leq n}^{RBW} \supset  \mathcal{T}_{\leq n}^{R} \cong \mathcal{T}_{\leq n}.\] 

We shall sometimes write $\mathcal{T}^{RBW}:=\mathcal{T}^{RBW}_{\leq\infty} $. Let $\mathcal{T}^{RBW}_{n}$ be the subcategory of $\mathcal{T}^{RBW}$ consisting of trees whose label set is of size $n$. 
\begin{definition}\label{bimodule_dfn}
For $n \leq \infty$,
 a $n$-truncated \textit{reduced $(B,R)$-bimodule} $W$ is specified by the data of a functor
     \[F:\mathcal{T}_{\leq n}^{RBW} \rightarrow \mathrm{Top}\]
     equipped with compatible, equivariant isomorphisms
       \[F(T)\cong \prod_{v \in V(T)} F(\bullet_{\mathrm{color}(v),\overline{cld}(v)})\]
    such that 
    if $T$ is covered by disjoint subtrees $T_1,\dots,T_i$ then the value of $F$ on a legal $RBW$-contraction defined via \[T \leq T / (T_1\sqcup \dots \sqcup T_i )\] \[\mathrm{Color}([T_k])= \eta_k\] is given by
       \[F(T \rightarrow T') \cong  \prod^i_{k=1} F(T_k \rightarrow \bullet_{\eta_k,\overline{cld}(T_k)}),\] so that
   \begin{itemize}
       \item The reduced $n$-truncated operad 
       defined by restricting $F$ to $\mathcal{T}_{\leq n}^R$ is $R$.
       \item The reduced $n$-truncated  operad 
       defined by restricting $F$ to $\mathcal{T}_{\leq n}^B$ is $B$.
   \end{itemize}
Write $W(m):=F(\bullet_{W,m})$ for ${2 \leq m \leq n}$.
 \end{definition}

\begin{remark}
    A reduced $(B,R)$-bimodule $W$ (for $n=\infty$) in the above sense is an operadic bimodule in that it admits bimodule structure maps
    \[B \circ W \circ R \rightarrow W.\]
The bimodule structure maps can be defined via the value of the functor on contractions of the form 
\begin{center}
    \begin{tikzpicture}
        \draw [blue] (0,0) circle [radius=0.4]; 
        \node at (0,0) {\tiny \textcolor{blue}{\begin{tabular}{c} a \\ $B$-tree 
        \end{tabular}}};
        \draw [fill] (-1,1) circle [radius=0.02];
        \node [right] at (-1,1) {\tiny {W}};
        \draw [fill] (0,1) circle [radius=0.02];
        \node [right] at (0,1) {\tiny {W}};
        \draw [fill] (1,1) circle [radius=0.02];
        \node [right] at (1,1) {\tiny {W}};
        \draw [red] (-1,1.7) circle [radius=0.4];
        \node at (-1,1.7) {\tiny \textcolor{red}{\begin{tabular}{c} an \\ $R$-tree 
        \end{tabular}}};
        \draw [red] (0,2) circle [radius=0.4];
        \node at (0,2) {\tiny \textcolor{red}{\begin{tabular}{c} an \\ $R$-tree 
        \end{tabular}}};
        \draw [red] (1,1.7) circle [radius=0.4];
        \node at (1,1.7) {\tiny \textcolor{red}{\begin{tabular}{c} an \\ $R$-tree 
        \end{tabular}}};
        \draw (-0.275,0.275) to (-1,1) to (-1,1.3);
        \draw (0,0.4) to (0,1) to (0,1.6);
        \draw (0.275,0.275) to (1,1) to (1,1.3);
        \node at (1.3,0.5) {$\cdots$}; 

        \node at (3.8,1) {$\longrightarrow$}; 

        \draw [fill] (6,1) circle [radius=0.02]; 
        \node [right] at (6,1) {\tiny W};
    \end{tikzpicture}
\end{center}
     However, one might be surprised that our definition of a reduced bimodule also implies maps of the form
    \[B  \rightarrow W,\]
    \[ R \rightarrow W,\]
    \[B \circ R \rightarrow W\]
    obtained by considering label changes from blue to white, red to white, and the collapse of adjacent blue and red vertices to a white vertex. The existence of these maps is a consequence of the hypothesis that $W(1)=\ast$. In general, to define a bimodule in terms of a category of three-colored trees, one must include $W$ colored vertices which have $|\overline{cld}(W)|=1$ to encode the composites involving $W(1)$. However, when $W(1)=\ast$ these composites are equivalent to the data of associative left and right actions:
    \[B \cong B \circ 1 \rightarrow B \circ W \rightarrow W,\]
    \[R \cong 1 \circ R \rightarrow W \circ R \rightarrow W,\]
    which commute with one another:
  \[B \circ R \cong B \circ 1 \circ R \rightarrow  B \circ W \circ R \rightarrow W.\]
  The diligent reader who has been worrying about infinitesimal left actions \cite[Definition 3.6]{aroneTurchin} versus actual left actions should now rest easy since a reduced left module is automatically an infinitesimal left module.
\end{remark}

\begin{definition}
    If $W$ is an $n$-truncated reduced $(B,R)$-bimodule and $|S| \leq n$, then the reduced bimodule decomposables are given by
    \[\mathrm{Decom}(W)(S) := \bigcup_{T \in \mathcal{T}^{RBW}(S) \setminus \{\bullet_{W}, \bullet_{R},\bullet_{B}\}} \mathrm{image}(\phi(T) \rightarrow \phi(\bullet_{W})).\]
\end{definition}
If we considered the actual bimodule decomposables rather than ``reduced bimodule decomposables'', one would also include the images of $B(S)$ and $R(S)$ in $W(S)$, as they correspond to the composites involving $W(1)$. It is easy to reformulate the coming definitions in terms of the unreduced decomposables if one so desires. We will simultaneously be considering operad decomposables\footnote{Our definition of operad decomposables is inherently reduced.} and reduced bimodule decomposables. Our notation for these is the same, and they are distinguished from each other implicitly by whether we evaluate on an operad or on a reduced bimodule.

\subsection{Bimodule cobordisms}\label{section: bimodule cobordisms}
In order to study cobordisms between manifold operads, we must clarify what we mean by a cobordism between two manifolds with boundary.

\begin{definition}\label{RBWmfld_dfn}
    An \textit{$RBW$-manifold} $(W,B,R)$ is a compact topological manifold with boundary $W$ with disjoint chosen subspaces $B,R\subset\partial W$. We require of the triple $(W,B,R)$ that $B,R,$ and $\partial W  \setminus (\mathring{R} \sqcup \mathring{B})$ are compact codimension $0$ submanifolds of $\partial W$ such that $\partial W  \setminus (\mathring{R} \sqcup \mathring{B})$ is a cobordism between $\partial R$ and $\partial B$.

\end{definition}

\[
\begin{tikzpicture}
\draw (0,0) coordinate (A) -- (1,0) coordinate (B) -- (1,1) coordinate (D) -- (0,1) coordinate (C) -- (A) -- (0.6,0.3) coordinate (E) -- (1.6,0.3) coordinate (F) -- (B);
\draw (C) -- (0.6,1.3) coordinate (G) -- (1.6,1.3) coordinate (H) -- (D);
\draw (E) -- (G); 
\draw (F) -- (H); 
\fill[red, opacity=0.5] (C) -- (D) -- (H) -- (G) -- (C) -- cycle; 
\draw [red, opacity=0.5] (C) -- (D) -- (H) -- (G) -- cycle;
\fill[blue, opacity=0.5] (A) -- (B) -- (F) -- (F) -- (E) -- (A) -- cycle; 
\draw [blue, opacity=0.5] (A) -- (B) -- (F) -- (E) -- cycle; 
\end{tikzpicture}
\]

Given a finite group $G$, an \textit{$RBW$ $G$-manifold} is an $RBW$-manifold $(W,B,R)$ with an action of $G$ on $W$ such that $R$ and $B$ are preserved by the action.

\begin{definition}\label{bimodulecobordism_dfn}
    Let $n \leq \infty$ and suppose $R,B$ are $n$-truncated manifold operads of dimension $d$. An $n$-truncated {\it bimodule cobordism} $W$ from $B$ to $R$ is an $n$-truncated reduced $(B,R)$-bimodule $W$ which for all $S$ with cardinality $m\le n$ satisfies 
    \begin{enumerate}
    \item \label{bimodcobor1} The canonical maps $B(S) \rightarrow W(S)$ and $R(S) \rightarrow W(S)$ are embeddings.
        \item \label{bimodcobor2} The triple $\big(W(S),B(S),R(S)\big)$ forms a $(md-d)$-dimensional $RBW$ $\Sigma_S$-manifold.
        \item \label{bimodcobor3} The boundary is given by $\partial W(S) = B(S) \cup \mathrm{Decom}(W)(S) \cup R(S)$.
        \item \label{bimodcobor4} The canonical map $$\colim_{{\mathcal{T}^{RBW}(S)\backslash\bullet_{ W }}}\phi \longrightarrow B(S) \cup \mathrm{Decom}(W)(S) \cup R(S) $$ is a homeomorphism.

    \end{enumerate}
\end{definition}

In other words, a bimodule cobordism contains the data of \begin{itemize}
    \item A symmetric sequence of nullbordisms of $\mathrm{Decom}(R)$ given by $R$ itself.
     \item A symmetric sequence of nullbordisms of $\mathrm{Decom}(B)$ given by $B$ itself.
     \item A symmetric sequence of cobordisms between the closed manifolds $\mathrm{Decom}(B)$ and $\mathrm{Decom}(R)$ given by $\mathrm{Decom}(W)$.
     \item A symmetric sequence of nullbordisms of the closed manifolds \[\mathrm{Decom}(B) \cup \mathrm{Decom}(W) \cup \mathrm{Decom}(R).\]
\end{itemize}

This definition implies, by  Lemma \ref{colimtopology_lmm} and Lemma \ref{operadtopology_lmm}, that for every $T\in \mathcal{T}^{RBW}(S)\backslash\bullet_W$, the obvious map $\phi(T)\to \colim_{\mathcal{T}^{RBW}(S)\backslash\bullet_W}\phi$ (therefore also the map $\phi(T)\to W(S)$) is a closed embedding, and the images of these maps only intersect ``when necessary''.

Given a bimodule cobordism $W$, we will sometimes write $\partial_L W$ for the manifold operad which acts on the left and $\partial_R W$ for the manifold operad which acts on the right. We say that $W$ is a bimodule cobordism from $\partial_L W$ to $\partial_R W$. As before, for $C\in\{R,B,W\}$, we write $C(m):=C(\{1,\ldots,m\})$ and $\mathcal{T}^{RBW}(m):=\mathcal{T}^{RBW}(\{1,\ldots,m\})$. 

\begin{definition}
    We say that an $n$-truncated bimodule cobordism $W$ is collarable if each $W(m)$ for $m \leq n$ admits a $\Sigma_m$-equivariant collar. 
\end{definition}

We emphasize that a bimodule cobordism from $B$ to $R$ being collarable does not imply $B$ and $R$ are collarable.

\begin{prop}\label{prop: thick W bimodule}
    Suppose $O$ is an $n$-truncated manifold operad such that $W(O)$ is also an $n$-truncated manifold operad, e.g. this is the case if $O$ is collarable. The thickened $W$-construction $W_{[0,1]}(O)$ is an $n$-truncated collarable $(O,W(O))$-bimodule cobordism where the bimodule structure is induced from embedding  $O$ as the suboperad given by the corollas with root hair of length $0$ and embedding $W(O)$ as the suboperad of all trees with root length $1$.

\end{prop}

\begin{proof}
    By definition and hypothesis, the conditions $(\ref{bimodcobor1})$ and $(\ref{bimodcobor2})$ are satisfied. Condition $(\ref{bimodcobor3})$ is satisfied because the boundary of $W_{[0,1]}(O)(m)$ is \[(O(m) \times \{0\})\cup (W(O)(m) \times \{1\}) \cup ((W(O)(m) \setminus \mathring{O}(m)) \times \{0\}) \cup (\mathrm{Decom}(W(O))(m) \times [0,1]) \]
where the first two factors correspond to the embedded operads, the third corresponds to the left $O$-module decomposables, and the fourth corresponds to the right $W(O)$-module decomposables. Condition $(\ref{bimodcobor4})$ holds by construction.

An equivariant collar of the thickened $W$-construction is given by pushing the length of the root hair towards one half and uniformly scaling down the lengths of the internal edges.
\end{proof}

   \textbf{Warning.} Although the operad embeddings also make $W_{[0,1]}(O)$ into a $(W(O),O)$-bimodule, this bimodule structure does not make it into a $(W(O),O)$-bimodule cobordism as both conditions $(3)$ and $(4)$ fail.
   
\begin{cor}\label{cor: reflexivity}
    If $n \leq \infty$ and $O$ is an $n$-truncated collarable manifold operad, then there is an $n$-truncated collarable bimodule cobordism from $O$ to $O$ which is homeomorphic as a symmetric sequence to $O \times [0,1]$.
\end{cor}

\begin{proof}
    By Proposition \ref{prop: isomorphic to w construction}, $O\cong W(O)$. Hence, the above bimodule cobordism $W_{[0,1]}(O)$ is an example.
\end{proof}

Let $n \leq \infty$ and denote by $1$ the unique reduced $n$-truncated operad which is empty in every degree $2 \leq m \leq n$. We call it the \textit{trivial $n$-truncated operad}, and it is an $n$-truncated manifold operad of every dimension.

\begin{definition}
    We say that an $n$-truncated manifold operad $O$ is left (right) nullbordant if it supplies the left (right) action of an $n$-truncated bimodule cobordism with $1$.
\end{definition}

\begin{example}\label{example: bordism between empty}
    The collection of nullbordisms (either left or right) of the operad $1$ as a manifold operad of dimension $d$ is equivalent to the collection of symmetric sequences of closed manifolds for which the $m$th space has dimension $md-d$.
\end{example}

\begin{prop} \label{leftnullbordant}
    The Fulton-MacPherson operad $\mathrm{FM}_d$ is left nullbordant through a collarable bimodule cobordism. The bimodule cobordism is equivalent as a bimodule to the commutative operad $\mathrm{com} = \{\ast \}_{m \geq 2}$ with left action induced from the unique operad map $\mathrm{FM}_d \rightarrow \mathrm{com}$.
\end{prop}
\begin{proof}
    The data of a $(\mathrm{FM}_d,1)$-bimodule is simply that of a left $\mathrm{FM}_d$-module, i.e. it is specified by its restriction to trees with no red vertices. We construct this left module $L$ as a quotient of the left $\mathrm{FM}_d$-module $W_{[0,1]}(\mathrm{FM}_d)$. As before, the left action will be given by grafting trees onto a corolla with length $0$ root hair.

   The equivalence relation $\sim$ that we impose on $W_{[0,1]}(\mathrm{FM}_d)$ is defined for any operad $O$: $\sim$ identifies any two metric trees with labels in $\mathrm{FM}_d$ which agree after collapsing all, not necessarily proper and not necessarily having the same root, subtrees which have their root hair of length $1$. We define $L:=W_{[0,1]}(\mathrm{FM}_d)/\sim $ and observe it is still a left $\mathrm{FM}_d$-module. It remains to show properties $(\ref{bimodcobor1})-(\ref{bimodcobor4})$ of the definition of a bimodule cobordism, as well as the existence of an equivariant collar. Property $(\ref{bimodcobor1})$ is obviously true.
   
    By Proposition \ref{prop: isomorphic to w construction}, the space $W_{[0,1]}(\mathrm{FM}_d)(m)$ is homeomorphic to $\mathrm{FM}_d(m) \times [0,1]$.\footnote{This case of the proposition was originally proven by the third author in \cite{salvatore_2021}.} The latter maps to the subspace of $(\mathbb{R}^d)^m/\mathrm{translation}$ where the maximal distance between any two points in a configuration is at most $1$ in two steps:

    \begin{itemize}
        \item  First, collapse any infinitesimal configurations to get a tuple in $(\mathbb{R}^d)^m/\mathrm{translation,scaling}$;
        \item Second, assert that the maximum distance between any two points is exactly $1-t$ where $t$ is the interval coordinate of $\mathrm{FM}_d(m) \times [0,1]$.
    \end{itemize}
    Note that if $t=1$, every tree will be sent to the unique tuple where all points coincide.  
    
    This map is surjective, and the codomain is a compact manifold with boundary: a disc.
    Hence, if we identify points in the domain which have equal image under this map, the result will be a compact manifold with boundary homeomorphic to the codomain. Observe that two elements \[((x_1,\dots,x_m),t),((x'_1,\dots,x'_m),t') \in \mathrm{FM}_d(m) \times [0,1]\] have the same image exactly when they agree after collapsing the infinitesimal configurations or  when the interval coordinate is $1$. 
    
    Under the operad identification
$W_{[0,1]}(\mathrm{FM}_d) \cong \mathrm{FM}_d \times [0,1] $ this equivalence relation exactly corresponds to $\sim$ and implies $(\ref{bimodcobor2})$. Under our identification, the boundary of $L$ corresponds to the tuples of points which have a maximum distance of $1$ between them. In $L$, this corresponds to trees which have root hair length $0$. This is the union of the embedded copy of $\mathrm{FM}_d$ and the reduced left module decomposables, and so property $(\ref{bimodcobor3})$ holds. Property $(\ref{bimodcobor4})$ holds because it holds before quotienting by $\sim$ by Proposition \ref{prop: thick W bimodule} and because $\sim$ is compatible with the colimit in question.

The existence of an equivariant collar is evident from the description of $L$ as tuples of points with maximum pairwise distance at most $1$. Finally, the claimed equivalence of $L$ with $com$, whose left action is given by the operad map $\mathrm{FM}_d \rightarrow \mathrm{com}$ is equivalent to the levelwise contractibility of the symmetric sequence $L$. This contraction is given by scaling the maximal pairwise distance to $0$.
\end{proof}

\begin{prop} \label{prop: rightnullbordant}
    The Fulton-MacPherson operad $\FM_d$ is right nullbordant through a collarable bimodule cobordism. 
\end{prop}

\begin{proof}
Let $G$ be a compact Lie group of dimension $d$, for example $G=(S^1)^d$. Consider the symmetric sequence $\FM_G$ consisting of compactified configuration spaces of points in $G$.  This forms a right module over $\FM_d$
\cite{markl}.
One can visualize elements of $\FM_G(m)$ as infinitesimal clusters of points (labeled by $m$) in $G$; we write $1 \leq k \leq m$ for the number of clusters. The action of $\FM_d$ replaces points by an infinitesimal cluster of points. 
The space $\FM_G(m)$ is a smooth manifold with corners of dimension $md$. Its interior is the configuration space $\mathrm{Conf}(G,m)$ ($k=m$), and its boundary are the configurations with an infinitesimal component $(k <m)$. From our description, it is clear these are the (nonreduced) right module decomposables.
Consider the free action of $G$ on $\FM_G(m)$ induced by the multiplication of $G$. The quotient $W(m)=\FM_G(m)/G$ is a compact manifold of dimension $md-d$. For instance, if $G=S^1$, then $W(m)$ is the disjoint union of $(m-1)!$ copies of a $(m-1)$-dimensional polyhedron called cyclohedron.
The symmetric sequence $W$ inherits a right module structure over $\FM_d$ whose boundary $\partial W(m)$  is exactly the (nonreduced) right module decomposable elements. In particular, $\partial W(m)$ is the union of a copy of $\FM_d(m)$, corresponding to $k=1$, and the reduced right module decomposables $\mathrm{Decom}(W)(m)$, corresponding to $1 < k <m$. This demonstrates that $W$ is a right nullbordism of $\mathrm{FM}_d$.
It is collarable since the symmetric group $\Sigma_m$ acts freely on  $W(m)$.
\end{proof}

We have just demonstrated, perhaps surprisingly, that the manifold operad $\mathrm{FM}_d$ is both right and left nullbordant. A priori, there should be no implication between the existence of left and right nullbordisms. Remarkably, right nullbordism of a manifold operad does formally imply left nullbordism. Before we state the result, let us first note that both left and right nullbordant manifold operads are plentiful. The constructions Theorem \ref{thm: left surgery} and Theorem \ref{thm: right surgery} guarantee that every $\Sigma_2$-nullbordism appears as the $2$-ary space of a right nullbordism and of a left nullbordism, respectively, by applying them to Example \ref{example: bordism between empty}. 

\begin{thm} \label{thm: swap}
If a collarable $n$-truncated manifold operad $P$ is right nullbordant through a collarable $n$-truncated bimodule cobordism $Y$, then it is canonically left nullbordant.
\end{thm} 

We provide a sketch of the proof which we will revisit in Section \ref{section: reverse} once we have developed the basic theory of stratifications. Before doing so, we introduce a definition:

\begin{definition}
Let $\mathcal{T}_\mathrm{local}^{RW}(S)$ be the category 
with objects $S$-labeled trees with vertices of color $R$ or $W$, with no requirement of legality as for $RBW$-trees, and with morphisms tree contractions $f: T\to T'$ such that for each vertex $v$ of $T'$, $f^{-1}(v)\to v$ is a $RBW$-contraction of $\overline{cld}(v)$-labeled trees. 
\end{definition}

\begin{proof}
Since $Y$ is an $n$-truncated right $P$-module,
for $|S|\leq n$ the hypothesis defines a functor $F_S:\mathcal{T}_\mathrm{local}^{RW}(S) \to \mathrm{Top}$. Given a tree $T$, 
$F_S(T)$ is the product of a copy 
of $P(\overline{cld}(v))$ for each red vertex $v$ of $T$, and of a copy of 
$Y(\overline{cld}(v))$ for each white vertex of $T$. The value of $F_S$ on morphisms is determined by the operad structure of $P$ and the right $P$-module structure of $Y$.
Let us define $Z(S) := \colim  F_S$. Any point of $Z(S)$ is represented (non-uniquely) by a $S$-labeled tree, with the vertices labeled either by $P$ or $Y$. 
We claim that $Z$ is an $n$-truncated $(P,1)$-bimodule cobordism. 
First of all we observe that $Z$
is an $n$-truncated left $P$-module. The action of $P$ is defined by gluing trees together along a new root vertex, labeled by $P$. In particular the elements of $\mathrm{Decom}(Z)$ are represented by trees with a root vertex labeled by $P$, that are not corollas. To conclude we need to prove that 
$Z(S)$ is a compact manifold with boundary $\partial Z(S)=
\mathrm{Decom}(Z)(S) \cup P(S)$ of dimension
$|S|d-d$, where $d$ is the dimension of the manifold operad $P$. Clearly $Z(S)$ is compact as it is a finite union of products of compact manifolds.
Any point $z \in Z(S)$ can be written uniquely 
as a $S$-labeled tree such that each vertex $v$ of red (resp. white) color
is labeled by an element in the interior of 
the manifold $P(\overline{cld}(v))$ (resp.
$Y(\overline{cld}(v))$ ). The maximal trees are those having only white vertices. A combinatorial count shows that the points represented by a given maximal tree form a manifold of dimension $|S|d-d$. 
More generally the points represented by a given tree form a manifold of dimension $|S|d-d-r$ where $r$ is the number of red vertices.
A neighborhood of a given point $z$ contains trees of the same shape, or trees obtained by iterated applications of one of the following two moves: replace a red vertex by a white vertex, or collapse an edge coming from a red vertex to its parent into a single vertex, with the same color of the parent.
 This shows that a neighborhood of a given point $z$ in $Z(S)$ is the product of an open disc of dimension  $|S|d-d-r$ by $r$ open intervals, if the root vertex is white, and the product of an open disc of dimension $|S|d-d-r$ by a copy of $[0,1)$ and  $r-1$ open intervals, if the root vertex is red.
\end{proof}

\begin{ex}
If we apply Theorem \ref{thm: swap} to the operad $FM_1$ and its right nullbordism as constructed in Proposition \ref{prop: rightnullbordant} with $G=S^1$, then the resulting left nullbordism $Z$ is not isomorphic to the left nullbordism $L$ described in Proposition \ref{leftnullbordant}. In fact $L(3)$ is a 2-dimensional disc, whereas  $Z(3)$ is a genus one surface with one boundary component, 
obtained by gluing together two hexagons and three squares. 
\end{ex}

Bimodule cobordisms allow one to formulate the theory of secondary obstructions to the extension of truncated manifold operads which measure the uniqueness of the extension.

\begin{prop}
Suppose $B,R$ are $d$-dimensional, collarable $(n+1)$-truncated manifold operads, and that $W$ is an $n$-truncated collarable bimodule cobordism from $B^{\leq n}$ to $R^{\leq n}.$ There is a canonical obstruction in $\Omega_{nd-1}^{\Sigma_{n+1}\mathrm{Top}}$ which vanishes if and only if $W$ extends to a collarable $(n+1)$-truncated bimodule cobordism from $B$ to $R$.
\end{prop}
\begin{proof}
    Let $W$ be represented by $\phi:\mathcal{T}_{\leq n}^{RBW} \rightarrow \mathrm{Top}$. The obstruction is given by 
    \[[\colim_{{\mathcal{T}^{RBW}(n+1)\backslash\bullet_{ W }}}\phi]  \in \Omega_{nd-1}^{\Sigma_{n+1}\mathrm{Top}}.\]
  The space has a canonical $\Sigma_{n+1}$-action, and the fact that it is a manifold follows from Theorem \ref{thm: bimodule cobordism is stratified} combined with Proposition \ref{prop: bimodule decomposables are manifold}.  Checking the dimension of a maximal face of the colimit, corresponding to a tree with a white root and red vertex, one sees that the dimension must be
    \[\mathrm{dim}\big(W(n+2-j) \times R (j)\big) = ((n+2-j)d -d) + (jd -d -1) = nd +2d -jd -d  +jd-d-1=nd-1. \]
    
    The argument for the completeness of this obstruction is identical to that of Proposition \ref{prop: operad obstruction}.
\end{proof}

As for manifold operads, it is easy to produce examples of $n$-truncated bimodule cobordisms for fixed $n<\infty$, for instance by surgery on the interior of $W_{[0,1]}(O)^{\leq n}(n)$ or on the interior of the $n$-ary operations of $O^{\leq n}(n)$.

We end the section with a brief discussion on the application of bimodule cobordisms to the determination of when two manifold operads are weakly homotopy equivalent. Let $n\leq \infty$ and recall an $n$-truncated operad $O$ is weakly homotopy equivalent to $P$ if there is a zigzag of $n$-truncated operad maps between $O$ and $P$ all of which induce isomorphisms on homotopy groups for any choice of basepoint. 

\begin{definition}

If $H$ is an $n$-truncated bimodule between reduced $n$-truncated operads, we say it is homotopically trivial if each map in the cospan
    \[B \hookrightarrow H \hookleftarrow R\]
    is a weak homotopy equivalence of $n$-truncated symmetric sequences.
\end{definition}
\begin{definition}
    An $n$-truncated bimodule $h$-cobordism $H$ from an $n$-truncated manifold operad $B$ to an $n$-truncated manifold operad $R$ is a homotopically trivial $n$-truncated bimodule cobordism.
\end{definition}

\begin{prop}
    If $H$ is a homotopically trivial bimodule from an $n$-truncated operad $B$ to an $n$-truncated operad $R$, then $B$ is weakly equivalent to $R$ as an $n$-truncated operad.
\end{prop}

\begin{proof}    
    We prove this when $n=\infty$, and the finitely truncated case is analogous. Recall that to put a $(B,R)$-bimodule structure on $H$ is equivalent to producing a $B$-algebra structure on $H$ in the category of right $R$-modules with the Day convolution product $\circledast$ \cite[9.1.2 Proposition]{Fresse2009}. This yields maps of operads
    \[B \rightarrow \mathrm{End}(H) \rightarrow \mathrm{End}^h(H),\]
    where $\mathrm{End}(H) \rightarrow \mathrm{End}^h(H)$ denotes the map from the endomorphism operad to the derived endomorphism operad of $H$, both computed in the category of right modules with Day convolution, given by applying a topological cofibrant replacement functor. Since $R \hookrightarrow H$ is a weak homotopy equivalence of right modules, in the homotopy category this composite is equivalent to a map
    \[B \rightarrow \mathrm{End}^h(R). \]
There are equivalences of operads \[\mathrm{End}^h(R)\simeq \mathrm{End}(R)\simeq R\]
    using that $R$ is a free right module over itself and the Day convolution of free right modules is free, a consequence of the universal property of Day convolution. It remains to show that the map $B \rightarrow \mathrm{End}^h(R)$ just constructed is an isomorphism in the homotopy category of operads. It suffices to prove the following composition is an equivalence in the homotopy category of symmetric sequences
    \[B \rightarrow \mathrm{End}^h(H):= \{\mathrm{RMod}_R^h(H^{\circledast m},H)\}_{m \geq 2} \simeq \{\mathrm{RMod}_R^h(R ^{\circledast m},H)\}_{m \geq 2 }\simeq H .\]
    A diagram chase shows this map is homotopic to the original inclusion $B \xhookrightarrow{\simeq}H$ which is an equivalence by assumption.
\end{proof}

Since a bimodule $h$-cobordism is a particular instance of a homotopically trivial bimodule, we have the following corollary:

\begin{cor}\label{cor: h cobordism implies equiv}
    If an $n$-truncated manifold operad $B$ is bimodule $h$-cobordant to an $n$-truncated manifold operad $R$, then $B$ is weakly homotopy equivalent to $R$ as an $n$-truncated operad.
\end{cor}

As a consequence, to demonstrate that manifold operads $B$ and $R$ are weakly equivalent, it is sufficient to produce a bimodule cobordism $H$ from $B$ to $R$ and ``surger'' it until it becomes a bimodule $h$-cobordism.

\begin{remark} \label{remark: homology h cob}
    There is an obvious notion of a homology bimodule $h$-cobordism; a variation of the above argument shows that if $O$ is homology bimodule $h$-cobordant to $P$ then there is a chain of quasi-isomorphisms
    $C_\ast(O) \simeq \dots \simeq C_\ast (P)$
    as operads in the category of chain complexes or even differential graded $E_\infty$-coalgebras. 
\end{remark}

\section{Combinatorics of three-colored trees} \label{section: Combinatorics of three colored trees}

\subsection{Elementary properties of three-colored trees}
In Section \ref{subsection: bimodules and trees}, we introduced the category of RBW-trees in order to describe reduced bimodules over reduced operads. In this section, we produce some elementary lemmas regarding the manipulation of such trees. Note that most of these lemmas can be applied directly to single-colored trees, as they embed into RBW-trees by choosing the vertices to be either entirely red or entirely blue.

For the rest of this section, all trees and contractions are assumed to be RBW-trees and contractions.

\begin{definition}
    If there exist contractions $T\to T'$ and $T'\to T''$, then $T\to T'$ is an \textit{intermediate contraction} of $T\to T''$. 
\end{definition}

\begin{lem}\label{elementarycontraction_lmm}
    Every contraction can be written (nonuniquely) as a composition of the following \textit{elementary contractions}.\footnote{By such a picture we mean that all other vertices and edges in the tree remain unaffected.} The elementary contractions $(2),(3),(5),$ are always legal contractions, though the legality of $(1),(4),(6)$ depends on a neighborhood of the vertex.
    \begin{center}
    \begin{tikzpicture}
        \draw [fill] (0,0) circle [radius=0.02]; 
        \node [below] at (0,0) {\tiny $R$};
        \draw [->] (0.4,0) to (1,0); 
        \draw [fill] (1.3,0) circle [radius=0.02]; 
        \node [below] at (1.3,0) {\tiny $ W $};
        \node at (0.8,-1) {\tiny (1) turn the color of one vertex from $R$ to $ W $};
    \end{tikzpicture}
    \hspace{2cm}
    \begin{tikzpicture}
        \draw [fill] (0,0) circle [radius=0.02]; 
        \node [right] at (0,0) {\tiny $R$};
        \draw [fill] (0,1) circle [radius=0.02];
        \node [right] at (0,1) {\tiny $R$};
        \draw (0,0) to (0,1); 
        \draw [->] (0.4,0.5) to (1,0.5); 
        \draw [fill] (1.3,0.5) circle [radius=0.02]; 
        \node [below] at (1.3,0.5) {\tiny $R$};
        \node at (0.8,-1) {\tiny (2) contract an $(R-R)$ edge to an $R$-vertex};
    \end{tikzpicture}

    \begin{tikzpicture}
        \draw [fill] (0,0) circle [radius=0.02]; 
        \node [right] at (0,0) {\tiny $ W $};
        \draw [fill] (0,1) circle [radius=0.02];
        \node [right] at (0,1) {\tiny $R$};
        \draw (0,0) to (0,1); 
        \draw [->] (0.4,0.5) to (1,0.5); 
        \draw [fill] (1.3,0.5) circle [radius=0.02]; 
        \node [below] at (1.3,0.5) {\tiny $ W $};
        \node at (0.8,-1) {\tiny (3) contract an $(R- W )$ edge to a $ W $-vertex};
    \end{tikzpicture}
    \hspace{2cm}
    \begin{tikzpicture}
        \draw [fill] (0,0) circle [radius=0.02]; 
        \node [below] at (0,0) {\tiny $B$};
        \draw [->] (0.4,0) to (1,0); 
        \draw [fill] (1.3,0) circle [radius=0.02]; 
        \node [below] at (1.3,0) {\tiny $ W $};
        \node at (0.8,-1) {\tiny (4) turn the color of one vertex from $B$ to $ W $};
    \end{tikzpicture}
    \hspace{2.5cm}
    \begin{tikzpicture}
        \draw [fill] (0,0) circle [radius=0.02]; 
        \node [right] at (0,0) {\tiny $B$};
        \draw [fill] (0,1) circle [radius=0.02];
        \node [right] at (0,1) {\tiny $B$};
        \draw (0,0) to (0,1); 
        \draw [->] (0.4,0.5) to (1,0.5); 
        \draw [fill] (1.3,0.5) circle [radius=0.02]; 
        \node [below] at (1.3,0.5) {\tiny $B$};
        \node at (0.8,-1) {\tiny (5) contract a $(B-B)$ edge to a $B$-vertex};
    \end{tikzpicture}
    \hspace{2.5cm}
    \begin{tikzpicture}
        \draw [fill] (0,0) circle [radius=0.02]; 
        \node [below] at (0,0) {\tiny $B$};
        \draw [fill] (-0.8,1) circle [radius=0.02]; 
        \node [above] at (-0.8,1) {\tiny $ W $}; 
        \draw (-0.8,1) to (0,0); 
        \draw [fill] (-0.5,1) circle [radius=0.02]; 
        \node [above] at (-0.5,1) {\tiny $ W $}; 
        \draw (-0.5,1) to (0,0); 
        \node at (0,1) {$\ldots$};
        \draw [fill] (0.5,1) circle [radius=0.02]; 
        \node [above] at (0.5,1) {\tiny $ W $}; 
        \draw (0.5,1) to (0,0); 
        \draw [fill] (0.8,1) circle [radius=0.02]; 
        \node [above] at (0.8,1) {\tiny $ W $}; 
        \draw (0.8,1) to (0,0); 
        \draw [->] (0.8,0.5) to (1.2,0.5); 
        \draw [fill] (1.5,0.5) circle [radius=0.02]; 
        \node [below] at (1.5,0.5) {\tiny $ W $};
        \node at (0.8,-1) {\tiny \begin{tabular}{c} (6) contract a subtree of the form\\ ``all of the $ W $-vertices attached to one $B$-vertex''\\ to a $ W $-vertex\end{tabular}};
    \end{tikzpicture}
    \end{center}
\end{lem}
\begin{proof}
    Given a contraction $\mathfrak{c}:T\to T'$, for each vertex $v\in V(T')$, $\mathfrak{c}^{-1}(v)$ is a subtree of $T$, and $\mathfrak{c}$ can be decomposed by contracting these subtrees one at a time. This reduces the lemma to showing a contraction from $T$ to a corolla can decompose into the elementary contractions above. 

    A contraction of the form $T\to \bullet_{R \ (\textnormal{resp. }B)}$ can be decomposed into elementary contractions of form (2) (resp. (5)), by contracting one arbitrary edge at a time, since necessarily all vertices are red (resp. blue).

    It remains to show that a contraction of form $T\to \bullet_{ W }$ can also be decomposed into elementary contractions. To do this, we perform the following algorithm on a tree $T$:  

    Step 1: If $T$ has only one vertex, then: 
    \begin{itemize}
        \item if it is of color $R$, then perform (1) and terminate the algorithm;
        \item if it is of color $B$, then perform (4) and terminate the algorithm; 
        \item if it is of color $ W $, then terminate the algorithm. 
    \end{itemize}
    If $T$ has more than one vertex, go to Step 2.
    
    Step 2: Pick an arbitrary leaf vertex $w$ of $T$ that is not of color $ W $; go to Step 3. If all leaf vertices are of color $ W $, then go to step 4. 

    Step 3: 
    \begin{itemize}
        \item If $w$ is of color $R$ and its parent is of color $R$, then perform (2) to $(w,\textnormal{parent}(w))$; 
        \item If $w$ is of color $R$ and its parent is of color $ W $, then perform (3) to $(w,\textnormal{parent}(w))$; 
        \item If $w$ is of color $R$ and its parent is of color $B$, then perform (1) to $w$;
        \item If $w$ is of color $B$, then its parent must be of color $B$ as well; perform (5) to $(w,\textnormal{parent}(w))$. 
    \end{itemize}
    Go back to Step 1. 

    Step 4: All leaf vertices are of color $ W $, and so there is no $R$ vertex left. If there exists a vertex of color $B$, then choose a vertex $w$ of color $B$ such that no descendant of $w$ is also of color $B$, and perform (6) on $w$ and its children. Go back to Step 1. 

    By construction, the algorithm terminates when it has decomposed the contraction $T\rightarrow \bullet_W$ into elementary contractions. The algorithm will terminate in finite time since each of Steps 3 and 4 removes at least one vertex from the tree.

\end{proof}

In general, given a poset $P$ and elements $a \leq c$, we may form the half open interval $(a,c]$ which is the subposet of elements $b$ which satisfy $a < b \leq c$. Note that $(a,a]$ is always empty. We also denote by $(-,a]$ and $[a,-)$ the subposet of all elements $b$ satisfying $b\le a$ and $b\ge a$, respectively. 

\begin{definition}\label{definition: link of a tree}
    The poset $\mathcal{P}_\mathfrak{c}$ of nontrivial intermediate contractions associated to a contraction $\mathfrak{c}:T \rightarrow T'$ is $(T,T']$. We sometimes abbreviate $\mathcal{P}_T:=\mathcal{P}_{T\to\bullet_{ W }}$.
\end{definition}

It is helpful to think about the elements of $\mathcal{P}_\mathfrak{c}$ as being the uniquely determined contractions $T \rightarrow T''$, rather than simply the tree $T''$, and we will refer to them in this way.

\begin{definition}
    For a colored, labeled tree $T$, we define 
    $$\text{Codim}(T)=\text{the number of vertices of }T\text{ of color }R\text{ or }B.$$ For a contraction $\mathfrak{c}:T\to T'$, define 
    $$\text{Codim}(\mathfrak{c})=\text{Codim}(T)-\text{Codim}(T').$$
\end{definition}

It follows immediately from the definition that for any contraction $\mathfrak{c}$, $\textnormal{codim}(\mathfrak{c})\ge 0$ with equality  only if $\mathfrak{c}$ is the trivial contraction.

Recall that a \textit{chain} $C$ in a poset $P$ is a subset of $P$ such that every two elements in $C$ are comparable. The \textit{length} of a chain $C$ is $|C|-1$. We say a chain $C$ \textit{ends} at an element $x\in C$ if $x$ is maximal in $C$.  Note that the length of the empty chain is $-1$.

\begin{lem}
    Suppose $T\to T'$ is a contraction. Then the maximal length of a chain in $\mathcal{P}_{T\to T'}$ is $\textnormal{codim}(T\to T')-1$.
\end{lem}
\begin{proof}
   Since each elementary contraction (as in Lemma \ref{elementarycontraction_lmm}) has codimension 1 and $T\to T'$ can be decomposed into elementary contractions, we have a chain in $\mathcal{P}_{T\to T'}$ of length $\textnormal{codim}(T\to T')-1$. This chain is necessarily maximal because otherwise $T'$ would have too few vertices of color $R$ and $B$.
\end{proof}

    \begin{definition}
    A \textit{system of subtrees with colors} of $T$ is  
    \begin{itemize}
        \item a collection $\{T_i\}_{i\in I}$ of non-empty pairwise disjoint subtrees of $T$ which cover $T$,
        \item a color $\eta_i\in\{R,B, W \}$ assigned to each $T_i$, such that if $T_i$ consists of a single vertex $v$, then $\eta_i\neq\textnormal{Color}(v)$. 
    \end{itemize}
    We denote such a system of subtrees with colors as a set of pairs $\Theta=\{(T_i,\eta_i)\}_{i\in I}$, and denote by $T/\Theta$ the tree obtained by contracting each subtree $T_i$ and coloring the new vertex $\eta_i$ (although this coloring may be illegal). 
\end{definition}
\begin{definition}\label{contractionsystem_dfn}
    A \textit{contraction system} of $T$ is a system of subtrees with colors $\Theta=\{(T_i,\eta_i)\}_{i\in I}$ of $T$ such that each $T_i \rightarrow \bullet_{\eta_i}$ is a legal contraction and $T/\Theta$ is a legal RBW-tree.
\end{definition}

By definition, if $\Theta$ is a contraction system of $T$, then $T \rightarrow T/\Theta$ is a legal contraction.

\begin{definition}
    The poset $\mathcal{P}_T^\mathrm{sys}$  of contraction systems on $T$ is the poset with objects the contraction systems on $T$, and
    \[(T_i,\eta_i) \leq (T'_j,\eta'_j)\]
    if for each $T_i$ there is some $T'_j$ that contains it as a subtree  and if $\eta'_j$ is red then $\eta_i$ is red and if $\eta'_j$ is blue $\eta_i$ is blue. For a contraction system $\Theta$ of $T$ we let $\mathcal{P}^\mathrm{sys}_\Theta$ denote $(-,\Theta]$.
\end{definition}

If $T$ is a tree with all red (blue resp.) vertices, then there is a contraction system $\{(T,{R \textnormal{ resp. }B})\}$, and this is less than the contraction system $\{(T,{W})\}$. We denote these contraction systems $\Theta^T_{\bullet{R}},\Theta^T_{\bullet{B}},\Theta^T_{\bullet{ W }}$. Note $\mathcal{P}^\mathrm{sys}_{\Theta^T_{\bullet{ W }}}=\mathcal{P}^\mathrm{sys}_T$.

\begin{lem}\label{lem: iso between contractions and systems}
    Given a colored, labeled tree $T$, there is a poset isomorphism
    \[\mathcal{P}_T \cong \mathcal{P}_T^\mathrm{sys}.\]

  This poset isomorphism is defined by
    \begin{itemize}
        \item Given a contraction $\mathfrak{c}:T\to T'$, we associate to it the contraction system 
        $$\{\mathfrak{c}^{-1}(v),\textnormal{Color}(v)\}_{v \in V(T')}$$
 
        \item Given a contraction system $\Theta$ of $T$, we associate to it $T\to T/\Theta$.
    \end{itemize}
    
   Hence, for any contraction $\mathfrak{c}$ and associated contraction system $\Theta$ there is another poset isomorphism
    \[\mathcal{P}_\mathfrak{c} \cong \mathcal{P}_\Theta^\mathrm{sys}.\]

\end{lem}

\begin{proof}
    This follows from the definitions.
\end{proof}

Recall that the product of posets $\prod_{i=1}^n P_i$ has underlying set the cartesian product and 
\[(p_i) \leq (p'_i) \iff p_i \leq p'_i \quad \forall i \leq n.\]

\textbf{Notation:} for a poset $P$ we denote by $\underline{P}$ the poset obtained from $P$ by adding in a minimal element. Namely, $\underline{P}=\{0\}\sqcup P$ as a set, $0<p,\forall p\in P$, and the inequalities between elements of $P$ are unchanged.

\begin{prop}\label{prop: join of contraction systems}
    Suppose $\Theta=\{(T_1,\eta_1),\ldots,(T_k,\eta_k)\}$ is a contraction system of a tree $T$. There is a bijection of posets

    \[\prod^k_{i=1} \underline{\mathcal{P}}^\mathrm{sys}_{\Theta^{T_i}_{\bullet{\eta_i}}} \cong \underline{\mathcal{P}}^\mathrm{sys}_\Theta, \]
    \[(\theta_1,\dots,\theta_k) \mapsto \bigcup^k_{i =1} \theta_i.\]
\end{prop}
\begin{proof}
We first show that the function is well defined. Fix contraction systems $\theta_i$ of $T_i$ such that $\theta_i \leq \Theta^{T_i}_{\bullet{\eta_i}}$: $\theta_i$ is a contraction system which assigns each subtree of $T_i$ in $\theta_i$ the color $\eta_i$ if $\eta_i=R,B$ and has no restriction beyond legality if $\eta_i= W $. We must show that the union as $i$ varies yields a contraction system $\bigcup^k_{i =1} \theta_i$ of $T$.

The first condition is automatic: the contraction of each subtree in the contraction system to its specified color is a legal contraction because this condition already had to hold as contraction systems of the relevant $T_i$. It remains to see that the contracted tree associated to $\bigcup^k_{i =1} \theta_i$ is a legal coloring. One can verify legality of colorings by checking every adjacent pair of vertices $X,Y$ where $X$ is a descendant of $Y$. If $X,Y$ are the images of subtrees of the same $T_i$, then this adjacent pair is legal because $\theta_i$ is a contraction system of $T_i$. What remains are the vertices which are adjacent, but for which $X$ is in the image of $T_i$, $Y$ is in the image of $T_j$ and $i\neq j$.  If $T_i$ is assigned to $R$ by $\Theta$, then this edge coloring is legal. This is because necessarily $X$ will be assigned to $R$ in $\theta_i$
and any color before $R$ is legal. Similarly, if $T_j$ is assigned to $B$ by $\Theta$. However, any legal edge is of one of these forms, so $\bigcup^k_{i =1} \theta_i$ is a contraction system of $T$.

The definition of the poset structure on contraction systems implies that any contraction system $\Theta'$ of $T$ such that $\Theta' \leq \Theta$ naturally decomposes into a union of contraction systems $\Theta'_1,\dots, \Theta'_k$ where $\Theta'_i$ is a contraction system for $T_i$. Recording this tuple of contraction systems gives us a function
\[  \underline{\mathcal{P}}^\mathrm{sys}_\Theta \rightarrow \prod^k_{i=1} \underline{\mathcal{P}}^\mathrm{sys}_{\Theta^{T_i}_{\bullet{\eta_i}}}.\]
It is straightforward to see these functions are poset maps and inverses to each other.
\end{proof}

\subsection{Regular CW-complexes and posets}
Let $X_n$ denote the $n$-skeleton of a CW-complex.
\begin{definition}
    A \textit{regular CW-complex} is a CW-complex $X$ such that, for each $n>0$ and each $n$-dimensional cell $e^n$, the attaching map $t_{e^n}: S^{n-1}\to X_{n-1}$ is a homeomorphism onto its image, and the image of $t_{e^n}$ is a subcomplex of $X_{n-1}$. 
\end{definition}
\begin{remark}
    The standard definition of a regular CW-complex does not require the last condition. However, a space is homeomorphic to a regular CW-complex using one definition, if and only if it is homeomorphic to a regular CW-complex using the other \cite[Chapter III, Lemma 1.3, Theorem 2.1]{CWcplx}.
\end{remark}
\begin{definition}
    The \textit{face poset} $\mathcal{F}(X)$ of a regular CW-complex $X$ is the set of all the cells of $X$ with partial order $e_1\leq e_2\iff e_1\subset {e_2}$. 
\end{definition}

The following is well known and follows from the contractibility of the boundary-fixing homeomorphism group of the standard disc. 

\begin{lem}\label{CWuniqueness_lmm}
    Suppose $X,Y$ are regular CW-complexes. If $f:\mathcal{F}(X)\to\mathcal{F}(Y)$ is an isomorphism of posets, then there exists a CW-map $\hat{f}:X\to Y$ which is a homeomorphism, and $\hat{f}$ maps a cell $e$ to the cell $f(e)$. 
 \end{lem}

In other words, a regular CW-complex is determined by its face poset up to the parametrization of its cells (which is a contractible space of choices). If a poset $P$ is the face poset of a regular CW-complex, then the poset is called a ``CW-poset'' (see \cite{Bjorner}). 
\begin{remark}\label{remark: inductive cw poset}

The property that a poset $P$ is a CW-poset can be characterized inductively: $P$ is a CW-poset if for each cell $e$ of $P$, the subposet $\{e'\in P|e'<e\}$ of $P$ is a CW-poset, and its corresponding regular CW-complex is homeomorphic to a sphere whose dimension is the length of the longest chain in $P$ ending at $e$ minus 1. 
\end{remark}

Let $ \mathrm{CW}^\mathrm{reg}$ be the category of regular CW-complexes with morphisms given by inclusions of subcomplexes. Let $\mathrm{Poset}^\mathrm{CW}$ denote the subcategory of the category of posets with objects the CW-posets and morphisms the poset morphisms corresponding to inclusions of subcomplexes.
\begin{lem}\label{CWcolim_lmm}
    Suppose $I:D \rightarrow \mathrm{CW}^\mathrm{reg}$ is a diagram, then its colimit exists and the face poset of $\textnormal{colim}_DI$ is canonically isomorphic to the poset colimit $\textnormal{colim}_D(\mathcal{F}\circ I)$. 
\end{lem}

\begin{proof}
    Using the coequalizer formula for colimits, it suffices to check that both $\mathrm{CW}^\mathrm{reg}$ and $\mathrm{Poset}^\mathrm{CW}$ are closed under coproducts and coequalizers. Clearly, the coproduct (i.e. disjoint union) of regular CW-complexes is still regular and corresponds to the coproduct of the face posets. Almost as obvious is the fact that the coequalizer of two inclusions of regular CW-complexes is regular: simply glue the decompositions along the identified cells. On face posets this corresponds to the coequalizers of the associated poset morphism, so we are done.
\end{proof}

\begin{definition}\label{join_dfn}
    The \textit{join} $P_1*P_2*\ldots*P_n$ of posets $P_1,\ldots,P_n$ is the following poset: 
    \begin{itemize}
        \item the underlying set is $\underline{P_1}\times\ldots\times\underline{P_n}\backslash(0,\ldots,0)$, where the $i$-th 0 denotes the minimal elements of $P_i$; 
        \item the partial order is $(x_1,\ldots,x_n)\leq (y_1,\ldots,y_n)\iff\forall\ 1\le i\le n, x_i\leq y_i$. 
    \end{itemize}
Define the \textit{cone} of a poset $P$,   $\textnormal{Cone}(P)$ to be $P*\{\textnormal{pt}\}$, where $\{\textnormal{pt}\}$ is the poset with one element. 
\end{definition}

\begin{definition}\label{suspension_dfn}
    The \textit{suspension} $\Sigma P$ of a poset $P$ is the following poset: 
    \begin{itemize}
        \item the underlying set is $\textnormal{Cone}(P)\sqcup\{D\}$; 
        \item the partial order on $\textnormal{Cone}(P)$ is as above; we then add the relations $D>x,\forall\ x\in P
        \times\{0\}$. 
    \end{itemize}
\end{definition}

   \[
\begin{tikzpicture}

\coordinate (o) at (0,0);
\coordinate (a) at (-1,-1.2); 
\coordinate (b) at (-0.3,-1.5); 
\coordinate (c) at (1,-1.2); 
\coordinate (d) at (0.3,-0.9);

\fill[gray!25]
  (o)--(a)--(c)--cycle;

\fill[gray!25]
  (o)--(a)
  arc (-180:0:1)
  --(o)--cycle;

\foreach \p in {o,a,b,c,d}
  \draw[fill] (\p) circle [radius=0.02];

\draw (a)--(b)--(c);
\draw[gray] (c)--(d)--(a);
\draw (o)--(a); 
\draw (o)--(b);
\draw (o)--(c);
\draw[gray] (o)--(d);
\draw (a) arc (-180:0:1);

\node at (-2,-1) {$\Sigma X$:};
\node at (-2.5,-1.5) {};

\coordinate (a0) at (-7,-1.2); 
\coordinate (b0) at (-6.3,-1.5); 
\coordinate (c0) at (-5,-1.2); 
\coordinate (d0) at (-5.7,-0.9); 

\foreach \p in {a0,b0,c0,d0}
  \draw[fill] (\p) circle [radius=0.02];

\draw (a0)--(b0)--(c0)--(d0)--cycle;

\node at (-8,-1) {$X$:}; 
\node at (-8.2,-1.6) {};

\end{tikzpicture}
\]

Recall the join of two spaces $$X*Y:=X\times Y\times[0,1]\big/\{(x,y,0)\sim(x,y',0),(x,y,1)\sim (x',y,1),\forall x,x'\in X,y,y'\in Y\}.$$ The join of a finite number of finite CW-complexes is determined up to a natural homeomorphism by the fact the join is associative. The following lemma justifies the names ``join'' and ``suspension''. 

\begin{lem}\label{lem: join and suspension of CW}
\:

    \begin{enumerate}[label=(\arabic*)]
        \item Suppose $X_1,\ldots,X_k$ are finite regular CW-complexes, then $\mathcal{F}(X_1)*\ldots*\mathcal{F}(X_k)$ is the face poset of another regular CW-complex that is homeomorphic to $X_1*\ldots*X_k$. 
        \item Suppose $X$ is a regular CW-complex that is homeomorphic to $S^l$, then the unreduced suspension $\Sigma X$ has a regular CW-structure with face poset $\Sigma (\mathcal{F}(X))$. The CW-structure is the union of $\mathrm{Cone}(X)$ with the standard CW-structure induced from $X$ together with a single $(l+1)$-disc attached via $X \cong \mathrm{S}^l$.
    \end{enumerate}
\end{lem}

\begin{proof}
(1) It suffices to assume $k=2$ because both the join of posets and the join of spaces is associative.
As a set, the cells of $X \ast Y$ will be
$$\{\textnormal{cells in }X\}\sqcup\{\textnormal{cells in }Y\}\sqcup\{e_X*e_Y|e_X\textnormal{is a cell of }X\textnormal{ and }e_Y\textnormal{ is a cell of }Y\}.$$

We start with the CW-complex $X_1 \sqcup X_2$ with the disjoint union cell structure with cells $e^d_i$ and $f^d_i$, respectively. This is embedded into the join in the natural way. Start by attaching the joins of $0$-cells, which are intervals, by attaching the corresponding pair of cells in $X_1 \sqcup X_2$. The result is again embedded in the natural way and is still a regular CW-complex. We inductively attach cells in the following way:

 Since both CW-structures on $X_1$ and $X_2$ are regular, the join of the embedded cells $e^d_i \ast f^{d'}_{i'}$ necessarily embeds into $X_1 \ast X_2$, and the boundary must be embedded homeomorphically onto the union of the embedded joins of any pairs of cells $a \ast b$ where $a \subset e^d_i$ and $b \subset f^{d'}_{i'}$. By induction, the union of these cells is a sub CW-complex, and so we attach the new cell via this map. The result is still a regular CW-complex because the boundary is a union of lower dimensional cells.

 Once all cells are attached in this way, the result is a homeomorphism from $X_1 \ast X_2$ to a regular CW-complex whose face poset is $\mathcal{F}(X_1) \ast \mathcal{F}(X_2)$.

    (2) Add a single cell using the homeomorphism provided, corresponding to the maximal element of the CW-poset.
\end{proof}

\begin{warning}
The name ``suspension'' is justified only in the case the regular CW-complex associated to the regular CW-poset is a sphere. In any other case, Remark \ref{remark: inductive cw poset} shows the result cannot even be a CW-poset.
\end{warning}

\begin{lem}\label{link1_lmm}
    Let $T$ be an RBW-tree and $\Theta=\{(T_1,\eta_1),\ldots,(T_k,\eta_k)\}$ be a contraction system of $T$. Then 
    $$\mathcal{P}_{T\to T/\Theta}
    \cong
    \mathcal{P}_{T_1\to\bullet_{\eta_1}}*\ldots*\mathcal{P}_{T_k\to\bullet_{\eta_k}}.$$
\end{lem}
\begin{proof}
    This follows from Lemma \ref{lem: iso between contractions and systems} and Proposition \ref{prop: join of contraction systems}, and the definition of the join of posets.
\end{proof}

\subsection{Regular CW-complexes associated to three-colored trees}\label{link_sec}

The critical combinatorial result needed to develop the theory of surgery for manifold operads is:

\begin{prop}\label{link_prop}
    For all nontrivial contractions between colored, labeled trees $\mathfrak{c}:T\to T'$, $\mathcal{P}_{\mathfrak{c}}$ is the face poset of a regular CW-complex homeomorphic to a ball of dimension $\text{Codim}(\mathfrak{c})-1$.
\end{prop}

Before proving this proposition, we introduce some terminology. By Lemma \ref{CWuniqueness_lmm}, this regular CW-complex is unique up to CW-isomorphism, and we denote it by $Lk(\mathfrak{c})$.

Note $\mathcal{P}_\mathfrak{c}$ has one element $\mathfrak{c}$ that is bigger than all others, and so the assertion that its realization is homeomorphic to a ball is an immediate consequence of it being regular.  Moreover, since the maximal length of a chain in $\mathcal{P}_\mathfrak{c}$ is $\text{Codim}(\mathfrak{c})-1$, $Lk(\mathfrak{c})$ is necessarily a ball of dimension $\text{Codim}(\mathfrak{c})-1$. The boundary is a subcomplex homeomorphic to $S^{\text{Codim}(\mathfrak{c})-2}$, and we call this subcomplex $\partial Lk(\mathfrak{c})$. It is the regular CW-complex associated to the CW-poset $\partial\mathcal{P}_{\mathfrak{c}}$ defined to be the subposet of $\mathcal{P}_\mathfrak{c}$ with the maximal contraction $\mathfrak{c}$ removed.  When $\mathfrak{c}=(T\to \bullet_{ W })$ for some $T$, we may also use the abbreviations $Lk(T):=Lk(\mathfrak{c})$, $\partial Lk(T):=\partial Lk (\mathfrak{c})$, $\mathcal{P}_T:=\mathcal{P}_\mathfrak{c}$, $\partial\mathcal{P}_{T}:=\partial\mathcal{P}_{\mathfrak{c}}$. 

\begin{remark}
It is possible to prove Proposition \ref{link_prop} by reducing to the following two illuminating, yet quite technical lemmas:
\begin{lem}\label{link3_lmm}
    Let $T$ be a colored, labeled tree, and $v$ an $R$-colored leaf vertex of $T$. Define $R(T,v)$ as follows: 
    \begin{enumerate}
        \item  if the parent of $v$ is of color $R$ or $ W $, then define $R(T,v)$ as the quotient of $T$ by the edge between $v$ and $\text{parent}(v)$ colored by $\text{Color}(\text{parent}(v))$; 
        \item if the parent of $v$ is of color $B$, then define $R(T,v)$ as $T$ with the color of $v$ changed to $ W $.
    \end{enumerate}
    There is an isomorphism $$\mathcal{P}_{(T\to\bullet_{ W })}\cong \textnormal{Cone}\big(\mathcal{P}_{(R(T,v)\to\bullet_{ W })}\big).$$
\end{lem}

\begin{lem}\label{link2_lmm}
    Let $T$ be a colored, labeled tree with no $R$-colored vertices. Let $Br_1,\ldots Br_n$ be the subtrees of $T$ that are the connected components of $T$ with the root and the edges attached to the root removed (``$Br$'' stands for ``branch''). Suppose $Br_1,\ldots, Br_m$ have root color $B$ and $Br_{m+1},\ldots,Br_{n}$ have root color $ W $ (so they must be single-vertex trees). Then
    $$\partial\mathcal{P}_{(T\to \bullet_{ W })}\cong \Sigma \partial \mathcal{P}_{(Br_1\to\bullet_{ W })}*\ldots*\Sigma \partial \mathcal{P}_{(Br_m\to\bullet_{ W })}.$$
\end{lem}

The proof we present instead was suggested to us by Victor Turchin.
\end{remark}

Given a poset $P$, we write $P^{\mathrm{op}}$ for its opposite poset with the underlying set being the same as $P$ and all inequalities reversed. 
\begin{definition}
    A {\it (convex) polytope} is a subset $P$ of some $\mathbb{R}^n$ defined by a finite collection of linear inequalities
    $$P=\big\{ \mathbf{x}\in\mathbb{R}^n \,\big|\, f_i(\mathbf{x})\le b_i,\ \forall\ 1\le i\le k \big\}$$
    for some linear functions $f_1,\ldots,f_k:\mathbb{R}^n\to\mathbb{R}$ and numbers $b_1,\ldots,b_k\in\mathbb{R}$. We do not require $P$ to be bounded. We call a polytope a {\it cone} if all the defining inequalities are homogeneous, i.e. all the $b_i$ are $0$. 

    A {\it face} of a polytope $P$ is a non-empty subset of $\mathbb{R}^n$ of form 
    $$\big\{ \mathbf{x}\in\mathbb{R}^n \,\big|\, f_i(\mathbf{x})=b_i,\ \forall i\in I;\ f_i(\mathbf{x})\le b_i,\ \forall i\in \{1,\ldots,k\}\backslash
    I \big\}$$
    for some $I\subset\{1,\ldots,k\},I\neq\emptyset$. 
    
    Ordered by $f_1\le f_2\iff f_1\subset f_2$, the set of faces of $P$ forms a poset $\mathcal{F}(P)$ called the {\it face poset} of $P$. 
\end{definition}

By convexity, every bounded polytope $P$ is homeomorphic to a closed ball and has a regular CW-structure where the closed cells are its faces and $P$ itself; the face poset of $\partial P$ is $\mathcal{F}(P)$. 
Every bounded polytope $P$ has a dual polytope $P^*$ whose face poset is the opposite of that of $P$ -- we do not need any explicit description.

\begin{proof}[Proof of Proposition \ref{link_prop}]
The argument follows a construction in \cite[Section 2.2]{Ducoulombier}. Define a {\it height function} on an RBW-tree $T$ to be a map $h:V(T)\to \mathbb{R}$ such that
    \begin{itemize}
        \item for all $W$-vertices $v$ of $T$, $h(v)=0$;
        \item for all $R$-vertices $v$ of $T$, $h(v)\ge0$;
        \item for all $B$-vertices $v$ of $T$, $h(v)\le0$;
        \item if $v$ is a child of $w$, then $h(v)\ge h(w)$. 
    \end{itemize}
    
Let $V^R(T)$ (resp. $V^B(T)$) be the set of $R$- (resp. $B$-) vertices of $T$. Define $\hat{V}(T)=V^R(T)\cup V^B(T)$. 
Define the following subset of $\mathbb{R}^{\hat{V}(T)}$
$$X(T):=\big\{(x_v)_{v\in \hat{V}(T)}\in \mathbb{R}^{\hat{V}(T)}\,\big|\,\exists \textnormal{ height function }h \textnormal{ such that }h(v)=x_v, \forall v\big\}.$$
Clearly $X(T)$ can be viewed as the set of height functions on $T$. 
Equivalently, $X(T)$ is the set of points in $\mathbb{R}^{\hat{V}(T)}$ satisfying the following homogeneous linear inequalities
$$\{x_v\ge0\}_{v\in V^R(T)}, \{x_v\le0\}_{v\in V^B(T)},\{x_v\ge x_w\}_{v,w\in \hat{V}(T),\ v\textnormal{ is a child of }w}.$$
Therefore, the space of height functions $X(T)$ is a polytope in $\mathbb{R}^{\hat{V}(T)}$, and in particular a cone.

Let us demonstrate that there is a poset isomorphism $\mathcal{I}:\mathcal{F}(X(T))\longrightarrow\mathcal{P}_T^\mathrm{op}$: 

$\mathcal{F}(X(T))\xrightarrow{\mathcal{I}}\mathcal{P}_T$: Let $F\in \mathcal{F}(X(T))$. Take $h\in \mathring{F}\in \mathbb{R}^{\hat{V}(T)}$ (recall if $F$ is a point $\mathring{F}:=F$). Define a (nontrivial) uncolored tree contraction $T\to T'$ as follows: an edge of $T$ is contracted if and only if its two endpoints have the same height under the height function $h$. We can use $h$ to construct a function $h': V(T')\to \mathbb{R}$, whose value on $v'\in V(T')$ is defined to be $h(v)$ for some vertex $v\in V(T)$ that contracts to $v'$. We then color $T'$ such that a vertex $v$ is red if $h'(v)>0$, white if $h'(v)=0$, and blue if $h'(v)<0$. It is clear that $T'$ is a legally colored RBW-tree. 
That the contraction $T\to T'$ is legal is because if $v'\in V(T')$ is red we have $h'(v')>0$, then for all vertices $v\in V(T)$ that contracts to $v'$, $h(v)>0$, and so they are also red; similarly for blue. Set $\mathcal{I}(F)$ equal to this contraction. 

$\mathcal{P}_T\xrightarrow{\mathcal{I}^{-1}}\mathcal{F}(X(T))$:
Given a nontrivial contraction $T\to T'$, let $F\subset X(T)$ be the subset of height functions on $T$ such that if $v_1,v_2\in V(T)$ are contracted to the same vertex in $T'$, then $h(v_1)=h(v_2)$. Using the least upper bound of $v_1,v_2$ in the tree, we see that $F$ can be defined by turning some of the defining inequalities of $X(T)$ into equalities. The subset $F$ is non-empty because taking a height function on $T'$ and pulling it back to $T$ gives an element in $F$. We declare $\mathcal{I}^{-1}(T')=F$. 

It is easy to check that these two constructions are inverse to each other and reverse poset structures.

We first prove the statement of Proposition \ref{link_prop} in the case $\mathfrak{c}$ is a nontrivial contraction of the form $T\to\bullet_W$. 
Intersecting $X(T)$ with the hyperplane $\sum_{v\in\hat{V}(T)} \delta_v x_v=1$, where $\delta_v=+1$ if $v$ is red and $\delta_v=-1$ if $v$ is blue, we obtain a bounded polytope $\tilde{X}(T)$. 
The face poset of $\tilde{X}(T)$ is $\mathcal{F}(X(T))$ minus the minimal element (corresponding to the cone point in $X(T)$), and so the poset isomorphism $\mathcal{I}$ induces a poset isomorphism $\mathcal{F}(\tilde{X}(T)^*) \cong \partial \mathcal{P}_T$. 
Since the boundary of the polytope $\tilde{X}(T)^*$ is homeomorphic to a sphere, attaching a single cell allows us to conclude that $\mathcal{P}_T$ is a CW-poset whose realization is homeomorphic to a ball. 

The complete statement of Proposition \ref{link_prop} for a general nontrivial contraction $\mathfrak{c}:T\to T'$ follows since $\partial\mathcal{P}_\mathfrak{c}$ corresponds to the boundary of the subpolytope of $\tilde{X}(T)^*$ given by the face associated to the contraction $\mathfrak{c}\in \mathcal{P}_T$.
\end{proof}

\section{Stratifications of manifold operads and bimodule cobordisms}\label{section: Stratifications of manifold operads and bimodule cobordisms}

\subsection{Stratifications of operads}\label{section: strat operad}

In this section, we introduce operads stratified by the category $\mathcal{T}$ of single-colored trees. In a stratified operad, the poset of nontrivial contractions $Lk(T)$ of a tree $T \in \mathcal{T}(S)$ controls the regular neighborhoods of the stratum associated to $T$. This structure generalizes the combinatorics that the strata of $\mathrm{FM}_d$ have. A priori only the strata of a stratified operad are required to be manifolds, but basic combinatorics of single-colored trees will imply a stratified operad is actually a manifold operad. Conversely, we will demonstrate that if a manifold operad is collarable, then it is stratifiable, i.e. it has the structure of a stratified operad.

 We fix some $n\leq \infty $. In this section, all $n$-truncated operads are assumed to be levelwise compact, normal, and second countable.

Given an $n$-truncated operad represented by $\phi:\mathcal{T}_{\leq n} \rightarrow \mathrm{Top}$, the open stratum associated to $T$ is
$$\mathring{\phi}(T):=\phi(T)\Big\backslash\bigcup_{\textnormal{non-trivial contractions }T'\to T}\textnormal{image}\big(\phi(T'\to T)\big).$$

\begin{definition}
    Given an $n$-truncated operad represented by $\phi:\mathcal{T}_{\leq n} \rightarrow \mathrm{Top}$, we say it is \textit{stratified} if for all finite sets $S$ of cardinality $ \leq n$,  each contraction of $S$-labeled trees $T\to T'$ is given the data \label{Ncondition5_item} of an open subset $\mathcal{N}_{T\to T'}\subset \phi(T')$ and a homeomorphism\footnote{
    Given a CW-complex $X$ with face poset $P$, we write $\mathring{\textnormal{Cone}}(X)$ or $\mathring{\textnormal{Cone}}(P)$ for the open cone $\big(X\times [0,1)\big)/X\times\{0\}$, and write ${\textnormal{Cone}}(X)$ or ${\textnormal{Cone}}(P)$ for the closed cone $\big(X\times [0,1]\big)/X\times\{0\}$. 
    }
    $$\nu_{T\to T'}: \mathring\phi(T)\times\mathring{\textnormal{Cone}}(Lk(T\to T'))\xlongrightarrow{\cong}\mathcal{N}_{T\to T'} $$
 commuting with the symmetric group actions and subject to the following conditions:

    \begin{enumerate}
    
        \item \label{stratifiedoperad1_item}
        The spaces $\mathring{\phi}(\bullet_S)$ are topological manifolds.

    \item \label{stratifiedoperad2_item}
    The function
    $$\phi: \colim_{{\mathcal{T}(S)\backslash\bullet_{S }}}\phi 
            \longrightarrow\phi(\bullet_{ S})$$
            maps the domain homeomorphically onto $ \phi (\bullet_S) \setminus \mathring{\phi}(\bullet_S)$.

            \item \label{stratifiedoperad3_item}
            $\mathcal{N}_{T\to T'}$ is a neighborhood of the image of $\mathring\phi(T)$ under $\phi(T\to T')$, and 
            $$\nu_{T\to T'}|_{\mathring{\phi}(T)\times \{\text{the cone point}\}}=\phi(T\to T')|_{\mathring{\phi}(T)}: \mathring{\phi}(T)\longrightarrow \phi(T').$$

            \item \label{stratifiedoperad4_item}

            For contractions $T\to T''\to T'$, 
            $$
            \mathcal{N}_{T\to T''}=\mathcal{N}_{T\to T'}\cap \phi(T''\to T')\big(\phi(T'')\big),
            \qquad
            \nu_{T\to T''}=\nu_{T\to T'}|_{\mathring\phi(T)\times \mathring{\textnormal{Cone}}(Lk(T\to T''))}. 
            $$            
              \[
            \begin{tikzpicture}
               \coordinate (o) at (0,0); 
               \coordinate (a) at (2,1); 
               \coordinate (b) at (2.6, 0.4);
               \coordinate (c) at (2.2,-0.4);
               \coordinate (d) at (1.5,-0.5);
               \coordinate (e) at (1.3,0.9); 
               \draw[fill,purple] (o) circle [radius=0.03]; 
               \node [left,purple] at (o) {\tiny $\phi(T\to T')\big(\mathring\phi(T)\big)$};
               \draw [fill, seagreen, opacity=0.15] (a)--(b)--(c)--(d)--(e)--cycle; 
               \draw[fill,gray,opacity=0.3] (o)--(e)--(a)--(b)--(c)--(d)--cycle;

                \draw (o) to ($1.4*(d)$); 
                \draw [cyan] (o) to ($1.4*(e)$); 
                \node [cyan, above] at (-0.3,1) {\tiny $\phi(T''\to T')\big(\phi(T'')\big)$}; 
                \draw [cyan, ->] (-0.3,1) to ($0.5*(e)$); 
                \draw [fill,blue] (e) circle [radius=0.02];
                \node [above, blue] at ($(e)+(0.3,0.5)$) {\tiny $Lk(T\to T'')$};
                \draw [blue, ->] ($(e)+(0.3,0.5)$) to (e);
                
                \node [seagreen] at (3.5,-0.2) {\tiny $Lk(T\to T')$};
                \draw [seagreen, ->] (2.75,-0.2) to ($0.5*(b)+0.5*(c)$);
                \draw  (o) to ($1.3*(a)$); 
                \draw  (o) to ($1.3*(b)$);
                \draw  (o) to ($1.3*(c)$); 
                \node at (3.5,1) {\tiny $\phi(T')$};
                \node [gray] at (1,-0.8) {\tiny $\mathcal{N}_{T\to T'}$};
                \draw[gray, ->] (1,-0.7) to (0.9,0);
            \end{tikzpicture}
            \]

\end{enumerate}

Finally, associated to contractions of decomposable trees $T \rightarrow T'$ we require the stratification take the following form. Denote by $\Theta=\{(T_i,\eta_i)\}_{i\in I}$ the contraction system of $T$ such that $T'=T/\Theta$. Since the contraction is not to $\bullet$, each $T_i$ lies in $\mathcal{T}_{\le n-1}$. Denote by $[T_i]$ the vertex in $T'$ obtained by contracting $T_i$.

By the definition of an operad, the map $\phi(T\to T'):\phi(T)\to \phi(T')$ is given by the product
\begin{align*}
    \prod_{i\in I} \prod_{v\in V(T_i)}\phi(\overline{cld}(v))&\xlongrightarrow{\prod_{i \in I} \phi}
    \prod_{i\in I}\phi(\overline{cld}([T_i])).
\end{align*}
We let
$$f: \mathring{\textnormal{Cone}}(Lk(T\to T'))\xrightarrow{\cong} \prod_{i\in I}\mathring{\textnormal{Cone}}(Lk(T_i\to \bullet))$$
be the homeomorphism induced from the obvious isomorphism of posets \footnote{Suppose 
$$\alpha: X_1*X_2*\ldots*X_n\xrightarrow{\cong} X$$
is an isomorphism of regular CW-complexes. 
Recall by how the join is defined, 

a point in $X_1*\ldots *X_n$ can be written in the form $(x_{i_1},\ldots,x_{i_k})$, where $\emptyset\neq\{i_1,\ldots,i_k\}\subset\{1,\ldots n\}$ and $x_{i_j}\in X_{i_j}$. 
We define the induced map 
\begin{align*}
\hat{\alpha}:\mathring{\textnormal{Cone}}(X_1)\times\ldots\times\mathring{\textnormal{Cone}}(X_n)&\xrightarrow{\cong}\mathring{\textnormal{Cone}}(X)\\
\Big(\big((x_{i},t_i)\big)_{i\in I\subset\{1,\ldots,n\}},(\textnormal{Cone point})_{i\notin I}\Big)
&\longrightarrow
\Big(\alpha\big((x_i)_{i\in I})\big),\prod_{i\in I}t_i\Big).
\end{align*}
} $$Lk(T\to T')\cong *_{i\in I}\,Lk(T_i\to\bullet).$$

We require 
$$\mathcal{N}_{T\to T'}:=\prod_{i\in I}\mathcal{N}_{T_i\to\bullet}\subset\prod_{v\in V(T')}\phi(\overline{cld}(v))=\phi(T'),$$ and require $\nu_{T\to T'}$  is the following composition\footnote{The vertices which are not contracted have empty links, and thus their cones consist of a single point. They do not affect this product.}
\[
\begin{tikzcd}
        \mathring{\phi}(T)\times \mathring{\textnormal{Cone}}(Lk(T\to T'))\dar[equal]
        \\
        \Big(
        \displaystyle\prod_{i\in I} \displaystyle\prod_{v\in V(T_i)} \mathring{\phi}(\overline{cld}(v))
        \Big)
        \times \mathring{\textnormal{Cone}}(Lk(T\to T'))
        \dar{(\textnormal{id}, f)}
        \\
        \displaystyle\prod_{i\in I}\Big(\mathring{\phi}(T_i)\times \mathring{\textnormal{Cone}}(Lk(T_i\to \bullet))\Big)
         \dar{\prod_{i\in I}\nu_{T_i\to \bullet}}
        \\
        \displaystyle\prod_{i\in I} \mathcal{N}_{T_i\to\bullet}.
\end{tikzcd}
\]

\end{definition}

This finishes the definition of a stratified operad. Observe that the equivariance of these decompositions is immediate and that the conditions $(\ref{stratifiedoperad3_item})$ and $(\ref{stratifiedoperad4_item})$ for such contractions of decomposable trees immediately hold by induction.

The remainder of this section will be dedicated to proving the following theorem.

\begin{thm}\label{thm: manifold operad is stratified}
    Let $n \leq \infty$. An $n$-truncated stratified operad is an $n$-truncated manifold operad. Conversely, an $n$-truncated collarable manifold operad has a stratification.
\end{thm}

\begin{proof}
$(\impliedby)$: We construct $\mathcal{N},\nu$ inductively on the size of the finite set $S$. The initial step is when $S=2$. 

\textbf{Initial $|S|=2$ case.} 

When $|S|=2$, there are no nontrivial tree contractions, and so it suffices to check property (\ref{stratifiedoperad1_item}) which is true by Definition \ref{manifoldoperad_dfn}.

\textbf{Inductive step.}
Properties (\ref{stratifiedoperad1_item}) and (\ref{stratifiedoperad2_item}) are true by the definition of a manifold operad. 

Suppose for $N$ such that $n \geq N\ge3$ we have constructed $\mathcal{N}_{T\to T'},\nu_{T\to T'}$ for all contractions where $T,T'\in\mathcal{T}_{\le N-1}$, subject to the listed conditions. We will now construct $\mathcal{N}_{T\to T'},\nu_{T\to T'}$ for $T,T'\in \mathcal{T}_N$ such that conditions $(\ref{stratifiedoperad3_item}),(\ref{stratifiedoperad4_item})$, and equivariance are satisfied.

\textbf{Step 1: understanding the topology of the boundary}

\begin{lem}\label{lem: stratification of boundary of operad}
    Suppose $(\phi,\mathcal{N},\nu)$ has been shown to be a stratified operad, up to sets of cardinality $|\mathcal{S}|-1$.
    Then for every $T\in \mathcal{T}(S)\backslash\{\bullet\}$, $\mathring\phi(T)$ has an open neighborhood $\mathcal{N}^\partial_T$ in  $\colim_{{\mathcal{T}(S)\backslash\bullet_{ }}} \phi$ canonically homeomorphic to 
$\mathring{\textnormal{Cone}}\big(\partial Lk(T)\big)\times \mathring{\phi}(T)$, up to reparametrization of the cells of $\partial Lk(T)$. 

\end{lem}

\begin{proof}
Let $T\in \mathcal{T}_N\backslash\{\bullet\}$. Define $\mathcal{T}_{\geq T,\neq\bullet}$ to be the subcategory of $\mathcal{T}_N\backslash\{\bullet\}$ consisting of trees $T'$ such that there exists a, possibly trivial, contraction $T\to T'$. Define $\mathcal{N}^\partial_T\subset \colim_{{\mathcal{T}(S)\backslash\bullet_{ }}} \phi$ to be 
$$\mathcal{N}^\partial_{T}=\bigcup_{T'\in \mathcal{T}_{\geq T,\neq\bullet}}\mathcal{N}_{T\to T'}.$$
Then, 
\begin{align*}
    \mathcal{N}^\partial_{T}&=\colim_{T'\in \mathcal{T}_{\geq T,\neq\bullet}} (\mathcal{N}_{T\to T'})\\
    &\stackrel{\nu}{=}\colim_{T'\in \mathcal{T}_{\geq T,\neq\bullet}}
    \big(\mathring{\phi}(T)\times\mathring{\textnormal{Cone}}(Lk(T\to T'))\big)\\
    &=\mathring{\phi}(T)\times \colim_{T'\in \mathcal{T}_{\geq T,\neq\bullet}}
    \mathring{\textnormal{Cone}}(Lk(T\to T'))\\
    &=\mathring{\phi}(T)\times \mathring{\textnormal{Cone}}\big(\colim_{T'\in \mathcal{T}_{\geq T,\neq\bullet}}
    Lk(T\to T')\big).
\end{align*}
The first hypothesis of property $(4)$ allows us to form the first commutative diagram which is the union by definition.

The second hypothesis of property $(4)$ allows us to make the identification of colimit diagrams ``$\stackrel{\nu}{=}$''. Since in the category of posets
$$\colim_{T'\in \mathcal{T}_{\geq T,\neq\bullet}}
    Lk(T\to T')=\partial Lk(T),$$
the conclusion of the lemma follows by Lemma \ref{CWcolim_lmm}. 

\end{proof}

\textbf{\mynameis{Step 2}\label{step 3 of operad strat}: define $\mathcal{N}_{T\to \bullet}$ and $\nu_{T\to \bullet}$.}

We need to define an embedding 
\[\nu_{T\to\bullet}:\mathring{\phi}(T)\times\mathring{\textnormal{Cone}}(Lk(T\to\bullet))\longrightarrow \phi(\bullet)=\phi(S)\]
for each $S$ such that $|S|=N$ and each $T\in \mathcal{T}(S)\backslash\bullet$, that is suitably equivariant and natural enough to satisfy properties (\ref{stratifiedoperad3_item}) and (\ref{stratifiedoperad4_item}). 
Since $\phi(S)$ is assumed to have a $\Sigma_S$-equivariant collar, it suffices to embed instead into the product
\[\colim_{{\mathcal{T}(S)\backslash\bullet_{ }}} \phi \times [0,1].\]

We denote the cone parameter map by
$r:\mathcal{N}^\partial_T\to [0,1)$. 
On the other hand, define 
\begin{align*}
    h:\colim_{{\mathcal{T}(S)\backslash\bullet_{ }}} \phi\times[0,1]&\longrightarrow [0,1],\qquad\qquad\quad\  h(p,t)=t;\\
    \pi:\colim_{{\mathcal{T}(S)\backslash\bullet_{ }}} \phi\times[0,1]&\longrightarrow \colim_{{\mathcal{T}(S)\backslash\bullet_{ }}} \phi,\qquad \pi(p,t)=p.
\end{align*}

For a contraction $T\to \bullet$, define  
$$\mathcal{N}_{T\to \bullet}=\big\{p\in \mathcal{N}^\partial_{T}\times[0,1)\big|\ r(p)^2+h(p)^2\leq 1\big\}\subset \colim_{{\mathcal{T}(S)\backslash\bullet_{ }}} \phi \times [0,1].$$

We define $\nu_{T\to\bullet}:\mathring{\phi}(T)\times \mathring{\textnormal{Cone}}(Lk(T\to\bullet))\to \mathcal{N}_{T\to\bullet}$ to be the unique homeomorphism such that 
\begin{itemize}
    \item it agrees with the cone structure on $(\colim_{{\mathcal{T}(S)\backslash\bullet_{ }}} \phi )\times\{0\}$;
    \item $\sqrt{r^2+h^2}$ is the cone parameter;
    \item for each cone ray $l$, $\pi(l)$ is again a cone ray in $\colim_{{\mathcal{T}(S)\backslash\bullet_{ }}} \phi$;
    \item along each cone ray, $r/h$ is constant.  
\end{itemize}
See Figure \ref{halfball_fig} on page \pageref{halfball_fig} for an illustration of the three-colored variant.

It is straightforward to check the neighborhood and identification are natural, and so this completes the inductive step. 

 $(\implies)$: We now prove that an $n$-truncated stratified operad is an $n$-truncated manifold operad.

\begin{prop}\label{prop: decomposables are manifold}
    If $(\phi,\mathcal{N},\nu)$ is an $n$-truncated stratified operad, then for every finite set $S$ with $|S|\leq n$,
$\phi(\bullet_S)$ is a manifold with boundary 
    $\phi(\bullet_S) \setminus \mathring{\phi}(\bullet_S) \cong \colim_{{\mathcal{T}(S)\backslash\bullet_{ }}} \phi$.
\end{prop}
\begin{proof}
For $|S|=2$ the statement is true.  
    Inductively, suppose that it holds for every finite set $S'$ of cardinality
    $2 \leq |S'|<|S|$, so that the spaces  $\phi(\bullet_{S'})$ are manifolds with boundary. By Lemma \ref{colimtopology_lmm} and Lemma \ref{operadtopology_lmm}, $\colim_{T\in \mathcal{T}(S)\setminus \bullet} \phi$ is a Hausdorff space. 
    By Lemma \ref{lem: stratification of boundary of operad} each stratum $\mathring{\phi}(T)$ has neighborhoods in $\phi(\bullet_S)$, respectively 
in $\phi(\bullet_S) \setminus \mathring{\phi}(\bullet_S) $, that are homeomorphic to the product
of $\mathring{\phi}(T)$ with the open cones 
on $Lk(T)$, respectively $\partial Lk(T)$,
    
that  are the CW-complexes associated to the combinatorial links of Definition \ref{definition: link of a tree} and their boundary. The analog of Proposition \ref{link_prop} for single-colored trees applied to $T \rightarrow \bullet$ implies these are topologically discs, respectively their boundary spheres, and so  $\phi(\bullet_S)$ is a manifold with boundary  $\phi(\bullet_S) \setminus \mathring{\phi}(\bullet_S)$. 
    
\end{proof}
This finishes the proof of Theorem \ref{thm: manifold operad is stratified}.
\end{proof}

\subsection{Stratification of bimodules}\label{section: strat bimodule}
In this section, we introduce bimodules stratified by the category $\mathcal{T}^{RBW}$. The combinatorial analysis of the links performed in Section \ref{link_sec} is much more delicate than in the case of stratified operads, however the construction and conclusion of this section will be analogous. We take some number $n\leq \infty$.
In this section, all $n$-truncated operads and bimodules are assumed to be levelwise compact, normal, and second countable.

Given a bimodule $\phi:\mathcal{T}^{RBW} \rightarrow \mathrm{Top}$, denote
    $$\mathring{\phi}(T):=\phi(T)\Big\backslash\bigcup_{\textnormal{non-trivial contractions }T'\to T}\textnormal{image}\big(\phi(T'\to T)\big).$$
    
For each $v\in V(T)$, denote 
    $\mathring{\mathcal{S}}(v):=\mathring{\phi}(\bullet_{\textnormal{Color}(v),\overline{cld}(v)})$; define $\mathring{R}(S),\mathring{B}(S),\mathring{W}(S)$ to be $\mathring{\phi}(\bullet_{R,S})$, $\mathring{\phi}(\bullet_{B,S})$, $\mathring{\phi}(\bullet_{W,S})$, respectively. 

\begin{definition}\label{stratifiedbimodule_dfn}
    Given an $n$-truncated bimodule represented by $\phi:\mathcal{T}^{RBW}_{\le n}\rightarrow \mathrm{Top}$, we say it is {\it stratified} if for all finite sets $S$ of cardinality $ \leq n$,  each contraction of colored $S$-labeled trees $T\to T'$ is associated with the data of an open subset $\mathcal{N}_{T\to T'}\subset \phi(T')$ and a homeomorphism
    $$\nu_{T\to T'}: \mathring\phi(T)\times\mathring{\textnormal{Cone}}(Lk(T\to T'))\longrightarrow\mathcal{N}_{T\to T'} $$
commuting with symmetric group actions and subject to the following conditions:

\begin{enumerate}
    
    \item \label{RBWringmfld_item} The spaces $\mathring{R}(S),\mathring{B}(S),\mathring{W}(S)$ are topological manifolds;

    \item\label{boundaryconditionX_item} The function
    $$\phi: \colim_{{\mathcal{T}^{RBW}(S)\backslash\bullet_{W,S}}}\phi 
            \longrightarrow\phi(\bullet_{ W ,S})= W (S)$$
            maps the domain homeomorphically onto $ W (S) \setminus \mathring{W}(S)$; 
    \item\label{boundaryconditionB_item}{}
  Restricting to the category $\mathcal{T}^B$,  
            $$\phi_B: \colim_{{\mathcal{T}^{B}(S)\backslash\bullet_{B,S}}}\phi 
            \longrightarrow\phi(\bullet_{B,S})=B(S)$$
            maps the domain homeomorphically onto $ B (S) \setminus \mathring{B}(S)$; 
     \item\label{boundaryconditionR_item}{}
    Restricting to the category $\mathcal{T}^R$ 
            $$\phi_R: \colim_{{\mathcal{T}^{R}(S)\backslash\bullet_{R,S}}}\phi 
            \longrightarrow\phi(\bullet_{R,S})=R(S)$$
            maps the domain homeomorphically onto $ R(S) \setminus \mathring{R}(S)$;

    \item \label{Ncondition0_item} 
        $$\nu_{T\to T'}|_{\mathring{\phi}(T)\times \{\text{the cone point}\}}=\phi({T\to T'})|_{\mathring{\phi}(T)}: \mathring{\phi}(T)\longrightarrow \phi(T');$$
        
    \item \label{Ncondition_item}
            For contractions $T\to T''\to T'$, 
            $$
            \mathcal{N}_{T\to T''}=\mathcal{N}_{T\to T'}\cap \phi(T''\to T')\big(\phi(T'')\big),
            \qquad
            \nu_{T\to T''}=\nu_{T\to T'}|_{\mathring\phi(T)\times \mathring{\textnormal{Cone}}(Lk(T\to T''))}. 
            $$

\end{enumerate}

As before, we require the following decomposability for contractions $T \rightarrow T'$ such that $T' \neq \bullet_W$.

Suppose $T'=\bullet_{R}$, then $T$ must be a tree with all red vertices, and so we define this to be the stratification of the operad $R$ guaranteed by the previous section. Similarly for $T'=\bullet_{B}$. The remaining situation is that $T'\neq \bullet_{R \text{ or }B\text{ or } W }$.

If $\Theta=\{(T_i,\eta_i)\}_{i\in I}$ is the contraction system of $T$ such that $T'=T/\Theta$, then each $T_i$ lies in $\mathcal{T}^{RBW}_{\le n-1}$.

Then, by Definition \ref{bimodule_dfn}, the map $\phi(T\to T'):\phi(T)\to \phi(T')$ is given by 
\begin{align*}
    \prod_{i\in I} \prod_{v\in V(T_i)}\mathcal{S}(v)&\longrightarrow
    \prod_{i\in I}\mathcal{S}([T_i])\\      
     ((p_v)_{v\in T_i}) _{i\in I}
     &\longrightarrow  
    \phi_{T_i\to \bullet_{\eta_i}} ((p_v)_{v\in V(T_i)}) 
\end{align*}
We require 
$$\mathcal{N}_{T\to T'}:=\prod_{i\in I}\mathcal{N}_{T_i\to\bullet_{\eta_i}}\subset\prod_{v\in V(T')}\mathcal{S}(v)=\phi(T')$$
and require $\nu_{T\to T'}$ to be the following composition
\[
\begin{tikzcd}
        \mathring{\phi}(T)\times \mathring{\textnormal{Cone}}(Lk(T\to T'))\dar[equal]
        \\
        \Big(
        \displaystyle\prod_{i\in I} \displaystyle\prod_{v\in V(T_i)} \mathring{\mathcal{S}}(v)
        \Big)
        \times \mathring{\textnormal{Cone}}(Lk(T\to T'))
         \dar{(\textnormal{id}, f)}
        \\ 
        \displaystyle\prod_{i\in I}\mathring{\phi}(T_i)\times \mathring{\textnormal{Cone}}(Lk(T_i\to \bullet_{\eta_i}))
        \dar{(\nu_{T_i\to \bullet_{\eta_i}})_{i\in I}}
        \\ {}
                \end{tikzcd}\]
\[\begin{tikzcd}
        \displaystyle\prod_{i\in I} \mathcal{N}_{T_i\to\bullet_{\eta_i}} \subset\mathcal{N}_{T\to T'},
\end{tikzcd}
\]
where 
$$f: \mathring{\textnormal{Cone}}(Lk(T\to T'))\xrightarrow{\cong} \prod_{i\in I}\mathring{\textnormal{Cone}}(Lk(T_i\to \bullet_{\eta_i}))$$
is induced from the map
$$Lk(T\to T')\cong *_{i\in I}\,Lk(T_i\to\bullet_{\eta_i})$$
given by Lemma \ref{link1_lmm}.

\end{definition}

The equivariance of these choices is immediate. That conditions (\ref{Ncondition0_item}), (\ref{Ncondition_item}) automatically hold for such contractions is clear from induction.

The remainder of this section will be dedicated to proving the following theorem.

\begin{thm}\label{thm: bimodule cobordism is stratified}
    Let $n \leq \infty$. An $n$-truncated stratified bimodule is an $n$-truncated bimodule cobordism. Conversely, an $n$-truncated collarable bimodule cobordism between $n$-truncated collarable manifold operads has a stratification.
\end{thm}

\begin{proof}

Denote by $\phi_B,\phi_R$ the restriction of $\phi$ to $\mathcal{T}^B,\mathcal{T}^R$, respectively.

$(\impliedby)$: Let $\phi:\mathcal{T}^{RBW}_{\leq n}\rightarrow \mathrm{Top}$ be a collarable $n$-truncated bimodule cobordism. Denote by $\phi_B,\phi_R$ the restriction of $\phi$ to $\mathcal{T}^B,\mathcal{T}^R$, respectively. 
The functors represent $\phi_B,\phi_R$ are manifold operads and $\phi$ is a bimodule cobordism from $\phi_B$ to $\phi_R$. 

Properties (\ref{RBWringmfld_item})-(\ref{boundaryconditionR_item}) in Definition \ref{stratifiedbimodule_dfn} follow from the definitions of manifold operads and bimodule cobordisms. 
It remains to construct $\mathcal{N},\nu$ satisfying (\ref{Ncondition0_item}), (\ref{Ncondition_item}). We construct the $\mathcal{N},\nu$ inductively on the size of the finite set $S$. 

\textbf{Initial $|S|=2$ case.} 

 There are two non-trivial tree contractions: $\bullet_{R}\to\bullet_{ W },\bullet_B\to\bullet_{ W }$. We define $(\mathcal{N}_{\bullet_{R \text{ (resp. }B)}\to \bullet_{ W }},\nu_{\bullet_{R \text{ (resp. }B)}\to \bullet_{ W }})$ to be the inclusion of a chosen equivariant collar neighborhood of $R(S)$ (resp. $B(S)$) in $ W(S)$, which exists by the assumption that $\phi$ is collarable.
It is clear that conditions (\ref{Ncondition0_item}), (\ref{Ncondition_item}) are satisfied. 

\textbf{Inductive step.}
Suppose for some $N\ge2$ we have already constructed $\mathcal{N}_{T\to T'},\nu_{T\to T'}$ for all contractions where $T,T'\in\mathcal{T}^{RBW}_{\le N-1}$, satisfying all the conditions listed in Definition \ref{stratifiedbimodule_dfn} for finite sets $S$ with $|S|\le N-1$. We now construct $\mathcal{N}_{T\to T'},\nu_{T\to T'}$ for $T,T'\in \mathcal{T}^{RBW}_N$ such that all the conditions for finite sets $S$ with $|S|\le N$ are satisfied.

\textbf{Step 1: understanding the topology of the boundary}

\begin{lem}\label{lem: stratification of boundary of bimodule}
    Suppose we have defined $\mathcal{N}_{T\to T'}$ and $\nu_{T\to T'}$ for all contractions $T\to T'$ such that $T,T'\in \mathcal{T}^{RBW}_{N}$, and $T'\neq \bullet_{W}$, satisfying properties (\ref{Ncondition0_item}), (\ref{Ncondition_item}) in Definition \ref{stratifiedbimodule_dfn}. 
    Then, for every $T\in \mathcal{T}^{RBW}(S)\backslash\{\bullet_W\}$, $\mathring\phi(T)$ has an open neighborhood $\mathcal{N}^\partial_T$ in  $\colim_{{\mathcal{T}^{RBW}(S)\backslash\bullet_{W}}} \phi$ canonically homeomorphic to 
$\mathring{\textnormal{Cone}}\big(\partial Lk(T)\big)\times \mathring{\phi}(T)$, up to reparametrization of the cells of $\partial Lk(T)$.
\end{lem}

\begin{proof}
Let $T\in \mathcal{T}^{RBW}_N\backslash\{\bullet_W\}$. Define $\mathcal{T}^{RBW}_{\ge T,\neq\bullet_W}$ to be the subcategory of $\mathcal{T}^{RBW}_N\backslash\{\bullet_W\}$ consisting of trees $T'$ such that there exists a (possibly trivial) contraction $T\to T'$. Define $\mathcal{N}^\partial_T\subset \colim_{{\mathcal{T}^{RBW}(S)\backslash\bullet_{W}}} \phi$ to be 
$$\mathcal{N}^\partial_{T}=\bigcup_{T'\in \mathcal{T}^{RBW}_{\geq T,\neq\bullet_W}}\mathcal{N}_{T\to T'}.$$
Then, 
\begin{align*}
\mathcal{N}^\partial_{T}&=\colim_{T'\in \mathcal{T}^{RBW}_{\geq T,\neq\bullet_W}} (\mathcal{N}_{T\to T'})\\
    &\stackrel{\nu}{=}\colim_{T'\in \mathcal{T}^{RBW}_{\geq T,\neq\bullet_W}}
    \big(\mathring{\phi}(T)\times\mathring{\textnormal{Cone}}(Lk(T\to T'))\big)\\
    &=\mathring{\phi}(T)\times \colim_{T'\in \mathcal{T}^{RBW}_{\geq T,\neq\bullet_W}}
    \mathring{\textnormal{Cone}}(Lk(T\to T'))\\
    &=\mathring{\phi}(T)\times \mathring{\textnormal{Cone}}\big(\colim_{T'\in \mathcal{T}^{RBW}_{\geq T,\neq\bullet_W}}
    Lk(T\to T')\big).
\end{align*}
The first hypothesis of property (\ref{Ncondition_item}) allows us to form the first commutative diagram which is the union,

and the last two equations follow from the second hypothesis of property (\ref{Ncondition_item}).

By Lemma \ref{CWcolim_lmm}, the last colimit above can be viewed as taken in the category of posets. 
Since in the category of posets
$$\colim_{T'\in \mathcal{T}^{RBW}_{\geq T,\neq\bullet_W}}
    Lk(T\to T')=\partial Lk(T)$$
the conclusion of the lemma follows. 

\end{proof}

\textbf{\mynameis{Step 2}\label{step 3 of bimodule strat}: define $\mathcal{N}_{T\to \bullet_W}$ and $\nu_{T\to \bullet_W}$.}

We need to define an embedding 

\[\nu_{T\to\bullet_W}:\mathring{\phi}(T)\times\mathring{\textnormal{Cone}}(Lk(T\to\bullet_W))\longrightarrow \phi(\bullet_W)=W(S)\]
for each $S$ such that $|S|=N$ and each $T\in \mathcal{T}^{RBW}(S)\backslash\bullet_W$, that is suitably equivariant and natural enough to satisfy properties (\ref{Ncondition0_item}) and (\ref{Ncondition_item}). 
Since $W(S)$ is assumed to have a $\Sigma_S$-equivariant collar of its boundary $\partial W(S)=\colim_{{\mathcal{T}^{RBW}(S)\backslash\bullet_{W}}} \phi$, it suffices to embed instead into the product
$\partial W(S)\times [0,1]$.
Here, the $[0,1]$ parameter is the collar neighborhood thickness, with 0 being on the boundary. 

By Lemma \ref{lem: stratification of boundary of bimodule}, $\mathring{\phi}(T)$ has a cone-shaped neighborhood $\mathcal{N}^\partial_T$ in $\colim_{{\mathcal{T}^{RBW}(S)\backslash\bullet_{W}}} \phi=\partial W(S)$. 
We denote the cone parameter map by
$r:\mathcal{N}^\partial_T\to [0,1)$. 
On the other hand, define 
\begin{align*}
    h:\partial W(S)\times[0,1]&\longrightarrow [0,1],\qquad\quad  h(p,t)=t;\\
    \pi:\partial W(S)\times[0,1]&\longrightarrow \partial W(S),\qquad \pi(p,t)=p.
\end{align*}

For a contraction $T\to \bullet_W$, define  
$$\mathcal{N}_{T\to \bullet_W}=\big\{p\in \mathcal{N}^\partial_{T}\times[0,1)\big|\ r(p)^2+h(p)^2\leq 1\big\}\subset \partial W(S)\times [0,1].$$

We define $\nu_{T\to\bullet_{ W }}:\mathring{\phi}(T)\times \mathring{\textnormal{Cone}}(Lk(T\to\bullet_{ W }))\to \mathcal{N}_{T\to\bullet_{ W }}$ (see Figure \ref{halfball_fig}) to be the unique homeomorphism such that 
\begin{itemize}
    \item it agrees with the cone structure on $\partial W(S)\times\{0\}$;
    \item $\sqrt{r^2+h^2}$ is the cone parameter;
    \item for each cone ray $l$, $\pi(l)$ is again a cone ray in $\partial W(S)$;
    \item along each cone ray, $r/h$ is constant.  
\end{itemize}

\begin{figure}
    \begin{tikzpicture}
        \coordinate (O) at (0,0);
        \coordinate (A) at (-4,-1);
        \coordinate (B) at (2,-1);
        \coordinate (C) at (4,1);
        \coordinate (D) at (-2,1);
        
        \fill [orange, name path=ell, opacity=0.2] (O) ellipse (1.5 and 0.6);
        \node [orange] at (-0.3,0.3) {\tiny $\mathcal{N}^\del_T$};

        \fill[red, opacity=0.2] ($(O)-(1.5,0)$) arc (180:360:1.5) arc (0:-180:1.5 and 0.6); 
        \node [red] at (0,-1.3){\tiny $\mathcal{N}_{T\to\bullet_W}$};

        \draw [gray] (A) -- (B) -- (C) -- (D) -- cycle; 
        
        \draw [fill] (5,0) circle [radius=0.02];
        \draw [->] (5,0) to (5,-1);
        \node [right] at (5,0) {$h=0$};
        \node [right] at (5,-1) {$h$};

        \node at (3,0.5) {$\del W$};
        \node at (3,-1.5) {$W$};
        
        \draw[fill] (O) circle [radius=0.03];

        \draw [name path=ray1] (O) to (2,0);
        \draw [name path=ray2] (O) to (0.8,0.8);
        \draw [name path=ray3] (O) to (-2,0.3);
        \draw [name path=ray4] (O) to (-2,-0.4);
        \draw [name path=ray5] (O) to (1,-0.75);    

        \node at (-0.1,-0.3) {\tiny $\mathring{\phi}(T)$};
        \node [right] at (2,0) {\tiny $\mathring{\phi}(T_1)$}; 
        \node [right] at (0.8,0.8) {\tiny $\mathring{\phi}(T_2)$}; 
        \node [above] at (-2,0.3) {\tiny $\mathring{\phi}(T_3)$}; 
        \node [left] at (-2,-0.4) {\tiny $\mathring{\phi}(T_4)$}; 
        \node [right] at (1,-0.75) {\tiny $\mathring{\phi}(T_5)$}; 
    \end{tikzpicture}
    \caption{An illustration of $\mathcal{N}_{T\to\bullet_W}$: suppose $T_1,\ldots,T_5$ are some trees that $T$ contracts to; the center dot is $\mathring{\phi}(T)$, and the rays are $\mathring{\phi}(T_1),\ldots,\mathring{\phi}(T_5)$. The orange 2-dimensional disc, lying in $\del W$, represents $\mathcal{N}^\del_T$, and the red 3-dimensional half-ball, lying in $W$, represents $\mathcal{N}_{T\to\bullet_W}$. The cone rays (not explicitly shown in the picture) are straight lines originating from $\mathring{\phi}(T)$. }
    \label{halfball_fig}
\end{figure}
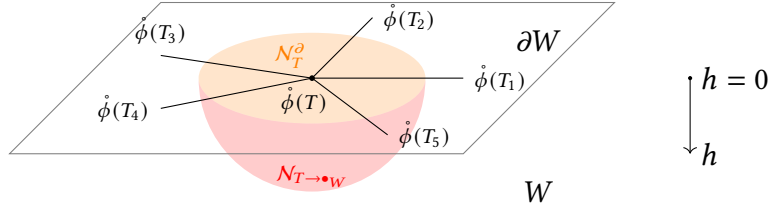
The neighborhood and identification are defined so as to commute with the inclusions, and so this completes the inductive step.

$(\implies)$: Let $N \leq n$ and assume the assertion has been proven up to cardinality $N$.

\begin{prop}\label{prop: bimodule decomposables are manifold}
     If $(\phi,\mathcal{N},\nu)$ is an $N$-truncated stratified bimodule, then for every finite set $S$ with $|S|\leq N$, $\colim_{{\mathcal{T}^{RBW}(S)\backslash \bullet_{W }}} \phi$ is a manifold.
\end{prop}

\begin{proof}
That the space is locally Euclidean follows from Lemma \ref{lem: stratification of boundary of bimodule} combined with Proposition \ref{link_prop}. That it is Hausdorff follows from Lemma \ref{colimtopology_lmm} and Lemma \ref{operadtopology_lmm}.

\end{proof}
This immediately implies that an $n$-truncated stratified bimodule is, in fact, a bimodule cobordism and finishes the proof of Theorem \ref{thm: bimodule cobordism is stratified}.

\end{proof}

\subsection{Composition of stratified bimodules}
In this section, we address the question of whether bimodule cobordisms have a reasonable notion of composition.

\begin{definition}
    The category of $5$-colored trees labeled by a finite set $S$ is notated $\mathcal{T}^{5}(S)$. It has objects the trees with vertices assigned a ``color'' from $\{B,W,O,V,R\}$ with the requirement that a path from a leaf to the root encounter colors in the order 

$$\underbrace{R\ R\ \ldots\ R\ }_{\ge0 \text{ times}}\underbrace{ V }_{0 \text{ or }1 \text{ time}}\underbrace{\ O\ldots\ O\ O}_{\ge0\text{ times} } \underbrace{W}_{0 \text{ or }1 \text{ time}} \underbrace{BB \ldots B}_{\ge0 \text{ times}}.$$
We declare the morphisms to be contractions of uncolored trees which have the property that 
\begin{enumerate}
    \item the preimage of a $B$-vertex consists only of $B$-vertices,
    \item  the preimage of a $W$-vertex consists only of $B$-vertices, $W$-vertices, and $O$-vertices,
    \item  the preimage of an $O$-vertex consists only of $O$-vertices,
    \item  the preimage of a $V$-vertex consists only of $O$-vertices $V$-vertices, and $R$-vertices,
    \item  the preimage of an $R$-vertex consists only of $R$-vertices.
\end{enumerate}

\end{definition}

As expected, one studies collections of functors \[\Phi_S: \mathcal{T}^{5}(S) \rightarrow \mathrm{Top}\] together with coherence isomorphisms identifying $\Phi_S(T)$ as a product indexed over the vertices of $T$. These collections are evidently equivalent to the data of a reduced $(B,O)$-bimodule $W$ and a reduced $(O,R)$-bimodule $V$. We omit the specifics of such a definition, and from now on implicitly assume all bimodules are reduced.

There is a poset morphism $\mathcal{T}^5 (S) \rightarrow \mathcal{T}^{RBW}(S)$ which contracts the maximal subtrees which have vertices only of the colors $W,O,V$ to vertices of color $W$.

\begin{definition}
    Let $W$ be a $(B,O)$-bimodule and $V$ be a (reduced) $(O,R)$-bimodule. It is naturally represented by some collection of functors $\Phi_S: \mathcal{T}^5(S) \rightarrow \mathrm{Top}$. The relative composite $W \circ_O V$ is defined to be the $(B,R)$-bimodule which is represented by the left Kan extension of the $\Phi_S$ along $\mathcal{T}^5 (S) \rightarrow \mathcal{T}^{RBW}(S)$.
\end{definition}

To ensure that the relative composite is indeed a bimodule, we must assume that by $\mathrm{Top}$ we mean a convenient category of spaces such as the category of compactly generated, weak Hausdorff spaces. This is because a formal categorical argument shows that since $\mathcal{T}^{RBW}(S)$ has a maximal element given by $\bullet_W$, this left Kan extension evaluated at $\bullet_W$ is the space $\colim \Phi_S$. In a convenient category of spaces colimits commute with cartesian product individually in each variable, and so the coherence isomorphisms for decomposable trees can be checked.

As before, one may define what it means for the representing functors of a pair of bimodules $\Phi_S: \mathcal{T}^5(S) \rightarrow \mathrm{Top}$ to be stratified. The only modification is that the poset $\mathcal{P}_T$ appearing as the ``link'' of $T$ should be defined as the poset of all elements strictly greater than $T$. This is because the category of $5$-colored trees doesn't have a global maximum. It is immediate that such a stratification of the functors $\Phi_S$ is equivalent to a pair of stratifications of the individual bimodules which agree on the shared embedded operad.

\begin{thm}\label{thm: composition is manifold}
    If $W$ is a stratified $(B,O)$-bimodule, $V$ a stratified $(O,R)$-bimodule, and the induced stratifications of $O$ agree, then the relative composite $W \circ_O V$ has a canonical stratification as a $(B,R)$-bimodule. As a consequence, the relative composite of two collarable bimodule cobordisms over collarable manifold operads is a stratifiable bimodule cobordism.
\end{thm}

The definition of the stratifications (of operads, bimodules, etc.) we use consists of a collection of functorial data that satisfies various properties. There is thus a canonical candidate for the stratification of the relative composite $W \circ_O V$: left Kan extend the functorial stratification data associated to the pair of compatible bimodules represented by $\Phi_S: \mathcal{T}^5 \rightarrow \mathrm{Top}$ using the functor $\mathcal{T}^5(S) \rightarrow \mathcal{T}^{RBW}(S)$.

In order to verify the result is a stratified bimodule, we must then check that all of the properties in Definition \ref{stratifiedbimodule_dfn} hold. Of all of the properties, only property (\ref{RBWringmfld_item}) does not interact well with colimits. We omit what is largely a formal verification of these other properties. 

To prove property (\ref{RBWringmfld_item}) we will proceed analogously to the proof of Proposition \ref{link_prop} to show that $\colim \Phi_S$ is a manifold with boundary.

\begin{proof}

To inductively demonstrate that $\colim \Phi_S$ is a manifold with boundary, what needs to be verified is that the link of any stratum in $\colim \Phi_S$ is a CW-poset associated to a sphere or disc, depending on whether or not the stratum intersects the interior, of the appropriate dimension. Here the links are the posets $\mathcal{P}_T$ defined as the posets of trees $T'$ for which $T$ admits a nontrivial contraction to $T'$. A $5$-colored tree $T$ will correspond to the boundary of $\colim \Phi_S$ if it has vertices of color $B$ or $R$, and otherwise correspond to interior. An elementary study of chains in $\mathcal{T}^5(S)$ shows that, provided these links really are $CW$-posets, the dimensions of these spheres and discs must be as expected.

 Define a height function on a five-colored tree $T$ to be a map $h:V(T)\to \mathbb{R}$ such that
    \begin{itemize}
        \item for all $W$-vertices $v$ of $T$, $h(v)=-1$;
        \item for all $V$-vertices $v$ of $T$, $h(v)=1$;
        \item for all $R$-vertices $v$ of $T$, $h(v)\ge 1$;
        \item for all $B$-vertices $v$ of $T$, $h(v)\le -1$;
        \item for all $O$-vertices $v$ of $T$, $-1\le h(v) \le 1$;
        \item if $v$ is a child of $w$, then $h(v)\ge h(w)$. 
    \end{itemize}
    
Clearly, every 5-colored tree has a height function. Let
 $V^R(T)$ (resp. $V^B(T),V^O(T)$) be the set of $R$- (resp. $B$-, resp. $O$-) vertices of $T$. Define $\hat{V}(T)=V^R(T)\cup V^B(T) \cup V^O(T)$. 
Define the following subset of $\mathbb{R}^{\hat{V}(T)}$
$$X(T):=\big\{(x_v)_{v\in \hat{V}(T)}\in \mathbb{R}^{\hat{V}(T)}\,\big|\,\exists \textnormal{ height function }h \textnormal{ such that }h(v)=x_v, \forall v\big\}.$$

The set of height functions for $T$ is in canonical bijection with the polytope $X(T) \subset \mathbb{R}^{\hat{V}(T)}$ given by

$$\{x_v\ge1\}_{v\in V^R(T)}, \{x_v\le-1\}_{v\in V^B(T)}, \{-1 \le x_v\le 1\}_{v\in V^O(T)}, \ \{x_v\ge x_w\}_{v,w\in \hat{V}(T),\ v\textnormal{ is a child of }w}.$$

An argument identical to one appearing in Proposition \ref{link_prop} shows that there is an isomorphism \[\mathcal{F}(X(T)) \cong \mathcal{P}_T^\mathrm{op}.\]

Observe that the polytope $X(T)$ decomposes as a product of the polytopes 
\begin{itemize}
    \item $X_R(T) \subset \mathbb{R}^{V^R(T)}: \{x_v\ge1\}_{v\in V^R(T)},\{x_v\ge x_w\}_{v,w\in V^R(T),\ v\textnormal{ is a child of }w}$
        \item $X_B(T) \subset \mathbb{R}^{V^B(T)}: \{x_v\le -1\}_{v\in V^B(T)},\{x_v\ge x_w\}_{v,w\in V^B(T),\ v\textnormal{ is a child of }w}$
   \item $X_O(T) \subset \mathbb{R}^{V^O(T)}: \{-1 \le x_v\le1\}_{v\in V^O(T)},\{x_v\ge x_w\}_{v,w\in V^O(T),\ v\textnormal{ is a child of }w}$
\end{itemize}

This induces a poset isomorphism
\[\mathcal{F}(X(T)) \cong \mathcal{F}(X_R(T)) \ast  \mathcal{F}(X_B(T)) \ast \mathcal{F}(X_O(T)),\]
and so to finish the proof we should show the first two factors are CW-posets associated to discs, and the third is a CW-poset associated to a sphere.

In the first case, we can renormalize so that $x_v \geq 0$ upon which $X_R(T)$ becomes a cone. Intersecting with the hyperplane $\sum_{v\in V^R(T)} x_v=1$, we are left with a bounded polytope with face poset $\mathcal{F}(X_R(T)) \setminus \{\mathrm{cone \: point}\}$. By the existence of the dual polytope, the opposite of $\mathcal{F}(X_R(T)) \setminus \{\mathrm{cone \: point}\}$ is the CW-poset of a sphere. This implies $\mathcal{F}(X_R(T))^\mathrm{op} $ is the CW-poset associated to a disc. The argument for the second case is identical.

In the third case, we have that $X_O(T)$ is a bounded polytope. This means we can directly apply the existence of the dual polytope to deduce that  $\mathcal{F}(X_O(T))^\mathrm{op}$ is the CW-poset associated to a sphere, finishing the analysis of the links.
    
\end{proof}

\subsection{Reversing right nullbordisms} \label{section: reverse}

In this short section, we give the proof of Theorem \ref{thm: swap} which asserts that a right nullbordism induces a canonical left nullbordism. Recall the poset introduced in Section \ref{section: bimodule cobordisms}:

\begin{definition}
Let $\mathcal{T}_\mathrm{local}^{RW}(S)$ be the category 
with objects $S$-labeled trees with vertices of color $R$ or $W$, with no requirement of legality as for $RBW$-trees, and with morphisms tree contractions $f: T\to T'$ such that for each vertex $v$ of $T'$, $f^{-1}(v)\to v$ is a $RBW$-contraction of $\overline{cld}(v)$-labeled trees. 
\end{definition}

As in our three previous situations, once we require the data of product decompositions, collections of functors
\[\Phi_S: \mathcal{T}_\mathrm{local}^{RW}(S) \rightarrow \mathrm{Top}\]
are equivalent to reduced $(1,O)$-bimodules, though presented in an unexpected way. Here $O$ is the operad determined by the restriction of the collection $\Phi$ to the subcategory of red trees. As before, there is a notion of stratified functors out of this category, and it is equivalent to stratified right nullbordisms.

Surprisingly, there is an interesting functor from the poset $\mathcal{T}_\mathrm{local}^{RW}(S) $ to $\mathcal{T}^{RBW}(S)$ which lands in the subcategory of legal $RBW$-trees with only white and \textit{blue} vertices. Namely take a given $T \in \mathcal{T}_\mathrm{local}^{RW}(S) $ and set the color of the red vertices blue, then search for the largest blue subtree containing the root and contract all the subtrees which make up its complement into white vertices.

\begin{definition}
    Given a right $O$-module $W$ represented by $\Phi_S: \mathcal{T}_\mathrm{local}^{RW}(S) \rightarrow \mathrm{Top}$, the left Kan extension along the functors $\mathcal{T}_\mathrm{local}^{RW}(S)\rightarrow \mathcal{T}^{RBW}(S)$ defines the \textit{dual} $(O,1)$-bimodule $W^*$.
\end{definition}

The following implies Theorem \ref{thm: swap}.

\begin{prop}
    If $W$ is a stratified right nullbordism of $O$, then the dual $W^*$ is canonically a stratified left nullbordism of $O$.
\end{prop}
As in the previous section, there is a canonical candidate for this stratification obtained via left Kan extension. The crucial step in showing it is indeed a stratification is the assertion that the poset of elements $\mathcal{P}_T$ strictly greater than some $T \in \mathcal{T}_\mathrm{local}^{RW}(S)$ is the CW poset of either a sphere or disc of the appropriate dimension, depending on if $T$ will correspond to a stratum in the interior or on the boundary. 
\begin{proof}
    First, if $T$ has more than $1$ vertex we may assume without loss of generality that it has no external white vertices (i.e. a white vertex with a label attached). This is because in the poset $\mathcal{T}_\mathrm{local}^{RW}(S)$ an external white vertex cannot be involved in an (uncolored) contraction or have its color change to red. 

    Now take a red external vertex $v$ which is adjacent to the edge $e_v$ of $T$; we will relate the link of $T$ to the link of the tree 
    \[T':=T \setminus \{v,e_v\}.\]

    Any contraction of $T'$ gives rise to a contraction of $T$ by either $(1)$ doing nothing or $(2)$ contracting $e_v$. The former is strictly less than the latter, and so together this gives an embedding 
    \[\mathrm{Cone}(\mathcal{P}_{T'}) \rightarrow \mathcal{P}_T.\]

    However, this is not the only such embedding: another option is to change $v$ from red to white. This provides a second embedding, and these two embedded  cones intersect exactly on their base. Thus, if we know $\mathcal{P}_{T'}$ is a CW-poset then $\mathcal{P}_T$ will be the CW-poset corresponding to the suspension with the standard cell structure.

 By iteratively removing external red and white vertices, we are reduced to the case of a single vertex of either white or red color. The single white vertex has empty link. The suspension of the empty CW-complex is a sphere, and so we see that the link of $T$ is a sphere if the root is white.
 
The single red vertex has its link a point, and so if the root of $T$ is red then the link is homeomorphic to a disc. Since length of chains in $\mathcal{T}_\mathrm{local}^{RW}(S)$ may be identified with the difference in number of red vertices, we conclude that these discs and spheres are of the correct dimension.
\end{proof}

\section{Constructing manifold operads}\label{section: surgery on manifold operads} 

In this section, we develop cut-and-paste techniques for the construction of manifold operads. We refer to this as surgery. Surgery on operads is developed in two stages. First, we describe surgery on the top arity of a truncated operad and its associated truncated bimodule cobordism trace. In the following section, we describe how to canonically propagate surgeries and their traces up arity.

\subsection{Surgery on the top arity}
  Up until this point, we have only given one example of a $d$-dimensional manifold operad, the Fulton-MacPherson operad $\mathrm{FM}_d$. However, $n$-truncated manifold operads are plentiful and can be obtained by truncating the Fulton-MacPherson operad and performing a series of equivariant surgeries on $\mathring{\mathrm{FM}_d}(n)$, leaving us with an $n$-truncated manifold operad which agrees with $\mathrm{FM}_d$ up to arity $n-1$, but uses a different resolution of the obstruction of Proposition \ref{prop: operad obstruction}. As the $n$th space of operations is obtained by surgery, one might expect that the resulting operad is bimodule cobordant to the Fulton-MacPherson operad, and this is indeed true. We fix some finite $n< \infty$. Suppose we are supplied with a $RBW$ $\Sigma_n$-manifold $W$ from $O(n)$ to $M$ which satisfies 
  \begin{enumerate}
  \item  \label{surgerycollaritem1} $O(n)$ has a $\Sigma_n$-equivariant collar;
  \item \label{surgerycollaritem2} $W$ has a $\Sigma_n$-equivariant collar;
      \item \label{surgerycollaritem3} $M$ has a $\Sigma_n$-equivariant collar;
      \item \label{surgerycollaritem4}  $W$ is equipped with an identification 
\[f: \partial W \setminus (\mathring O(n) \sqcup \mathring M) \cong \partial O(n) \times [0,1],\]
and this restricts to the identity on $\partial O(n)$.
\item \label{surgerycollaritem5} The trace of the surgery $(O(n)\times[0,1]) \cup_{O(n)} W $ admits an equivariant collar.
  \end{enumerate}

Define $\sim$ to be the equivalence relation on $W$ generated by 
\[p \sim p' \iff p,p' \in  \partial W \setminus (\mathring O(n) \sqcup \mathring M) \: \text{and}\: pr_1(f(p))=pr_1(f(p')),\]
where $pr_1$ is the projection of $\partial O(n)\times[0,1]$ to the first factor composed with $f$. 
Let $\bar{W}:= W/\sim$. Note that the quotient map embeds $O(n)$ and $M$ into $\bar{W}$, intersecting along their boundary. By construction, there is a canonical equivariant homeomorphism $\partial (W/\sim) \cong O(n) \cup_{\partial O(n)} M $

Our assumptions imply 

\begin{lem}

There is an isomorphism of $\, \Sigma_n$-equivariant pairs of spaces $(W,O(n)) \cong (\bar{W},O(n))$.
\end{lem}

\begin{prop}
    Suppose $O$ is a collarable $n$-truncated  manifold operad with $W$ and $M$ as above. There is a canonical choice of $n$-truncated collarable manifold operad  $S$ which satisfies
\[
  S(m) \cong
  \begin{cases}
    O(m) & m <n\\
   M  & m =n.
  \end{cases}
\]

Additionally the $n$-truncated symmetric sequence $T$
\[
  T(m) \cong
  \begin{cases}
    O(m) \times [0,1] & m <n\\
   W  & m =n.
  \end{cases}
\]
  has both an $(O,S)$-bimodule structure and a $(S,O)$-bimodule structure that make it into a collarable $n$-truncated bimodule cobordism.

\end{prop}

\begin{proof}
    By hypothesis, there is an equivariant homeomorphism $\partial M \cong \partial O(n) = \mathrm{Decom}(O)(n)$, and so since $S^{\leq n-1} = O ^{\leq n-1}$ we can define the truncated operad composites of $S$ to agree with $O$ if the target arity is $<n$ and
    to factor through the given isomorphism and boundary inclusion if the target arity is $n$.
    By construction, the result is an $n$-truncated manifold operad which is collarable by hypothesis.

We now define the relevant truncated bimodules which we will call $T_1$ and $T_2$. 

We define \[
  T_1(m) \cong
  \begin{cases}
    W_{[0,1]}(O)(m) & m <n\\
   W_{[0,1]}(O)(n) \cup_{O \times \{0\}} W  & m =n.
  \end{cases}
\]
 There is a reduced $(S,W(O))$-bimodule structure on this truncated symmetric sequence: note that the reduced bimodule decomposables of $W_{[0,1]}(O)$ still embed into this $n$-truncated symmetric sequence, and so we only need to describe a suitable  cospan $M \rightarrow T_1(n) \leftarrow W(O)(n)$. The map $T_1(n) \leftarrow W(O)(n)$ is defined as before embedding at $W(O) \times \{1\}$. The map $M \rightarrow T_1(n)$ is then defined using the isomorphism mentioned above
 \[M \hookrightarrow  W_{[0,1]}(O)(n) \cup_{O \times \{0\}} \bar{W} \cong W_{[0,1]}(O)(n) \cup_{O \times \{0\}} W =: T_1(n).\]

 By Proposition \ref{prop: isomorphic to w construction}, $T_1 \cong T$ as symmetric sequences. The introduction of the reparametrization using $W \cong \bar{W}$ ensures that the resulting bimodule satisfies the property that the reduced bimodule decomposables form an embedded cobordism between $\partial O (n)$ and $\partial S(n)$. This truncated bimodule cobordism is collarable by hypothesis (\ref{surgerycollaritem5}).

Similarly, one defines 
\[
  T_2(m) \cong
  \begin{cases}
    W_{[0,1]}(O)(m) & m <n\\
   W_{[0,1]}(O)(n) \cup_{W(O) \times \{1\}} W  & m =n.
  \end{cases}
\]
where the gluing uses the operad isomorphism $O \cong W(O)$ in arity $n$ guaranteed by Proposition \ref{prop: isomorphic to w construction}. The argument that this can be made a collarable $(O,S)$-bimodule cobordism is as before.
\end{proof}

\subsection{The arity extension of truncated bimodule cobordisms}\label{X_sec}

In this section, we discuss how to extend truncated bimodule cobordisms by arity given an extension of either the truncated operad acting on the left or the truncated operad acting on the right.  

\subsubsection{Left surgery}
  Fix $l< \infty$ and $n \leq \infty$ such that $l \leq n$. Let $R$ be an $n$-truncated manifold operad.

\begin{thm}\label{thm: left surgery}
    Given a stratifiable $l$-truncated bimodule cobordism $W$ from $B$ to $R^{\leq l}$, there is a functor $\phi: \mathcal{T}^{RBW}_{\leq n} \rightarrow \mathrm{Top}$ determining a stratifiable $n$-truncated, $d$-dimensional bimodule cobordism, which we call $X_L(W)$, from the $n$-truncated manifold operad $\partial_L X_L(W)$ to $R$. We have that $X_L(W)$ is collarable if $W$ is collarable, and  $\partial_L X_L(W)$ is collarable if $B$ is collarable.
    
    The functor $\phi$ is defined by the conditions:
\begin{enumerate}
\item $\phi: \mathcal{T}^{RBW}_{\leq l}  \rightarrow \mathrm{Top}$ represents the reduced bimodule $W$.
    \item $\phi: \mathcal{T}^{RBW}_{\leq n} \rightarrow \mathrm{Top}$ represents a reduced bimodule whose right action is through the $n$-truncated operad $R$.
    
    \item  If $|S|>l$, then
    \[\phi(\bullet_{W,S}) =X_L(W)(S):=
    (\colim_{\mathcal{T}^{RBW}(S)\backslash\{\bullet_{ W ,S},\bullet_{B,S}\}}\phi \cup \text{external collar})   \times [0,1].\]
    \item If $|S|>l$, then \[\phi (\bullet_{B,S})=\partial_LX_L(W)(S):=\overline{ \partial X_L(W)(S) \setminus ((\colim_{\mathcal{T}^{RBW}(S)\backslash\{\bullet_{ W ,S},\bullet_{B,S}\}}\phi) \times \{1\})}.\]

    \item If $|S|>l$, then the values $\phi(\bullet_{B,S})\rightarrow \phi(\bullet_{W,S}) \leftarrow \phi(\bullet_{R,S})$ are given by the inclusions. 
    \[\begin{tikzpicture}
    \path (0,0) coordinate (A0) -- (1,0) coordinate (B0) -- (1.7,0.5) coordinate (C0) -- (1.2,1) coordinate (D0) -- (0.3,1) coordinate (E0) -- (-0.3,0.5) coordinate (F0) -- cycle;
    \draw ($(A0)+(0,1.5)$) coordinate (A1) -- ($(B0)+(0,1.5)$) coordinate (B1) -- ($(C0)+(0,1.5)$) coordinate (C1) -- ($(D0)+(0,1.5)$) coordinate (D1) -- ($(E0)+(0,1.5)$) coordinate (E1) -- ($(F0)+(0,1.5)$) coordinate (F1) -- cycle;
    \draw [fill=red,opacity=0.5] (0.2,1.8) coordinate (Ai) -- (0.9,1.8) coordinate (Bi) -- (1.3,2.06) coordinate (Ci) -- (1.1,2.25) coordinate (Di) -- (0.4,2.25) coordinate (Ei) -- (0.05,2) coordinate (Fi) -- cycle; 
    \draw ($1.4*(A1)-0.4*(Ai)$) coordinate (A11) -- ($1.4*(B1)-0.4*(Bi)$) coordinate (B11) -- ($1.4*(C1)-0.4*(Ci)$) coordinate (C11) -- ($1.4*(D1)-0.4*(Di)$) coordinate (D11) -- ($1.4*(E1)-0.4*(Ei)$) coordinate (E11) -- ($1.4*(F1)-0.4*(Fi)$) coordinate (F11) -- cycle;  
    \draw ($(A11)+(0,-1.5)$) coordinate (A00) -- ($(B11)+(0,-1.5)$) coordinate (B00) -- ($(C11)+(0,-1.5)$) coordinate (C00) -- ($(D11)+(0,-1.5)$) coordinate (D00) -- ($(E11)+(0,-1.5)$) coordinate (E00) -- ($(F11)+(0,-1.5)$) coordinate (F00) -- cycle;
    \draw (A00)--(A11); 
    \draw (B00)--(B11);
    \draw (C00)--(C11);
    \draw (D00)--(D11);
    \draw (E00)--(E11);
    \draw (F00)--(F11);
    \draw[fill=blue, opacity=0.2] (F00)--(A00)--(B00)--(C00)--(C11)--(D11)--(E11)--(F11)--cycle;  
    \path [fill=green, opacity=0.3, even odd rule]
    (A1)--(B1)--(C1)--(D1)--(E1)--(F1)--cycle
    (Ai)--(Bi)--(Ci)--(Di)--(Ei)--(Fi)--cycle;
    \draw [<-,red] (1,2) -- (2,2.5);
    \node [right, red] at (2,2.5) {\small $R(S)$};
    \draw [->, seagreen] (-0.5,1.7)--(0,1.7); 
    \node [left, seagreen] at (-0.5,1.7) {\small $\colim_{\mathcal{P}_{\mathcal{T}^{RBW}(S)\backslash\{\bullet_{W,S},\bullet_{B,S},\bullet_{R,S}\}}}\phi$};
    \draw [<-,blue] (1.6,0.4)--(2,0.4);
    \node [right,blue] at (2,0.4) {\small $\begin{aligned}\partial_L X_L(W)(S)\end{aligned}$};
    \node at (0.7,0.7) {\small $ X_L(W) (S)$};
\end{tikzpicture}\]
    \item The values of $\phi$ on all other contractions with codomain $\bullet_{W,S}$ factor through \[\colim_{\mathcal{T}^{RBW}(S)\backslash\{\bullet_{ W ,S},\bullet_{B,S},\bullet_{R,S}\}}\phi\] via inclusion at $\{1\}$. \
    
    \item The value of $\phi$ on contractions with target $\bullet_{B,S}$ factors through
    \[\colim_{\mathcal{T}^{B}(S)\backslash\{\bullet_{B,S}\}}\phi \subset \partial \colim_{\mathcal{T}^{RBW}(S)\backslash\{\bullet_{ W ,S},\bullet_{B,S},\bullet_{R,S}\}}\phi\]
    via inclusion at $\{1\}$.
    \item The value of $\phi$ on label changes for either $\bullet_B$ or $\bullet_W$ is induced by the natural action on the colimits.
\end{enumerate}

\end{thm}

\begin{proof}
     The functor $\phi$ is well defined provided we know that $\colim_{\mathcal{T}^{RBW}(S)\backslash\{\bullet_{ W ,S},\bullet_{B,S}\}}\phi$ is a manifold, and so can refer to its boundary to define $\phi(\bullet_{B,S})$. This assertion as well as the fact that the bimodule is an $n$-truncated bimodule cobordism between $n$-truncated manifold operads follows from demonstrating that $X_L(W)$ is an $n$-truncated stratified bimodule by Theorem \ref{thm: bimodule cobordism is stratified}.

Supposing the stratification $(\mathcal{N},\nu)$ has been defined up to cardinality $N$, we will extend the stratification to contractions of $S$-labeled trees where $|S|=N+1$. The stratification is automatically defined for contractions where the codomain is a decomposable tree, as in Step 1 in the proof of Theorem \ref{thm: bimodule cobordism is stratified}. The three remaining cases are the contractions to a tree with a single vertex. To define them in the case of $\bullet_{B,S}$, it suffices to prove the following lemma.

\begin{lem}
    For every finite set $S$ with $|S|\leq N+1$, the pair

   \[( \colim_{{\mathcal{T}^{RBW}_{S}\backslash\{\bullet_{ W ,S},\bullet_{B,S}\}}}\phi,\colim_{{\mathcal{T}^{B}_{S}\backslash\{{\bullet_{B,S}\}}}}\phi)\]
   is a topological manifold with its boundary. For every $T\in \mathcal{T}^{RBW}(S)\backslash\{\bullet_{ W },\bullet_{B}\}$, $\mathring\phi(T)$ has an open neighborhood $\mathcal{N}^\partial_T$ in $\colim_{{\mathcal{T}^{RBW}_{S}\backslash\{\bullet_{ W ,S},\bullet_{B,S}\}}}\phi$ homeomorphic to 
    \begin{itemize}
        \item $\mathring{\textnormal{Cone}}\big(\partial Lk(T)\big)\times \mathring{\phi}(T)$, if $T\notin \mathcal{T}^B(S)$;
        \item $\mathring{\textnormal{Cone}}\big(\partial Lk(T)\backslash\textnormal{the cell corresponding to }T\to \bullet_B\big)\times \mathring{\phi}(T)$, if $T\in \mathcal{T}^B(S)$. 
    \end{itemize}
    This homeomorphism is natural up to reparametrization of cells.
\end{lem}
\begin{proof}
We first show that if $T \in \mathcal{T}^B(S)$,
\[\partial Lk(T)\backslash\textnormal{the cell corresponding to }T\to \bullet_B\] is a disc. The argument is as follows:
$\partial Lk(T)$ is a regular CW-complex homeomorphic to a sphere by Proposition \ref{link_prop}, and the cell $T\to \bullet_B$ is a sub-CW-complex; taking barycentric subdivision, $\partial Lk(T)$ is a simplicial complex and the cell $T\to\bullet_B$ is a subcomplex which is also a disc of the same dimension. 
By \cite[Corollary 3.13]{PLtopology}, in the piecewise-linear setting, the complement of a same-dimensional disc in a sphere is always a disc. 

Let $T\in \mathcal{T}^{RBW}_{N+1}\backslash\{\bullet_{W},\bullet_B\}$. Define $\mathcal{T}_{ \geq T,\neq\bullet_{B},\bullet_{W}}$ to be the subcategory of $\mathcal{T}^{RBW}_{N+1}\backslash\{\bullet_{W},\bullet_B\}$ consisting of trees $T'$ such that there exists a possibly trivial contraction $T\to T'$. 

Define $\mathcal{N}^\partial_T\subset \colim_{{\mathcal{T}^{RBW}_{S}\backslash\{\bullet_{ W ,S},\bullet_{B,S}\}}}\phi$ to be 
$$\mathcal{N}^\partial_{T}=\bigcup_{T'\in \mathcal{T}_{ \geq T,\neq\bullet_{B},\bullet_{W}}}\mathcal{N}_{T\to T'}.$$
Then, 
\begin{align*}
    \mathcal{N}^\partial_{T}&=\colim_{T'\in \mathcal{T}_{ \geq T,\neq\bullet_{B},\bullet_{W}}} (\mathcal{N}_{T\to T'},\subset)\\
    &\stackrel{\nu}{=}\colim_{T'\in \mathcal{T}_{ \geq T,\neq\bullet_{B},\bullet_{W}}}
    \big(\mathring{\phi}(T)\times\mathring{\textnormal{Cone}}(Lk(T\to T'))\big)\\
    &=\mathring{\phi}(T)\times \colim_{T'\in \mathcal{T}_{ \geq T,\neq\bullet_{B},\bullet_{W}}}
    \mathring{\textnormal{Cone}}(Lk(T\to T'))\\
    &=\mathring{\phi}(T)\times \mathring{\textnormal{Cone}}\big(\colim_{T'\in \mathcal{T}_{ \geq T,\neq\bullet_{B},\bullet_{W}}}
    Lk(T\to T')\big),
\end{align*}

By Lemma \ref{CWcolim_lmm}, we can compute the colimit of regular CW-complexes by passing to the associated face posets. One computes
$$\colim_{T'\in \mathcal{T}_{ \geq T,\neq\bullet_{B},\bullet_{W}}}
    Lk(T\to T')=\begin{cases}
        \partial Lk(T),\textnormal{ if }T\notin\mathcal{T}^B,\\
        \partial Lk(T)\backslash(T\to\bullet_B), \textnormal{ if }T\in \mathcal{T}^B. 
    \end{cases}$$ 
\end{proof}

This defines the stratified structure for the contractions which don't have targets  $\bullet_{B,S}$ or $\bullet_{W,S}$. To construct the remaining contractions proceed as in \ref{step 3 of operad strat} of the proof of Theorem \ref{thm: manifold operad is stratified} and \ref{step 3 of bimodule strat} of the proof of Theorem \ref{thm: bimodule cobordism is stratified}, respectively.  We only need to verify that $\partial_L X_L (W)(S)$ and $ X_L (W)(S)$ have equivariant collars. The first has an external collar by construction, and the second is the product of two equivariant manifolds with equivariant collars.

\end{proof}

\subsubsection{Right surgery}

Fix $l< \infty$ and $n \leq \infty$ such that $l \leq n$. Let $B$ be an $n$-truncated manifold operad.

\begin{thm}\label{thm: right surgery}

      Given a stratifiable $l$-truncated bimodule cobordism $W$ from  $B^{\leq l}$ to $R$, there is a $\phi: \mathcal{T}^{RBW}_{\leq n} \rightarrow \mathrm{Top}$ determining a stratifiable $n$-truncated, $d$-dimensional bimodule cobordism, which we call $X_R(W)$, from $B$ to the $n$-truncated manifold operad $\partial_R X_R(W)$. $X_R(W)$ is collarable if $W$ is collarable, and $\partial_R X_R(W)$ is collarable if $R$ is collarable.
    
    The functor $\phi$ is defined by the conditions:

\begin{enumerate}
\item $\phi: \mathcal{T}^{RBW}_{\leq l}  \rightarrow \mathrm{Top}$ represents the reduced bimodule $W$.

    \item $\phi: \mathcal{T}^{RBW}_{\leq n} \rightarrow \mathrm{Top}$ represents a reduced bimodule whose left action is through the $n$-truncated operad $B$.

    \item  If $|S|>l$, then   \[\phi(\bullet_{W,S}) =X_R(W)(S):=(\colim_{\mathcal{T}^{RBW}(S)\backslash\{\bullet_{ W ,S},\bullet_{R,S}\}}\phi \cup \text{external collar}) \times [0,1]\]

    \item If $|S|>l$, then \[\phi (\bullet_{R,S})=\partial_RX_R(W)(S):= \overline{\partial X_R(W)(S) \setminus ((\colim_{\mathcal{T}^{RBW}(S)\backslash\{\bullet_{ W ,S},\bullet_{R,S}\}}\phi) \times \{1\})}  \]

    \item If $|S|>l$, the values $\phi(\bullet_{B,S})\rightarrow \phi(\bullet_{W,S}) \leftarrow \phi(\bullet_{R,S})$ are given by the inclusions.  
    \[\begin{tikzpicture}
    \path (0,0) coordinate (A0) -- (1,0) coordinate (B0) -- (1.7,0.5) coordinate (C0) -- (1.2,1) coordinate (D0) -- (0.3,1) coordinate (E0) -- (-0.3,0.5) coordinate (F0) -- cycle;
    \draw ($(A0)+(0,1.5)$) coordinate (A1) -- ($(B0)+(0,1.5)$) coordinate (B1) -- ($(C0)+(0,1.5)$) coordinate (C1) -- ($(D0)+(0,1.5)$) coordinate (D1) -- ($(E0)+(0,1.5)$) coordinate (E1) -- ($(F0)+(0,1.5)$) coordinate (F1) -- cycle;
    \draw [fill=blue,opacity=0.5] (0.2,1.8) coordinate (Ai) -- (0.9,1.8) coordinate (Bi) -- (1.3,2.06) coordinate (Ci) -- (1.1,2.25) coordinate (Di) -- (0.4,2.25) coordinate (Ei) -- (0.05,2) coordinate (Fi) -- cycle; 
    \draw ($1.4*(A1)-0.4*(Ai)$) coordinate (A11) -- ($1.4*(B1)-0.4*(Bi)$) coordinate (B11) -- ($1.4*(C1)-0.4*(Ci)$) coordinate (C11) -- ($1.4*(D1)-0.4*(Di)$) coordinate (D11) -- ($1.4*(E1)-0.4*(Ei)$) coordinate (E11) -- ($1.4*(F1)-0.4*(Fi)$) coordinate (F11) -- cycle;  
    \draw ($(A11)+(0,-1.5)$) coordinate (A00) -- ($(B11)+(0,-1.5)$) coordinate (B00) -- ($(C11)+(0,-1.5)$) coordinate (C00) -- ($(D11)+(0,-1.5)$) coordinate (D00) -- ($(E11)+(0,-1.5)$) coordinate (E00) -- ($(F11)+(0,-1.5)$) coordinate (F00) -- cycle;
    \draw (A00)--(A11); 
    \draw (B00)--(B11);
    \draw (C00)--(C11);
    \draw (D00)--(D11);
    \draw (E00)--(E11);
    \draw (F00)--(F11);
    \draw[fill=red, opacity=0.2] (F00)--(A00)--(B00)--(C00)--(C11)--(D11)--(E11)--(F11)--cycle;  
    \path [fill=green, opacity=0.3, even odd rule]
    (A1)--(B1)--(C1)--(D1)--(E1)--(F1)--cycle
    (Ai)--(Bi)--(Ci)--(Di)--(Ei)--(Fi)--cycle;
    \draw [<-,blue] (1,2) -- (2,2.5);
    \node [right, blue] at (2,2.5) {\small $B(S)$};
    \draw [->, seagreen] (-0.5,1.7)--(0,1.7); 
    \node [left, seagreen] at (-0.5,1.7) {\small $\colim_{\mathcal{P}_{\mathcal{T}^{RBW}(S)\backslash\{\bullet_{W,S},\bullet_{B,S},\bullet_{R,S}\}}}\phi$};
    \draw [<-,red] (1.6,0.4)--(2,0.4);
    \node [right,red] at (2,0.4) {\small $\begin{aligned}\partial_RX_R(W)(S)\end{aligned}$};
    \node at (0.7,0.7) {\small $ X_R(W) (S)$};
\end{tikzpicture}\]

    \item The values of $\phi$ on all other contractions with codomain $\bullet_{W,S}$ factor through \[\colim_{\mathcal{T}^{RBW}(S)\backslash\{\bullet_{ W ,S},\bullet_{B,S},\bullet_{R,S}\}}\phi\] via inclusion at $\{1\}$.

    \item The value of $\phi$ on contractions with target $\bullet_{R,S}$ factors through
    \[\colim_{\mathcal{T}^{R}(S)\backslash\{\bullet_{R,S}\}}\phi \subset \partial \colim_{\mathcal{T}^{RBW}(S)\backslash\{\bullet_{ W ,S},\bullet_{B,S},\bullet_{R,S}\}}\phi\]
    via inclusion at $\{1\}$.

    \item The value of $\phi$ on label changes for either $\bullet_R$ or $\bullet_W$ is induced by the natural action on the colimits.
\end{enumerate}

\end{thm}

\begin{proof}
     In the proof of Theorem \ref{thm: left surgery}, swap the roles of $R$ and $B$.
\end{proof}

\subsection{Basic applications of surgery}
Our main application of operadic surgery is the construction of manifold operads that are distinct from the Fulton-MacPherson operads.

\begin{thm}\label{thm: sufficient condition for existence}
    If $M$ is a $(d-1)$-dimensional $\Sigma_2$-manifold which is $\Sigma_2$-nullbordant, then there is a collarable manifold operad $O$ such that $O(2)=M$. Furthermore, if $M$ is connected and admits a connected $\Sigma_2$-cobordism $W$ to $S^{d-1}$ with the antipodal action, then the operad $O$ may be taken to be levelwise connected. The symmetric group actions on $O$ can be taken to be free if the $\Sigma_2$-action on the cobordism $W$ is free.

\end{thm} \label{thm: z2bord}
\begin{proof}
    If $M$ is $\Sigma_2$-nullbordant, then it is in particular $\Sigma_2$-cobordant to $S^{d-1}$ with the antipodal action since the latter is also $\Sigma_2$-nullbordant. By adding an external collar, we may assume this cobordism is collarable. It therefore gives a stratifiable $2$-truncated bimodule that we may apply either Theorem \ref{thm: left surgery} or Theorem \ref{thm: right surgery} to construct the desired operad.
    
    The connected version of the theorem follows from applying either Theorem \ref{thm: left surgery} or Theorem \ref{thm: right surgery} to the connected ($\Sigma_2$-free) cobordism and observing everything remains connected (and free).

\end{proof}

The thickened $W$-construction $W_{[0,1]}(O)$ can be interpreted as associating to a vertex $v$ and its unique outgoing edge a point of the trivial cobordism $O(\overline{cld}(v))\times [0,1]$. The $X$-construction generalizes the choice of labels to a general $n$-truncated bimodule cobordism. Let $O(2) \times [0,1]$ denote the $2$-truncated bimodule cobordism from $O^{\leq 2}$ to itself. Implicitly we consider $O(2)$ as the operad extending $O^{\leq 2}$.

\begin{prop}\label{prop: right trivial surgery}

    Let $n \leq \infty$ and suppose $O$ is an $n$-truncated stratifiable manifold operad. 
    
    There is an isomorphism of operads \[\partial_{R} X_R( O(2) \times [0,1])  \cong W(O),\] and a compatible isomorphism of $(O,W(O))$-bimodules
    \[X_R( O(2) \times [0,1]) \cong W_{[0,1]}(O).\]

\end{prop} 
\begin{proof}
Suppose $W_{[0,1]}(O)$ is represented by $\phi:\mathcal{T}^{RBW}_{\leq n} \rightarrow \mathrm{Top}$.
We proceed inductively by constructing the isomorphism on arities $N \leq n$. Note that the base case of $N=2$ is trivial.  Suppose we have inductively made the identifications \[\partial_{R} X_R( O(2) \times [0,1])^{\leq N} \cong W(O)^{\leq N},\] 
    \[X_R( O(2) \times [0,1]) ^{\leq N}\cong W_{[0,1]}(O)^{\leq N}.\]
By definition, $\partial_{R}X_R( O(2) \times [0,1])(N+1)$ is homeomorphic to the union of the value of $\phi$ at the following types of trees: \begin{itemize}
        \item the blue corolla,
        \item a blue root with at least one white vertex attached,
        \item a white root with a red vertex attached.
    \end{itemize}
together with an external collar.

We identify the values of $\phi$ on these $RBW$-trees with subspaces of $W_{[0,1]}(O)(N+1)$. The value of  $\phi$ on the blue corolla is the space $O(N+1)$, and we identify this with the space of  corollas labeled by $O(N+1)$ with a root length of $0$. Under $\phi$, the second is a product of the appropriate $O(-)$ with a copy of $W_{[0,1]}(O)(-)$ for each white vertex, by induction. We identify this product with the subspace of $W_{[0,1]}(O)(N+1)$ of decomposable metric trees with root length $0$. Under $\phi$ the third class of $RBW$-trees is sent to a product of $W_{[0,1]}(-)$ and a copy of $W(-)$, by induction. We identify this with the space of $O$-labeled metric trees which can be decomposed along a length $1$ edge with no restriction on the root length.

We have described it in this way, so that the union canonically embeds into \[W_{[0,1]}(O)(N+1)\cong W(O)(N+1) \times [0,1]\] as the subspace \[(W(O)(N+1)\times \{0\} )\cup (\partial W(O)(N+1) \times [0,1]).\]

Hence, $\partial_{R} X_R( O(2) \times [0,1])(N+1)$ with its boundary stratification is isomorphic to $W(O)$ with two external collars attached. Here one collar comes from the above description as a union, and the other via the construction itself. Since $W(O)(N+1)$ is equivariantly collarable, this completes the operad component of the induction step. A similar analysis holds for the bimodule structure.

\end{proof}

\begin{cor}\label{cor: w of stratified}
    If $O$ is an $n$-truncated stratified manifold operad, then $W(O)$ is an $n$-truncated manifold operad which is collarable, and $W_{[0,1]}(O)$ is a collarable bimodule $h$-cobordism from $O$ to $W(O)$.
\end{cor}
\begin{proof}
    Since $X_R(O(2) \times [0,1]) \cong W_{[0,1]}(O)$, Theorem \ref{thm: right surgery} shows $W_{[0,1]}(O)$ is a bimodule cobordism. Its collar is described in Proposition \ref{prop: thick W bimodule}, and it is classical that the relevant inclusions are homotopy equivalences.
\end{proof}

In contrast to Proposition \ref{prop: right trivial surgery} which exhibits the ``trivial'' right surgery as the $W$-construction, it turns out that the ``trivial'' left surgery is more-or-less the inverse of the $W$-construction. Let $W(O)(2) \times [0,1]$ denote the $2$-truncated bimodule cobordism from $W(O)^{\leq 2}$ to itself (which is of course homeomorphic to $O(2) \times [0,1]$). Implicitly we consider $W(O)$ as the operad extending $W(O)^{\leq 2}$.

\begin{prop} \label{prop: left trivial surgery}
      Let $n\leq \infty$ and suppose $O$ is an $n$-truncated collarable manifold operad. There is an isomorphism of operads \[\partial_{L} X_L( W(O)(2) \times [0,1])  \cong O,\] and a compatible isomorphism of $(O,W(O))$-bimodules

    \[X_L(W(O)(2) \times [0,1]) \cong W_{[0,1]}(O).\]
\end{prop}

\begin{proof}
We proceed as in the case of right surgery, noting that the base case of $N=2$ is trivial. Suppose we have inductively made the identifications 
  \[\partial_{L} X_L( W(O)(2) \times [0,1])^{\leq N}  \cong O^{\leq N},\] 
    \[X_L( W(O)(2) \times [0,1])^{\leq N}  \cong W_{[0,1]}(O)^{\leq N}.\]
    
    By definition, $\partial_{L}X_L( W(O)(2) \times [0,1])(N+1) $  is the union of the value of $\phi$ at the following types of trees:
\begin{itemize}
\item the red corolla,
        \item a blue root with at least one $W$ vertex attached directly to it,
        \item a white root with a red vertex attached to it.
\end{itemize}
together with an external collar.

   We identify the values of $\phi$ on these $RBW$-trees with subspaces of $W_{[0,1]}(O)(N+1)$. By definition, the value of $\phi$ at the red corolla is $W(O)(N+1)$, and we declare these trees as having root edge length $1$. Using the inductive hypothesis, we then identify the value of $\phi$ at the second type as consisting of decomposable trees which we declare have root length $0$. Using the inductive hypothesis once more, we identify the value at the third type with trees decomposable along a length $1$ edge and with arbitrary root length. As before, this canonically embeds in \[W_{[0,1]}(O)(N+1)\cong W(O)(N+1) \times [0,1],\] but this time as the union of the three subspaces \[W(O)(N+1) \times \{1\},\] \[\partial W(O)(N+1) \times [0,1],\] \[W(O)(N+1) \times \{0\} \setminus \mathring O(N+1) \times \{0\}.\] Since $O$ has an equivariant collar, this last subspace is a collar neighborhood of $\partial W(O)(N+1) \times \{0\} \subset W(O)(N+1) \times \{0\}$ by Proposition \ref{prop: isomorphic to w construction}.

   Hence, $\partial_{L}X_L( W(O)(2) \times [0,1])(N+1)$ with its boundary stratification is isomorphic to $W(O)(N+1)$ with two collars attached.
 Since $O$ is assumed to be collarable, this is isomorphic in a stratified way to $O(N+1)$ by Proposition \ref{prop: isomorphic to w construction}, and we have completed the operad component of the induction step. A similar analysis holds for the bimodule structure.

\end{proof}

\begin{remark}
    In either Proposition \ref{prop: right trivial surgery} or Proposition \ref{prop: left trivial surgery} one can replace the $2$-truncated bimodule cobordisms with $W_{[0,1]}(O)^{\leq l}$ for any $l$, and the conclusions remain true with identical proofs.
\end{remark}

\subsection{Examples of surgery} \label{subsection: examples}
\begin{ex} \label{ex: torus}
 Identify $S^{d-1}$ with the boundary of the product of discs $\del(D_m \times D_{d-m})=(S^{m-1} \times D_{d-m} \cup D_m \times S^{d-m-1})$, where $m \in \{0,\dots,d\}$ 
and $\Z/2$ acts everywhere by the antipodal action. Removing $S^{m-1} \times D_{d-m}$
and replacing it by a copy of $D_m \times S^{d-m-1}$ yields 
the $\Z/2$-manifold $S^m \times S^{d-m-1}$. The trace of the surgery $W$ is obtained by gluing $D_m \times D_{d-m}$ and
$S^{d-1} \times [0,1]$ along $S^{m-1} \times D_{d-m} \times \{1\}$ and is 
a $\Z/2$-cobordism from $S^{d-1}$ to $S^m \times S^{d-m-1}$. The $\Z/2$-action on $W$ is not free, but if $m<d-1$ then Theorem \ref{thm: sufficient condition for existence} guarantees there is a levelwise path-connected manifold operad $O$ such that $O(2)= S^m \times S^{d-m-1}$. 
\end{ex}

The next two examples show how left and right surgery differ when applied to the same cobordism on $O(2)$.

\begin{ex}\label{542_example}
Consider the above example when $m=d$. Then the cobordism is $W(2)=D_d$, the disc with the antipodal $\Z/2$-action thought of as a bordism between $S^{d-1}=\FM_d(2)$ and 
the empty manifold. In the case of left surgery, choose
$R=\FM_d$ and $B(2)=\emptyset$ in Theorem \ref{thm: left surgery} and write
$B(3):=\partial_L X_L(W)(3)$.
Then $B(3)$ is the configuration space of three distinct points in $S^d \cong\mathbb{R}^d \cup \{\infty\}$ modulo translations and positive dilations.
This space is homeomorphic to $STS^d$, the unit tangent bundle of $S^d$.
Namely given $(x,y) \in STS^d$ with 
$|x|=|y|=1$ and $x \cdot y=0 $, the map
$(x,y) \mapsto [x,y,-y]$ is a homeomorphism $STS^d \cong B(3)$.
\end{ex}

\begin{ex} \label{ex: null}
Now instead of the left surgery to the empty set, consider the right surgery. We claim the left nullbordism from $\FM_d$ to the trivial operad $1$ introduced in Proposition \ref{leftnullbordant} can be obtained by a countable sequence of right surgeries, one for each arity $l \geq 2$.
Let us choose $B=\FM_d$ in Theorem \ref{thm: right surgery}. The first step is for $l=2$: choose $W=D_d$ as a $\Z/2$-cobordism between $\FM_d(2)=S^{d-1}$ and the empty manifold as in the previous example. The result of the right surgery is an operad $O_2$ such that $O_2(2)=\emptyset$ and $O_2(3) \cong S^{2d-1}$,  identified to the space of non-diagonal $3$-tuples in $\mathbb{R}^d$, modulo translations and positive scalings, with the natural $\Sigma_3$-action. There is a $3$-truncated bimodule cobordism between $O_2$ and $1$, induced by the cone $S^{2d-1} \subset D_{2d}$. The corresponding right surgery produces an operad $O_3$ such that $O_3$ is empty in arity $2,3$, and $O_3(4) \cong S^{3d-1}$. In general $O_l$ is empty between arity 2 and $l$, and $O_l(l+1) \cong S^{ld-1}$.
\end{ex}

\begin{ex}
    From a nonequivariant cobordism between $S^{d-1}$ and $M$ there are several natural ways to produce a $\mathbb{Z}/2$-cobordism between $S^{d-1}$ with the antipodal action and $M \sharp M$.  After choosing one, we may apply the surgery extension theorems then to show, e.g., the genus $4$ surface 
\begin{center}
    \includegraphics[scale=0.2]{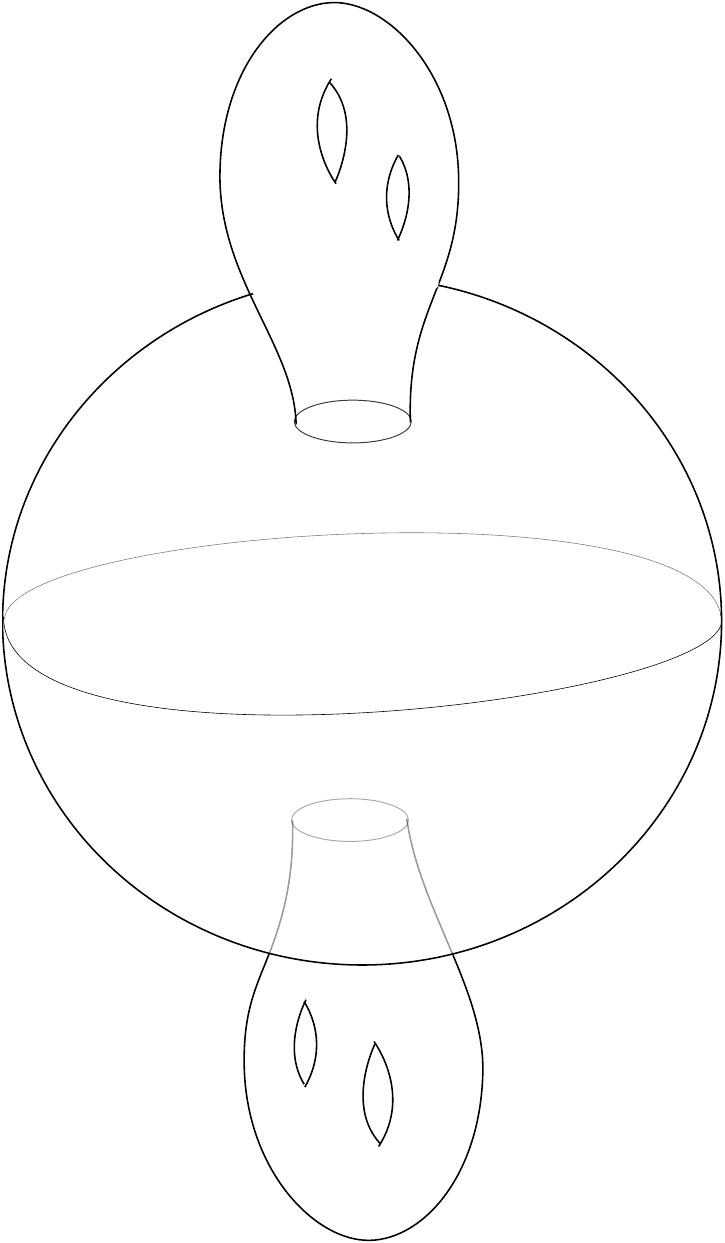}
\end{center}
with its antipodal action can be taken as the space of $2$-ary operations of a manifold operad.
\end{ex}

\begin{ex}
Fix a nontrivial homology sphere $X$ homology cobordant
to $S^{d-1}$ (these always exist in sufficiently high dimension, say $d-1\geq 6$). With a little care, the equivariant cobordism mentioned in the above example between $S^{d-1}$ and $X \sharp X$ can be chosen to be a homology $h$-cobordism. It is not difficult to see that the bimodule cobordism extension construction preserves homology bimodule $h$-cobordisms, and so by Remark \ref{remark: homology h cob} the output is a manifold operad $O$ such that \[H_\ast(O) \cong H_\ast (\mathrm{FM}_d),\] but $O(2) \cong X \sharp X$. In particular, $H_\ast(O)$ is $\mathrm{Pois}_{d-1}$, the operad governing Poisson algebras with a Lie bracket of degree $d-1$, but the fundamental group of $O(2)$ is the free product $\pi_1(X) \ast \pi_1(X)$ (since $d-1>2$).
\end{ex}

\begin{ex}
    Let $O$ be a manifold operad such that $O(2)$ admits a nontrivial collarable $\mathbb{Z}/2$-equivariant $h$-cobordism $W$ to a nonhomeomorphic manifold (such operads exist as a consequence of Theorem \ref{thm: sufficient condition for existence}), then applying Theorem \ref{thm: left surgery} or Theorem \ref{thm: right surgery} to $W$ gives examples of bimodule $h$-cobordisms on or under $O$
    which are not isomorphic to the thickened $W$-construction. In particular, by Corollary \ref{cor: h cobordism implies equiv} we can conclude that the surgered manifold operad is weakly homotopy equivalent as an operad to $O$, but is distinguished by the homeomorphism type of its space of $2$-ary operations.
\end{ex}

\subsection{Examples of algebras and modules over manifold operads}
\label{subsection: examples of algebras}

\begin{ex}
Let $W$ be an $(O,\FM_d)$-bimodule cobordism. 
For example, one could take $O=\partial X_L$ and $W=X_L$ from Theorem \ref{thm: left surgery}, obtained by left surgery applied to $R=\FM_d$. 
Let $A$ be a left $\FM_d$-module.
The relative composite $W \circ_{\FM_d} A$ is 
naturally a left module over the manifold operad $O$. An interesting case occurs when $A$ is an $\FM_d$-algebra, i.e. it is a symmetric sequence concentrated in degree $0$ ($A(n)=\emptyset$ for $n>0)$. From a homotopical perspective, these are the same as non-unital $E_d$-algebras. The relative composite $W \circ_{\FM_d} A$ is naturally an
$O$-algebra. Some explicit examples of $A$ as configuration spaces in $\mathbb{R}^d$ with summable labels are described in \cite{salvatore_proc}. 
\end{ex}

\begin{ex}
Consider the case that $W=\FM_{S^1}/S^1$ is the right nullbordism of $\FM_1$ described in Proposition \ref{prop: rightnullbordant} and $O=1$ is the trivial operad; then
$W \circ_{\FM_1} A$ is the homotopy quotient of the topological Hochschild homology of the nonunital $E_1$-algebra $A$ by its usual circle action.
\end{ex}

\begin{ex} \label{ex: right}
Let $Z$ be a collarable $(\FM_d,P)$-bimodule cobordism
where $P$ is a collarable manifold operad.
For example, one could take $Z=X_R, P=\partial X_R$ obtained from
Theorem \ref{thm: right surgery} by
performing right surgery on the Fulton-MacPherson operad 
$B=\FM_d$.
If $C$ is an $\FM_d$-right module, then 
$C \circ_{\FM_d} Z$ is a right module over $P$.

The most interesting examples of such $C$ are the partially compactified configuration spaces 
$C=\FM_M$ with $M$ a framed $d$-manifold
\cite{markl}.
In this case, $\FM_M \circ_{\FM_d} Z$ is a symmetric sequence of manifolds of dimension $md$ by a similar argument to the proof of Theorem \ref{thm: composition is manifold}.

There is a sequence of embeddings given by inclusion into the first factor
\[\mathrm{Conf}(M,\bullet) \hookrightarrow \mathrm{FM}_M \circ_{\FM_d} Z\] 
which allows one to regard this right module as a sequence of exotic partial compactifications
of the configuration spaces of $M$. These partial compactifications are controlled by infinitesimal information, defined by the preimage of the fat diagonal in $M^{\times \bullet}$ under the composite
\[\mathrm{FM}_M \circ_{\FM_d} Z \rightarrow \mathrm{FM}_M \circ_{\mathrm{FM}_d} \mathrm{com}\cong M^{\times \bullet}.\]
Unlike $\mathrm{FM}_M$, the infinitesimal information controlling these partial compactifications is stored in the bimodule cobordism $Z$ rather than the Fulton-MacPherson operad $\FM_d$. 
\end{ex}

\begin{ex}
Consider the case that
$Z$ is the left nullbordism $L$ of $\FM_d$ described in Proposition \ref{leftnullbordant} and $P=1$ is the trivial operad; then $\FM_M \circ_{\FM_d} Z $
is isomorphic to
$M^{\times \bullet}$, and this isomorphism can even be made one of right $\mathrm{com}$-modules (since we noted that $L$ has a lift to a $(\mathrm{FM}_d,\mathrm{com})$-bimodule). Unlike the standard compactification $\mathrm{Conf}(M,\bullet) \hookrightarrow M^{\times \bullet}$, the infinitesimal data of this compactification is not the fat diagonal, but rather a tubular neighborhood of the fat diagonal.
\end{ex}

\

\section{Appendix: Point set topology of colimits} \label{section: appendix}
\begin{lem}\label{colimtopology_lmm}
Let $D$ be a finite poset, i.e. a category whose object set is finite and for every pair of objects $(x,y)$ of $D$ the number of elements in $\mathrm{Mor}_D(x,y)$ is 0 or 1. Let $I:D\to \mathrm{Top}$ be a functor such that 
\begin{itemize}
    \item for all objects $x\in D$, $I(x)$ is normal (i.e., satisfies separation axioms $T_1+T_4$);
    \item for all morphisms $f$ in $D$, $I(f)$ is closed and injective; 
    \item if $f_1:x_1\to y,f_2:x_2\to y$ are morphisms of $D$ and $I(f_1)(I(x_1))\cap I(f_2)(I(x_2))\neq\emptyset$ in $I(y)$, then there exist objects $w^1,\ldots,w^n$ of $D$ and morphisms $g_1^1,\ldots,g_1^n,g_2^1,\ldots,g_2^n$ of $D$, where $g_i^j\in \mathrm{Mor}_D(w^j,x_i)$, such that 
    $f_1\circ g^j_1=f_2\circ g_2^j$ for all $j$ and 
    $$I(f_1)\big(I(x_1)\big)\cap I(f_2)\big(I(x_2)\big)=\bigcup_{j=1}^nI(f_1\circ g^j_1)\big(I(w^j)\big).$$ 
\end{itemize}
Then, 
\begin{itemize}
    \item $\colim_DI$ is normal; 
    \item for each object $x\in D$, the induced map $I(x)\to\colim_DI$ is closed and injective;  

    \item
    
    Suppose $\hat{f}_1:I(x_1)\to \colim_DI,\hat{f}_2:I(x_2)\to\colim_DI$ are the induced morphisms and $\hat{f}_1(I(x_1))\cap \hat{f}_2(I(x_2))\neq\emptyset$. Suppose $w^1,\ldots,w^n$ are those objects of $D$ such that $\mathrm{Mor}_D(w^i,x_1)\neq\emptyset$ and  $\mathrm{Mor}_D(w^i,x_2)\neq\emptyset$, and denote by $g^i_1\in \mathrm{Mor}_D(w^i,x_1)$ the morphisms. Then 
    $$\hat{f}_1\big(I(x_1)\big)\cap \hat{f}_2\big(I(x_2)\big)=\bigcup_{j=1}^n\big(\hat{f}_1\circ I(g^j_1)\big)\big(I(w^j)\big).$$ 
\end{itemize}

\end{lem}
\begin{proof}
We first note what the third condition yields: 
recall that $\colim_DI=\bigsqcup_{x\in D}I(x)/\sim$, where $\sim$ is the equivalence relation generated by $a=I(f)(b)\implies a\sim b$ for all morphisms $f$ in $D$. 
Suppose $x,y\in D$ and $a\in I(x), b\in I(y)$, then $a\sim b$ if and only if there is a sequence of morphisms $f_0,\ldots, f_n$ in $D$ and elements $c_1,\ldots,c_n$, each $c_i$ in some $I(x_i)$ for some object $x_i$ of $D$, such that 
\[
\begin{tikzcd}
a \rar[shift left=0.2cm]{I(f_0)} \rar[phantom]{\tiny \textnormal{or}} &
c_1 \lar[shift left=0.2cm]{I(f_0)} \rar[shift left=0.2cm]{I(f_1)} \rar[phantom]{\tiny \textnormal{or}}&
c_2\lar[shift left=0.2cm]{I(f_1)}\rar[shift left=0.2cm]{I(f_2)} \rar[phantom]{\tiny \textnormal{or}}&
\cdots\lar[shift left=0.2cm]{I(f_2)}\rar[shift left=0.2cm]{I(f_n)} \rar[phantom]{\tiny \textnormal{or}}&
b\lar[shift left=0.2cm]{I(f_n)}.
\end{tikzcd}
\]
Here between each entries $c_i,c_{i+1}$ (resp. $a,c_1$ and $c_n,b$) we mean to draw only one arrow $I(f_i)$, but it can be either left-pointing or right-pointing. 

(Be aware that the diagrams in this proof mean something different than usual: $a\xrightarrow{f}b$ does not mean a map $f$ from $a$ to $b$, but rather means $f(a)=b$.)
Since if $I(f_i),I(f_{i+1})$ are both left-pointing or right-pointing, we can just compose them, we conclude that $a\sim b$ if and only if there is such a sequence where the arrows $I(f_0),\ldots,I(f_n)$ are alternatingly left-pointing and right-pointing. 
Now, by the third condition of the lemma and the injectivity condition, if part of the sequence is of the form 
$$c_{i-1}\xrightarrow{I(f_{i-1})}c_i\xleftarrow{I(f_i)}c_{i+1},$$
then there exists some object $w$ of $D$, morphisms $g_1,g_2$, and a point $p\in I(w)$ such that 
\[
\begin{tikzcd}
& p \dlar["I(g_1)"'] \drar["I(g_2)"] &\\
c_{i-1} \drar["I(f_{i-1})"'] & & c_{i+1} \dlar["I(f_i)"]\\
& c_i &
\end{tikzcd}
\]
Therefore, by replacing 
$$c_{i-1}\xrightarrow{I(f_{i-1})}c_i\xleftarrow{I(f_i)}c_{i+1}$$
with 
$$c_{i-1}\xleftarrow{I(g_1)}p\xrightarrow{I(g_2)}c_{i+1},$$
and doing this multiple times, we can change the original sequence to the form where all the arrows before a certain $c_i$ are pointing left and all the arrows after $c_i$ are pointing right. By composing some arrows together, we conclude that $a\sim b$ if and only if there exist some object $z\in D$, a point $p\in I(z)$, and morphisms $f_0,f_1$ in $D$ such that 
$$a\xleftarrow{I(f_0)}p\xrightarrow{I(f_1)}b.$$

Now that we have illustrated how the third condition in the lemma is used, we begin the proof of the statements of the lemma. 
The first statement follows from the fact that the image of a normal space under a closed surjection is normal (\cite[Proposition 1.1]{ncatlab}), and that the quotient map $q:\sqcup_{x\in D}I(x)\to \sqcup_{x\in D}I(x)/\sim=\colim_DI$ is closed. The reason $q$ is closed is the following: Let $C\subset\sqcup_{x\in D}I(x)$ be a closed subset. Since $q(C_1\cup C_2)=q(C_1)\cup q(C_2)$, without loss of generality we can assume $C\subset I(x)$ for some object $x\in D$. To say $q(C)$ is closed is to say that for all objects $y\in D$ (including $x$ itself), $\{a\in I(y)|\,\exists b\in C, a\sim b\}$ is closed in $I(y)$. 
But for $a\in I(y)$, the condition $\exists b\in C, a\sim b$ is the same as $\exists b\in C, z\in D, p\in I(z), f\in\mathrm{Mor}_D(z,y),g\in\mathrm{Mor}_D(z,x)$, such that $I(f)(p)=a,I(g)(p)=b$. Therefore, 
$$\{a\in I(y)|\,\exists b\in C, a\sim b\}=\bigcup_{z\in D,f\in\mathrm{Mor}_D(z,y),g\in\mathrm{Mor}_D(z,x)} I(f)\circ I(g)^{-1}(C).$$
Since $D$ is finite, this is a finite union. By the condition that $I(f),I(g)$ are closed, each term is closed. So, $q(C)$ is closed.

The proof of the injectivity statement is similar to the above: suppose $a,b\in I(x)$, $q(a)=q(b)\in\colim_DI$, then there exist $z, p, f,g$ as in the last paragraph, such that $I(f)(p)=a,I(g)(p)=b$. By the assumption that $|\mathrm{Mor}_D(z,x)|\le1$, the fact that $f,g\in \mathrm{Mor}_D(z,x)$ means $f=g$, and so by the injectivity of $I(f)$, $a=b$. 

It remains to prove the last statement of the lemma. That RHS $\subset$ LHS is clear. To show LHS $\subset$ RHS, suppose $c\in$ LHS, then there exist
$a\in I(x_1)$ and $b\in I(x_2)$ such that $\hat{f}_1(a)=c$, $\hat{f}_2(b)=c$. This means $a\sim b$, and 
so there exist $z, p, f,g$ as in the previous paragraph, such that $I(f)(p)=a,I(g)(p)=b$. This shows that $z$ is one of the $w^i$'s in the RHS and $c\in$ RHS. 
\end{proof}

What Lemma \ref{colimtopology_lmm} says is as follows: if a functor $I$ satisfies the three conditions, then adding a new final term ---the colimit--- to the diagram $I$, these three conditions are still satisfied for the new diagram. 

\begin{lem}\label{operadtopology_lmm}
    Suppose $O$ is an $n$-truncated operad (resp. bimodule) represented by $\phi:\mathcal{T}_{\le n}\to \mathrm{Top}$ (resp. $\phi:\mathcal{T}^{RBW}_{\le n}\to \mathrm{Top}$), satisfying the conditions 
    \begin{enumerate}
        \item For every finite set $S$ such that $|S|\le n$, $\phi(\bullet_S)$ is a normal space (resp., for all $C\in\{R,B,W\}$, $\phi(\bullet_{C,S})$ is a normal space);
        \item \label{702embeddingcondition_item}
        for every finite set $S$ with $|S|\le n$, the map(s) 
        $$\colim_{\mathcal{T}(S)\backslash\bullet_S}\phi \longrightarrow \phi(\bullet_S)$$
        \begin{align*}  
        \big(resp. \qquad &\colim_{\mathcal{T}^{RBW}(S)\backslash\bullet_{W,S}}\phi \longrightarrow \phi(\bullet_{W,S}),\\
        &\colim_{\mathcal{T}^{R}(S)\backslash\bullet_{R,S}}\phi \longrightarrow \phi(\bullet_{R,S}),\\
        &\colim_{\mathcal{T}^{B}(S)\backslash\bullet_{B,S}}\phi \longrightarrow \phi(\bullet_{B,S})\quad\big)
        \end{align*}
        are closed and injective. 
    \end{enumerate}
     Then, for every $S$ such that $|S|\le n$ the functor 
    $$\phi: \mathcal{T}(S)\backslash\bullet_S\longrightarrow\mathrm{Top} \qquad \big(resp.\quad \phi:\mathcal{T}^{RBW}(S)\backslash\bullet_{W,S}\longrightarrow\mathrm{Top}\big)$$
    satisfies the conditions imposed on $I$ in Lemma \ref{colimtopology_lmm}.
\end{lem}
\begin{proof}
    By induction on the size of $S$. When $|S|=2$ this is clear, as there is no object in $\mathcal{T}(S)\backslash\bullet_S$ and no morphism in $\mathcal{T}^{RBW}(S)\backslash\bullet_{W,S}$.

    Suppose the conclusion of the lemma holds for $|S|<m$, then by the conclusion of Lemma \ref{colimtopology_lmm} and condition (\ref{702embeddingcondition_item}) above the functor
    $$\phi: \mathcal{T}(S)\longrightarrow\mathrm{Top} \qquad \big(resp.\quad \phi:\mathcal{T}^{RBW}(S)\longrightarrow\mathrm{Top}\big)$$
    also satisfies the conditions in Lemma \ref{colimtopology_lmm}. 

    It remains to show the conclusion of the lemma for $S$ with $|S|=m$.     
    In the operad case, this works because $\mathcal{T}(S)\backslash\bullet_S$ consists of trees with more than one vertex, and so $\phi(T)$ is defined by taking products of some $\phi(\bullet_{S'})$ with $|S|<m$.
   The conditions of Lemma \ref{colimtopology_lmm} can then be checked using the inductive hypothesis.

    In the bimodule case, the argument is similar, except that we also need to apply the conclusion of the lemma in the operad case to the $R$- and $B$-parts first. 
\end{proof}

\bibliographystyle{plain}
\bibliography{main}

\vspace{1cm}

\includegraphics[scale=.33]{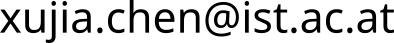}\\
Institute of Science and Technology Austria\\
Am Campus 1, 3400 Klosterneuburg, Austria

~\\
\includegraphics[scale=0.33]{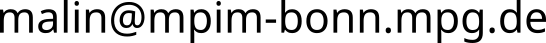}\\
Max Planck Institute for Mathematics\\
Vivatsgasse 7, 53111 Bonn, Germany
~\\

\includegraphics[scale=0.33]{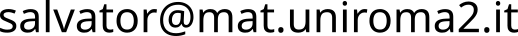}

Dipartimento di Matematica \\
Universit\`a di Roma Tor Vergata\\
Via della Ricerca Scientifica,
00133 Roma, Italy
\end{document}